\documentclass[11pt]{amsart}
\usepackage{amsmath,amssymb,amsthm,amsfonts,setspace,hyperref,dsfont,yhmath,xcolor,marginnote}

\newcommand{\NN}{\mathbb{N}}
\newcommand{\RR}{\mathbb{R}}

\newcommand{\CC}{\mathbb{C}}

\newcommand{\ZZ}{\mathbb{Z}}
\newcommand{\GG}{\mathbb{G}}

\newcommand{\R}{\mathbb{R}}

\newcommand{\N}{\mathbb{N}}
\newcommand{\Z}{\mathbb{Z}}

\renewcommand{\sl}{\operatorname{\mathfrak s\mathfrak l}}

\newcommand{\norm}[1]{\lVert#1\rVert}
\newcommand{\abs}[1]{\lvert#1\rvert}
 \DeclareMathOperator{\SL}{SL}
 
\newtheorem{theorem}{Theorem}[section]
\newtheorem{corollary}[theorem]{Corollary}

\newtheorem{lemma}[theorem]{Lemma}
\newtheorem{proposition}[theorem]{Proposition}

\newcommand{\comment}[1]{}
\theoremstyle{definition}
\newtheorem{definition}[theorem]{Definition}
\newtheorem{remark}[theorem]{Remark}
\newtheorem{fact}[theorem]{Fact}

\newtheorem{case}{Case}

\numberwithin{equation}{section}

\makeatletter
\renewcommand*\env@matrix[1][*\c@MaxMatrixCols c]{%
  \hskip -\arraycolsep
  \let\@ifnextchar\new@ifnextchar
  \array{#1}}
\makeatother

\title[Cohomological equation and cocycle rigidity ]
{Cohomological equation and cocycle rigidity of discrete parabolic actions in some higher rank Lie groups}
\author{James Tanis}
\address[James Tanis]{The MITRE Corporation \\ McLean, VA 22102, USA \footnote{\textbf{Approved for Public Release; Distribution Unlimited. Case Number 18-1114.}  The first author's affiliation with The MITRE Corporation is provided for identification purposes only, and is not intended to convey or imply MITRE's concurrence with, or support for, the positions, opinions or viewpoints expressed by the authors.   
}}
\email{jhtanis@mitre.org}

\author{ Zhenqi Jenny Wang$^1$ }
\thanks{ $^1$ Based on research supported by NSF grant   DMS-1700837}
\address[Zhenqi Jenny Wang]{Department of Mathematics\\
        Michigan State University\\
        East Lansing, MI 48824,   USA}
\email{wangzq@math.msu.edu}

\subjclass[2010]{} 

\begin{document}

\begin{abstract}
Let $\mathbb{G}$ denote a higher rank $\RR$-split simple Lie group of the following type:
$SL(n,\RR)$, $SO_o(m,m)$, $E_{6(6)}$, $E_{7(7)}$ and $E_{8(8)}$, where $m\geq 4$ and  $n \geq 3$. We study the cohomological equation for discrete parabolic actions on $\mathbb G$ via representation theory.
Specifically, we characterize the obstructions to solving the cohomological
equation and construct smooth solutions 
with Sobolev estimates.  We prove that global estimates of the solution are generally not tame, and our  non-tame estimates in the case $\GG=SL(n,\RR)$ are sharp up to finite loss of regularity.  Moreover, we prove that for general $\GG$ the estimates are tame in all but one direction, and as an application, we obtain tame estimates for the common solution of the cocycle equations.
We also give a sufficient condition for which the first cohomology with coefficients in smooth vector fields is trivial. In the case that $\GG=SL(n,\RR)$, we show this condition is also necessary.
A new method is developed to prove tame directions involving computations within maximal unipotent subgroups of the unitary duals of $SL(2,\RR)\ltimes\RR^2$ and $SL(2,\RR)\ltimes\RR^4$.  A new technique is also developed to prove non-tameness for solutions of the cohomological equation.
\end{abstract}

\maketitle

\section{Introduction}
\subsection{Various algebraic actions}
Main examples concern higher rank abelian {partially hyperbolic homogeneous actions
on symmetric spaces and twisted symmetric spaces
and higher rank abelian parabolic homogeneous actions on semisimple Lie groups.
Specifically, for $k, \ell \in \N$ such that $k+\ell\geq 1$,
consider the $\ZZ^k\times \RR^\ell$ algebraic actions defined as follows.
Let $G$ be a connected Lie group,
$A\subseteq G$ a closed abelian subgroup which
is isomorphic to $\ZZ^k\times \RR^\ell$, $M$ a subgroup of
the centralizer $Z(A)$ of $A$, and $\Gamma$ a torsion free lattice in
$G$. Then $A$ acts by left translation on the space
$\mathcal{M}=M\backslash G/\Gamma$, and
we have the following examples:
\begin{itemize}
  \item  Symmetric space examples, where $G$ is a semisimple Lie group of non-compact
type, and $A$ is a subgroup of a maximal $\RR$-split Cartan subgroup in $G$.

\smallskip
  \item  Twisted symmetric space examples, where $G=H\ltimes _\rho\RR^m$ or $G=H\ltimes_\rho N$ is a semidirect
product of a reductive Lie group $H$ with a semisimple factor of non-compact
type with $\RR^m$ or a simply connected nilpotent group $N$, and $A$ is a subgroup of a maximal $\RR$-split Cartan subgroup in $H$.

\smallskip
  \item Parabolic action examples, where $G$ is a semisimple Lie group of non-compact
type and $A$ is a subgroup of a maximal abelian unipotent subgroup in $G$.
\end{itemize}
The cohomological equation arises in several problems in dynamics, for example,
in the study of the existence of invariant measures, in conjugacy problems, in the study
of reparametrisations of flows, etc.
Significant progress has been made over the past two decades in the study of smooth cocycle rigidity for higher rank (partially) hyperbolic algebraic actions  on symmetric and twisted symmetric spaces (see \cite{Damjanovic1}, \cite{Kononenko}, \cite{Spatzier1}, \cite{Spatzier2} and \cite{Zhenqi}).
Generally speaking, 
higher rank strongly chaotic algebraic systems display local rigidity,
and various rigidity phenomena are now well understood.
This is in contrast to the rank-one situation, where Livsic showed
there is an infinite-dimensional space of obstructions to solving the cohomological
equation for a hyperbolic action by $\RR$ or $\ZZ$.

Rigidity results for actions without any hyperbolicity, like parabolic actions, are substantially more difficult to obtain, and accordingly, much less is known about them. All tools and theories developed so far for (partially) hyperbolic systems rely on the following fact: Most orbits grow exponentially under the action of (partially) hyperbolic elements.
In contrast, parabolic actions have at best polynomial growth along orbits,
which prevents similar geometric arguments 
from being effective here.

Results concerning the cohomology of parabolic actions have instead relied on
representation theory as an essential tool, beginning
with the representation theory of $SL(2, \RR)$ or $SL(2, \CC)$.  
Specifically, the first such result, due to L. Flaminio and G. Forni in \cite{Forni}, proved Sobolev estimates of the cohomological equation of the classical horocycle flow 
in irreducible, unitary representations of $PSL(2, \R)$.
Global estimates were then obtained by glueing estimates from each irreducible component.
This approach was later used
in \cite{T1} and \cite{FFT} to study the cohomological equation of
the classical (discrete) horocycle map, 
and it was also employed in \cite{Mieczkowski}, \cite{Mieczkowski1}
and \cite{Ramirez} to study the cohomogical equation or cohomology for
some models of algebraic parabolic actions.

The above results used the representation theory of the entire
group.  However, this cannot be done
in the higher rank setting, even for
$SL(3,\RR)$, whose unitary dual is well-understood \cite{Vogen}.
In general, the unitary dual of many higher rank almost-simple
algebraic groups is not completely classified,
and even when the classification is known,
it is too complicated to apply. To handle these cases,
a new method was introduced in \cite{W1} that is
based on an analysis of the unitary dual of various
subgroups in $\GG$ rather than that of $\GG$ itself.
Mackey theory was used to find models for the representations
of these groups
that appeared in a restricted non-trivial representation of $\GG$.
After explicit computations in irreducible models,
global properties of the solution came with having
sufficiently many semidirect product groups
that contain the one-parameter root subgroup in
the cohomological equation.

Parabolic and (partially) hyperbolic actions are different in other ways as well.
Unlike (partially) hyperboilc actions,
an analysis of the cohomological equation
and cohomology of parabolic maps involves
significantly more technical difficulties than for parabolic flows.
 For example, the space of obstructions to a smooth solution
 of the cohomological equation of horocycle maps
 has infinite countable dimension in each irreducible
 component of $PSL(2,\RR)$, see \cite{T1},
 as opposed to being at most two dimensional in each
 component for the horocycle flow, see \cite{Forni}.
 Moreover, Sobolev estimates for solutions
 of the cohomological equation of horocycle maps are
 not tame in \cite{T1}, \cite{T2} and \cite{FFT},
 as opposed to the tame estimates obtained for the horocycle
 flow in \cite{Forni} and parabolic flow in \cite{W1}.
Because tame estimates for solutions of the cohomological equation
lays the groundwork for proving smooth action rigidity,
see \cite{Damjanovic2} and \cite{Damjanovic3},
not having them complicates this effort.

The purpose of this paper is to extend a careful analysis of the study of the
cohomological equation and cohomology of the horocycle
map to parabolic maps
on some higher rank simple Lie groups.
We characterize the obstructions to solving the cohomological
equation, construct
smooth solutions of the cohomological equation and
obtain non-tame Sobolev estimates for the solution, see
Theorem~\ref{th:6} and \ref{po:2}.
Theorem~\ref{main_thm2-lower} proves that in the case of
$SL(n,\RR)$, for $n\geq 3$, these Sobolev estimates are, in fact, generally
not tame, and our Sobolev estimates in Theorem~\ref{th:6}
are sharp up to finite loss of regularity.
Theorem~\ref{main_thm2-lower} gives the same lower bound
for $SL(2, \RR)$, which is part
of the proof of analogous sharp (up to finite loss of regularity),
non-tame estimates for horocycle maps,
to appear in the forthcoming paper \cite{T2}.

Even though tameness fails for the cohomological equation,
we prove that it holds for general $\mathbb G$ in all but one explicit direction,
see Theorem~\ref{th:101}.
This turns out to be enough in Theorem~\ref{th:8} to prove that tameness
holds for cocycle equations,
which is an important step toward obtaining the
first parabolic actions that display smooth local rigidity
on semisimple homogeneous spaces.
Finally, we give
a sufficient condition for when the
cocycle equation has a common
solution, and in the case of $\mathbb G = SL(n,\RR)$, $n\geq3$,
we show that this conditions is also necessary, see Theorem \ref{th:7}.

To our knowledge, Theorem~\ref{main_thm2-lower}
is the first proof of non-existence of tame Sobolev estimates for a solution of the
cohomological equation of a homogeneous parabolic
action in a non-commutative setting.  
An interesting question is to determine the other settings
where tameness fails.
The fact that it fails in every real special linear
group suggests that it may also fail for every real semisimple Lie group.
In addition, all previous methods to prove tame cocycle rigidity require that
solutions to the cohomological equation are tame.
As a consequence, we believe our method in proving Theorem~\ref{th:101} that proves tame directions
for solutions of cohomological equations
will find other applications, particularly for cycle rigidity, see Theorem~\ref{th:8}.  

Finally,
we comment on the proofs of Theorems~2.1 and 2.2.
Regarding Theorem~\ref{th:6}, the analogous theorem for parabolic flows
was proven in \cite{W1},
where Sobolev estimates of the solution were obtained from
estimates of lower rank subgroups of $\mathbb G$: $SL(2, \RR) \times \R$ and
$SL(2, \RR) \ltimes \RR^2$.
Trying the same approach for the map causes
a non-tame Sobolev loss of regularity in every direction
that does not commute with the unipotent flow direction.
To prove Theorem~\ref{th:101}, our new method proves tame vectors in the maximal unipotent subgroups of $SL(2,\RR)\ltimes\RR^2$ and $SL(2,\RR)\ltimes\RR^4$
rather than in the semidirect products themselves.
This gives us enough directions to prove tame cocycle rigidity,
thereby overcoming the above mentioned analysis difficulty resulting from non-tameness of solutions of the unipotent map.

A new method is also developed to prove Theorem~\ref{main_thm2-lower},
which is carried out in irreducible models of $SL(n, \R)$.
The main idea is to first prove the lower bound of the transfer
function for the twisted equation $(\mathfrak{v}+ \sqrt{-1} \lambda) f = g$,
where the vector field $\mathfrak{v}$ is an element of a unipotent root space and $\lambda \in \R^*$,
which is less complicated to study than the corresponding map, $\exp(\mathfrak{v})$.
We work in Fourier transform and consider the variables for which
the unipotent vector fields are second order differential operators in our model.
For a given $\mathfrak{v}$, we find a function $g$
whose Sobolev norm is bounded linearly in the representation parameter
and whose transfer function $f$ has a norm that is
much larger than that with respect to the same representation parameter.
Comparing the sizes of the representation parameters,
we prove tame estimates generally do not exist for the twisted equation.
We derive Theorem~\ref{main_thm2-lower} concerning unipotent maps from that.

\section{Background, definition, and statement of results}
\subsection{Preliminaries on cocycles} Let $\alpha:A\times
\mathcal{M}\rightarrow \mathcal{M}$ be an action of a topological group $A$ on a (compact)
manifold $\mathcal{M}$ by diffeomorphisms. For a topological group
$Y$, a $Y$-valued {\em cocycle} (or {\em an one-cocycle}) over
$\alpha$ is a continuous function $\beta : A\times \mathcal{M}\rightarrow Y$
satisfying:
\begin{align}
\beta(ab, x) = \beta(a, \alpha(b, x))\beta(b, x)
\end{align}
 for any $a, b \in A$. A
homomorphism $s: A\rightarrow Y$ satisfies the cocycle identity by setting $s(a, x)=s(a)$,
and is called a \emph{constant cocycle}, because it is independent of $x$. A cocycle is
{\em cohomologous to a constant cocycle} if there exists a homomorphism $s : A\rightarrow Y$ and a
continuous transfer map $H : \mathcal{M}\rightarrow Y$ such that for all $a
\in A $
\begin{align}\label{for:6}
 \beta(a, x) = H(\alpha(a, x))s(a)H(x)^{-1}
\end{align}
\eqref{for:6} is called the cohomology equation.
In particular, a cocycle is a {\em coboundary} if it is cohomologous
to the trivial cocycle $s(a) = id_Y$, $a \in A$, i.e. if for all
$a \in A$ the following equation holds:
\begin{align}
 \beta(a, x) = H(\alpha(a, x))H(x)^{-1}.
\end{align}
For more detailed information on cocycles adapted to the present setting
see \cite{Damjanovic1} and \cite{Katok}.

In this paper we will only consider smooth $\CC^k$-valued cocycles over
algebraic parabolic actions on smooth manifolds. By taking component functions we may always assume that $\beta$ is $\CC$-valued. Further, by taking real and imaginary parts, our cocycle results also hold for real-valued cocycles.
Adapted to the settings in this paper, $A$ is a subgroup of a maximal abelian unipotent subgroup in $\GG$ and consider the manifold $\GG/\Gamma$, where $\Gamma\subset \GG$ is a torsion free lattice. A cocycle is called \emph{smooth} if the map $\beta:\,A\rightarrow C^\infty(L^2(\GG/\Gamma))$ is smooth. We can also define $\beta$ to be of class $C^r$.  We note that if the cocycle $\beta$ is cohomologous to a constant cocycle,
then the constant cocycle is given by $s(a)=\int_{\GG/\Gamma}\beta(a,x)dx.$

In what follows, $C$ will denote any constant that depends only
  on the given  group $\GG$. $C_{x,y,z,\cdots}$ will denote any constant that in addition to the
above also depends on parameters $x, y, z$, etc.

\subsection{Main results}
In this paper, $\mathbb{G}$ denotes a higher rank $\RR$-split simple Lie group of the following type:
$SL(n,\RR)$, $SO_o(m,m)$, $E_{6(6)}$, $E_{7(7)}$ and $E_{8(8)}$ where $m\geq 4$ and $n\geq 3$ and $\mathfrak{G}$ denotes its Lie algebra. The conditions on the indices are given firstly to ensure that the groups in question are higher rank Lie groups, and then they are further
restricted to avoid the incidental local isomorphisms between various families of
groups in low dimensions.  For example, the groups $SO_0(2,2)$ and $SO_o(3,3)$ are
locally isomorphic to $SL(2,\RR)\times SL(2,\RR)$ and $SL(4,\RR)$ respectively.
Cohomology properties for lower rank cases will appear
in a forthcoming paper, see \cite{T2}.

Fix an inner product on $\mathfrak{G}$. Let $\mathfrak{G}^1$ be the set of unit vectors in $\mathfrak{G}$. Let $\Phi$ denote the set of roots of $\GG$ and $\mathfrak{u}_\phi$ denote the root space of $\phi$ for any $\phi\in\Phi$. Set
\begin{align*}
E_{\phi}&=\{\psi\in \Phi: \psi+\phi\notin\Phi,\, \phi-\psi\in\Phi\}\quad\text{ and}\\
\quad\bar{E}_{\phi}&=\{\psi\in \Phi: \psi+\phi\notin\Phi,\, \phi-\psi\notin\Phi\}.
\end{align*}
In fact, $\bar{E}_{\phi}$ consists of all roots $\psi$ such that $\mathfrak{u}_\psi\times \mathfrak{u}_\phi$ imbeds in a subalgebra of $\mathfrak{G}$ isomorphic to $\mathfrak{sl}(2,\RR)\times \RR$. Suppose $(\pi,\,\mathcal{H})$ is a unitary representation  of $\GG$ without non-trivial $\GG$-fixed vectors.
\subsection{Results for the cohomological equation}The next result  shows that the $\exp(\mathfrak{v})$-invariant distributions are the only obstructions to solving
the cohomological equation
\begin{align}\label{for:201}
 \mathcal{L}_{\mathfrak{v}}f:=\pi(\exp(\mathfrak{v}))f- f = g
\end{align}
where $\mathfrak{v}\in\mathfrak{u}_\phi$ and $g\in \mathcal{H}^\infty$.

\begin{theorem}\label{th:6}
 Suppose $(\pi,\,\mathcal{H})$ is a unitary representation of $\GG$ without non-trivial $\GG$-fixed vectors and $\mathfrak{v}\in\mathfrak{u}_\phi\bigcap\mathfrak{G}^1$.
 Also suppose $g\in \mathcal{H}^\infty$.  Then the following holds:
\begin{enumerate}
  \item \label{for:35}If $\mathcal{D}(g)=0$ for any $\exp(\mathfrak{v})$-invariant distribution $\mathcal{D}$, then the cohomological equation \eqref{for:201} has a solution $f\in \mathcal{H}$.
\medskip
\item \label{for:33} If the cohomological equation \eqref{for:201} has a solution $f\in \mathcal{H}$, then $f\in \mathcal{H}^\infty$ and satisfies
the Sobolev estimate
\begin{align*}
    \norm{f}_s\leq C_s\norm{g}_{2s+8}\qquad \forall s\geq 0,
\end{align*}
\end{enumerate}
\end{theorem}

The estimates of the solution in above theorem are not tame, i.e., there is no finite loss of
regularity (with respect to Sobolev norms) between the coboundary and the
solution.  Similar results were proven for the (discrete) classical horocycle map, see \cite{FFT}, \cite{T1} and \cite{T2}.
The next result shows that when $\GG=SL(n,\RR)$, $n\geq 3$, they are indeed
the best possible up to a finite loss of regularity.

Let $P$ be the maximal parabolic subgroup of $SL(n,\RR)$ which stabilizes the line $e_1=(\RR,0,\cdots,0)^\tau\in\RR^n$, where $\tau$ is the transpose map. Then $P$ has the form $\begin{pmatrix}a & v \\
0 & \mathcal A\\
 \end{pmatrix}$, where $v^\tau\in\RR^{n-1}$, $a\in\RR\backslash\{0\}$ and $\mathcal A\in GL(n-1,\RR)$. For any $t\in\RR$, $\lambda_t^{\pm}$ is the unitary character of $P$ defined by
\begin{align}\label{for:112}
\lambda_t^{\pm}\begin{pmatrix}a & v \\
0 & \mathcal A\\
 \end{pmatrix}=\varepsilon^{\pm}(a)\abs{a}^{t\sqrt{-1}}
\end{align}
with $\varepsilon^{+}(a)=1$ and $\varepsilon^{-}(a)=\text{sgn}(a)$.

\begin{theorem}\label{main_thm2-lower}
Let $n \geq 2$.  For any $s \geq 0$, for any $\sigma \in [0, s+1/2)$ and for any $C> 0$, there exists $\delta>0$, such that the following holds.
For any $t\in\RR$ with $\abs{t}\geq\delta$,
there are smooth vectors $f,\,g\in \text{Ind}_P^{SL(n,\RR)}(\lambda_t^{\pm})$,
such that $f$ and $g$ satisfy equation \eqref{for:201} with estimates
\[
\Vert (I - \mathfrak{a}^2)^{s/2} f \Vert > C \Vert g \Vert_{s +\sigma}\,,
\]
where $\mathfrak{a}\in\mathfrak{u}_{-\mathfrak{v}}\bigcap\mathfrak{G}^1$.
\end{theorem}
The above result shows that the solution is generally not tame
in the $\mathfrak{u}_{-\mathfrak{v}}$ direction. The next theorem states that the solution is tame in every other direction.
Let $\mathcal{C}$ denote the Cartan subalgebra of $\mathfrak{G}$.
\begin{theorem}\label{th:101}
For any $s \geq 0$, for any unitary representation $(\pi,\,\mathcal{H})$ of $\GG$ without non-trivial $\GG$-fixed vectors, for any $g\in \mathcal{H}^\infty$ and for any $\mathfrak{v}\in\mathfrak{u}_\phi\bigcap\mathfrak{G}^1$, the following holds.
If the cohomological equation \eqref{for:201} has a solution $f\in \mathcal{H}$, then $f\in \mathcal{H}^\infty$ and satisfies
\begin{align}\label{for:46}
 \norm{(1 - \mathfrak{a}^2)^{s/2}f}\leq C_s\norm{g}_{s+12}
\end{align}
for any $\mathfrak{a}\in\mathfrak{u}_\mu\bigcap\mathfrak{G}^1$, where $\mu\neq -\phi$; and
\begin{align}\label{for:39}
 \norm{(1 - \mathfrak{b}^2)^{s/2}f}\leq C_s\norm{g}_{s+12}
\end{align}
for any $\mathfrak{b}\in \mathcal{C}\bigcap\mathfrak{G}^1$.
\end{theorem}
\begin{remark}
Notice that Theorem~\ref{th:6} holds for the regular representation $\mathcal H = L_0^2(\mathbb G/\Gamma)$, consisting of square integrable functions on $\mathbb G/ \Gamma$ orthogonal to constants, where $\Gamma$ is an arbitrary lattice.  Our proof for this general lattice uses a analogous estimate for unitary representations of $SL(2, \R)$,
which will appear in \cite{T2}.
However, when $\Gamma$ is cocompact, the estimate for $SL(2, \R)$ representations is not needed, because
\begin{align*}
    \norm{f}_s\leq C_s\norm{g}_{2s+14}\qquad \forall s\geq 0,
\end{align*}
follows from the subelliptic regularity theorem (see Theorem \ref{th:5}) and the two estimates in the above theorem.

\end{remark}
\subsection{Results for the cocycle equations}
The next theorem states sufficient conditions
for which the infinitesimal version of cohomological
equations have a common solution.
Despite Theorem~\ref{main_thm2-lower},
estimates for cocycle equations are always tame.
\begin{theorem}\label{th:8}Suppose $(\pi,\,\mathcal{H})$ is a unitary representation of $\GG$ without $\GG$-fixed vectors and $f,\,g\in \mathcal{H}^\infty$
 and satisfy the cocycle equation $\mathcal{L}_\mathfrak{u}f=\mathcal{L}_\mathfrak{v}g$, where $\mathfrak{u}\in \mathfrak{u}_\phi\bigcap\mathfrak{G}^1$ and $\mathfrak{v}\in \mathfrak{u}_\psi\bigcap\mathfrak{G}^1$, $\phi,\,\psi\in\Phi$ satisfying $[\mathfrak{u}, \mathfrak{v}]=0$. If $\psi\in\bar{E}_{\phi}$, then the cocycle equation has a common solution $h\in \mathcal{H}^\infty$, that is, $\mathcal{L}_\mathfrak{v}h=f$ and $\mathcal{L}_\mathfrak{u}h=g$; and $h$ satisfies
the Sobolev estimate
\begin{align*}
    \norm{h}_s\leq C_s\max\{\norm{g}_{s+8}, \,\norm{f}_{s+8}\},\qquad \forall\,s>0.
\end{align*}
\end{theorem}

It turns out that when $\GG=SL(n,\RR)$, $n\geq 3$, the condition in the above theorem  is also necessary for the infinitesimal version of cocycle rigidity.  More precisely, there exist uncountably many irreducible unitary representations of $SL(n,\RR)$ such that cocycle rigidity fails whenever there is no rank-two subgroup in the acting group that imbeds in $SL(2,\RR)\times\RR$.

\begin{theorem}\label{th:7}
Let $n \geq 3$, and fix any $\phi\in \Phi$, and set $\phi_0=\phi$.
For any $0\leq m\leq n-2$ and for any $0\leq i\leq m$,
let $\phi_i\in E_{\phi}$ and $\mathfrak{u}_i\in \mathfrak{u}_{\phi_i}$.
For any $t\in\RR$,   
there are smooth vectors $(f_k)_{k = 0}^m \subset \text{Ind}_P^{SL(n,\RR)}(\lambda_t^{\delta})$, $\delta=\pm$,
such that for any $0\leq k,\,\ell\leq m$,
$(f_k)_{k = 0}^m$ satisfies the cocycle equations
$\mathfrak{u}_{k}f_\ell=\mathfrak{u}_{\ell}f_k$,
while none of the equations
$\mathfrak{u}_{k}\omega=f_k$ have a solution in
the attached Hilbert space of
$\text{Ind}_P^{SL(n,\RR)}(\lambda_t^{\delta})$.

\end{theorem}
As an application of Theorem \ref{th:8} we have:
\begin{theorem}\label{th:10}
Suppose $\phi \in \bar{E}_{\psi}$. Also suppose $\mathfrak{v}_1\in\mathfrak{u}_\phi\bigcap\mathfrak{G}^1$ and $\mathfrak{v}_2\in\mathfrak{u}_\psi\bigcap\mathfrak{G}^1$.
Let $U$ denote the discrete subgroup generated by $\exp(\mathfrak{v}_1)$ and $\exp(\mathfrak{v}_2)$. Let $V\subset \GG$ be an abelian unipotent
subgroup containing $U$.
Then any smooth $\CC^k$-valued cocycle over the $V$-action
on $\GG/\Gamma$, where $\Gamma$ is a lattice in $\GG$, is smoothly cohomologous to a constant cocycle.
\end{theorem}

Theorems~\ref{main_thm2-lower} and \ref{th:7} are statements
about irreducible, unitary representations of $SL(n, \R)$, denoted $\text{Ind}_P^{SL(n,\RR)}(\lambda_t^{\delta})$, $t \in \R$.
With regard to the regular representation, the natural question is
whether these irreducible representaitons are subrepresentations of $L^2(SL(n,\RR)/\Gamma)$ for
 $t$ tending to infinity.
In the case $n = 2$, they 
are principal series representations, which arise from both the
continuous and the discrete spectrum of the Laplacian on
$L^2(SL(2, \R)/\Gamma)$, the latter corresponding to Maass forms.
If $\Gamma$ is cocompact, there are infinitely many Maass forms,
so there are infinitely many such representations and the parameter $t$ goes to infinity.
If $\Gamma$ is non-compact and arithmetic, then there are again infinitely many and $t$
tends to infinity, by the Selberg trace formula, see \cite{Lang}.

For the case $n \geq 3$, 
the theory of theta series indicates that there is  
a set of arithmetic lattices such that for any $\Gamma$ in this set 
there is a sequence $t_n\to\infty$
such that $\text{Ind}_P^{SL(n,\RR)}(\lambda_{t_n}^{\delta})$ occurs as a
subrepresentation of $L^2(SL(n,\RR)/\Gamma)$.
Moreover, every arithmetic lattice in $SL(n,\RR)$ is commensurable
with one of the lattices stated above. 
Then because all lattices in $SL(n,\RR)$ are arithmetic \cite{Margulis},
the following much stronger statement is expected to hold: 
For any lattice $\Gamma$ of $SL(n,\RR)$, there is a finite index subgroup
$\Gamma_1\subset \Gamma$ and a sequence $t_n\to\infty$ such that
$\text{Ind}_P^{SL(n,\RR)}(\lambda_{t_n}^{\delta})$ occurs as a
subrepresentation of $L^2(SL(n,\RR)/\Gamma_1)$.  Since it has
been an open problem of giving a complete set of irreducible
representations that appear in $L^2(SL(n,\RR)/\Gamma)$ for
any lattice $\Gamma$, the above results are by far the best that is known
about cocycle rigidity in $SL(n,\RR)$.

\section{Representation Theory}\label{sec:101}

\subsection{Unitary representations of $SL(2, \R)$}\label{sect:SL2R_reps}
The Lie algebra of $SL(2, \RR)$ is generated by the vector
fields
 \begin{align}\label{eq:sl-com}
 X= \begin{pmatrix} {1}&0\\0& {-1}
 \end{pmatrix}, \quad U=\begin{pmatrix} 0&1\\0& 0
 \end{pmatrix}, \quad V=\begin{pmatrix} 0&0\\1& 0
 \end{pmatrix}.
 \end{align}
 The \emph{Casimir} operator is then given by
\begin{align*}
\Box:= -X^2-2(UV+VU),
\end{align*}
which generates the center of the enveloping algebra of $\mathfrak{sl}(2,\RR)$. The Casimir operator $\Box$
acts as a constant $u\in\RR$ on each irreducible unitary representation space  and its value classifies them into four classes.
Unitary representations are classified by a representation parameter $\nu$. The \emph{Casimir parameter} $u$ and the representation parameter $\nu$
are linked by the formula $\nu=\sqrt{1-u}$. Then all the irreducible unitary representations of $SL(2,\RR)$
must be equivalent to one the following:
\begin{itemize}
  \item principal series representations $\pi_\nu^{\pm}$, $u\geq 1$ so that
$\nu=i\RR$,
\medskip
  \item complementary series representations $\pi^0_\nu$, $0 <u< 1$, so that $0 < \nu< 1$,
  \medskip
  \item mock discrete or discrete series representations $\pi^0_\nu$ and $\pi^0_{-\nu}$, $u=-n^2+2n$, $n\geq 1$, so $\nu=n-1$,
  \medskip
  \item the trivial representation, $u=0$.
\end{itemize}
 Any unitary representation $(\pi,\mathcal{H})$ of $SL(2,\RR)$ is decomposed into a direct integral (see \cite{Forni} and \cite{Mautner})
\begin{align}\label{for:1}
\mathcal{H}=\int_{\oplus}\mathcal{H}_ud\mu(u)
\end{align}
with respect to a positive Stieltjes measure $d\mu(u)$ over the spectrum $\sigma(\Box)$. The
Casimir operator acts as the constant $u\in \sigma(\Box)$ on every Hilbert space $\mathcal{H}_u$. The
representations induced on $\mathcal{H}_u$ do not need to be irreducible. In fact, $\mathcal{H}_u$ is in general
the direct sum of an (at most countable) number of unitary representations equal
to the spectral multiplicity of $u\in \sigma(\Box)$. We say that \emph{$\pi$ has a spectral gap (of $u_0$)} if $u_0>0$ and $\mu((0,u_0])=0$ and $\pi$ contains no non-trivial $SL(2,\RR)$-fixed vectors.

\subsection{Introduction to Mackey representation theory}\label{sec:24} The problem of determining the complete set of equivalence classes of unitary irreducible
representations of a general class of semi-direct product groups has been solved by Mackey \cite{Mac}.
These results are summarized in this section with explicit application to groups
$SL(2,\RR)\ltimes\RR^2$ and $SL(2,\RR)\ltimes\RR^4$ to facilitate the study of cohomological equation and cocycle rigidity
that follows. There are two essential ingredients used by Mackey to determine the unitary irreducible
representations of a semi-direct product group $S$. The first is the general notion of inducing a unitary
representation of a group $S$ from a unitary representation of a subgroup $H$. The second is the dual action of $S$ on the characters of
the normal subgroup. If $S$ is second countable and every orbit is
locally closed (intersection of an open and a closed set), then Mackey theory gives the construction of
the complete set of equivalence classes of irreducible unitary representations on $S$ with an appropriate Borel
topology.

Suppose $S$ is a locally compact second countable group and $H$ is a closed subgroup. Let $\pi$ be a unitary
representation of $H$ on a Hilbert space $\mathcal{H}$. Suppose $S/H$ carries a $S$-invariant $\sigma$ finite measure $\mu$. Choose a Borel map $\Lambda:S/H\rightarrow S$ such that $p\circ \Lambda=Id$, where $p:S\rightarrow S/H$ is the natural projection. The representation $\pi$ on $H$ induces a representation $\pi_1$ on $S$ as:
\begin{align}\label{for:80}
   ( \pi_1(s)f)(\gamma)=\pi\bigl(\Lambda(\gamma)^{-1}s\Lambda(s^{-1}\gamma)\bigl)f(s^{-1}\gamma)
\end{align}
where $s\in S$, $\gamma\in S/H$ and $f\in L^2(S/H,\mathcal{H},\mu)$. More precisely, if $s^{-1}\Lambda(\gamma)$ decomposes as
\begin{align*}
  s^{-1}\Lambda(\gamma)=\bigl(s^{-1}\Lambda(\gamma)\bigl)_\Lambda \bigl(s^{-1}\Lambda(\gamma)\bigl)_H
\end{align*}
where $\bigl(s^{-1}\Lambda(\gamma)\bigl)_\Lambda\in \Lambda(S/H)$ and $\bigl(s^{-1}\Lambda(\gamma)\bigl)_H\in H$, then \eqref{for:80} has the expression
\begin{align*}
   ( \pi_1(s)f)(\gamma)=\pi(\bigl(s^{-1}\Lambda(\gamma)\bigl)_H^{-1})f(\bigl(s^{-1}\Lambda(\gamma)\bigl)_\Lambda).
\end{align*}
The representation $\pi_1$ is unitary and is called \emph{the representation of the group $S$ induced from $\pi$ in the sense of Mackey} and is denoted by Ind$_{H}^S(\pi)$. For the cases of interest to us, the groups are very well behaved and satisfy the requisite
properties.

\begin{theorem}\label{th:1}(Mackey theorem, see \cite[Ex 7.3.4]{Zimmer}, \cite[III.4.7]{Margulis}) Let $S$ be a locally compact second countable group and $\mathcal{N}$ be an abelian closed
normal subgroup of $S$. We define the natural action of $S$ on the group of characters $\widehat{\mathcal{N}}$ of the group $\mathcal{N}$ by setting
\begin{align*}
    (s\chi)(\mathfrak{n}):=\chi(s^{-1}\mathfrak{n}s),\qquad s\in S,\,\chi\in \widehat{\mathcal{N}}, \,\mathfrak{n}\in \mathcal{N}.
\end{align*}
Assume that every orbit $S\cdot \chi$, $\chi\in \widehat{\mathcal{N}}$ is locally closed in $\widehat{\mathcal{N}}$. Then for
any irreducible unitary representation $\pi$ of $S$, there is a point $\chi_0\in \widehat{\mathcal{N}}$ with $S_{\chi_0}$
its stabilizer in $S$, a measure $\mu$ on $\widehat{\mathcal{N}}$ and an irreducible unitary representation  $\sigma$ of $S_{\chi_0}$ such that
\begin{enumerate}
  \item $\pi=\text{Ind}_{S_{\chi_0}}^S(\sigma)$,
  \item $\sigma\mid_{\mathcal{N}}=(\dim)\chi_0$,
  \item $\pi(x)=\int_{\widehat{\mathcal{N}}}\chi(x)d\mu(\chi)$, for any $x\in \mathcal{N}$; and $\mu$ is ergodically supported on the orbit $S\cdot \chi_0$.
\end{enumerate}

\end{theorem}
\subsection{Sobolev space and elliptic regularity theorem}\label{sec:17} Let $\pi$ be a unitary representation of a Lie group $G$ with Lie algebra $\mathfrak{g}$ on a
Hilbert space $\mathcal{H}=\mathcal{H}(\pi)$.
\begin{definition}\label{de;1}
For $k\in\NN$, $\mathcal{H}^k(\pi)$ consists of all $v\in\mathcal{H}(\pi)$ such that the
$\mathcal{H}$-valued function $g\rightarrow \pi(g)v$ is of class $C^k$ ($\mathcal{H}^0=\mathcal{H}$). For $X\in\mathfrak{g}$, $d\pi(X)$ denotes the infinitesimal generator of the
one-parameter group of operators $t\rightarrow \pi(\exp tX)$, which acts on $\mathcal{H}$ as an essentially skew-adjoint operator. For any $v\in\mathcal{H}$, we also write $Xv:=d\pi(X)v$.
\end{definition}
We shall call $\mathcal{H}^k=\mathcal{H}^k(\pi)$ the space of $k$-times differentiable vectors for $\pi$ or the \emph{Sobolev space} of order $k$. The
following basic properties of these spaces can be found, e.g., in \cite{Nelson} and \cite{Goodman}:
\begin{enumerate}
  \item $\mathcal{H}^k=\bigcap_{m\leq k}D(d\pi(Y_{j_1})\cdots d\pi(Y_{j_m}))$, where $\{Y_j\}$ is a basis for $\mathfrak{g}$, and $D(T)$
denotes the domain of an operator on $\mathcal{H}$.

\medskip
  \item $\mathcal{H}^k$ is a Hilbert space, relative to the inner product
  \begin{align*}
    \langle v_1,\,v_2\rangle_{G,k}:&=\sum_{1\leq m\leq k}\langle Y_{j_1}\cdots Y_{j_m}v_1,\,Y_{j_1}\cdots Y_{j_m}v_2\rangle+\langle v_1,\,v_2\rangle
  \end{align*}
  \item The spaces $\mathcal{H}^k$ coincide with the completion of the
subspace $\mathcal{H}^\infty\subset\mathcal{H}$ of \emph{infinitely differentiable} vectors with respect to the norm
\begin{align*}
    \norm{v}_{G,k}=\bigl\{\norm{v}^2+\sum_{1\leq m\leq k}\norm{Y_{j_1}\cdots Y_{j_m}v}^2\bigl\}^{\frac{1}{2}}.
  \end{align*}
induced by the inner product in $(2)$. The subspace $\mathcal{H}^\infty$
coincides with the intersection of the spaces $\mathcal{H}^k$ for all $k\geq 0$.

\medskip
\item $\mathcal{H}^{-k}$, defined as the Hilbert space duals of
the spaces $\mathcal{H}^{k}$, are subspaces of the space $\mathcal{E}(\mathcal{H})$ of distributions, defined as the
dual space of $\mathcal{H}^\infty$.
  \end{enumerate}
We write $\norm{v}_{k}:=\norm{v}_{G,k}$ and $ \langle v_1,\,v_2\rangle_{k}:= \langle v_1,\,v_2\rangle_{G,k}$ if there is no confusion. Otherwise,
we use subscripts to emphasize that the regularity is measured with respect to $G$.

If $G=\RR^n$ and $\mathcal{H}=L^2(\RR^n)$, the set of square integrable functions on $\RR^n$, then $\mathcal{H}^k$ is the space consisting of all functions on $\RR^n$ whose first $s$ weak derivatives are functions in $L^2(\RR^n)$. In this case, we use the notation $W^k(\RR^n)$ instead of $\mathcal{H}^k$ to avoid confusion. For any open set $\mathcal{O}\subset\RR^n$, $\norm{\cdot}_{(C^r,\mathcal{O})}$ stands for $C^r$ norm for functions having continuous derivatives up to order $r$ on $\mathcal{O}$. We also write $\norm{\cdot}_{C^r}$ if there is no confusion.

We list the well-known elliptic regularity theorem which will be frequently
used in this paper (see \cite[Chapter I, Corollary 6.5 and 6.6]{Robinson}):
\begin{theorem}\label{th:4}
Fix a basis $\{Y_j\}$ for $\mathfrak{g}$ and set $L_{2m}=\sum Y_j^{2m}$, $m\in\NN$. Then
\begin{align*}
    \norm{v}_{2m}\leq C_m(\norm{L_{2m}v}+\norm{v}),\qquad \forall\, m\in\NN
\end{align*}
where $C_m$ is a constant only dependent on $m$ and $\{Y_j\}$.
\end{theorem}
Suppose $\Gamma$ is an
irreducible torsion-free cocompact lattice in $G$. Denote by $\Upsilon$ the regular representation of $G$ on $\mathcal{H}(\Upsilon)=L^2(G/\Gamma)$. Then we have the following subelliptic regularity theorem (see \cite{Spatzier2}):
\begin{theorem}\label{th:5}
Fix $\{Y_j\}$ in $\mathfrak{g}$ such that commutators of $Y_j$ of length at most $r$ span $\mathfrak{g}$. Also set $L_{2m}=\sum Y_j^{2m}$, $m\in\NN$. Suppose $f\in\mathcal{H}(\Upsilon)$ or  $f\in \mathcal{E}(\mathcal{H})$. If $L_{2m}f\in \mathcal{H}(\Upsilon)$ for any $m\in\NN$, then $f\in \mathcal{H}^\infty(\Upsilon)$ and satisfies
\begin{align}\label{for;1}
\norm{f}_{\frac{2m}{r}-1}\leq
C_m(\norm{L_{2m}f}+\norm{f}),\qquad \forall\, m\in\NN
\end{align}
where $C_m$ is a constant only dependent on $m$ and $\{Y_j\}$.
\end{theorem}
\begin{remark}
The elliptic regularity theorem is a general property, while the subelliptic regularity theorem can't be applied without extra
assumptions. For example, the assumption $G/\Gamma$ is essential in the above theorem. In \cite{Spatzier2} \eqref{for;1} is obtained in a local version. The compactness
guarantees the existence of the uniform constant $C_m$.

\end{remark}
\subsection{Direct decompositions of Sobolev space}\label{sec:3}
For any Lie group $G$ of type $I$ and its unitary representation $\rho$, there is a decomposition of $\rho$ into a direct integral
\begin{align}\label{for:66}
 \rho=\int_Z\rho_zd\mu(z)
\end{align}
of irreducible unitary representations for some measure space $(Z,\mu)$ (we refer to
\cite[Chapter 2.3]{Zimmer} or \cite{Margulis} for more detailed account for the direct integral theory). All the operators in the enveloping algebra are decomposable with respect to the direct integral decomposition \eqref{for:66}. Hence there exists for all $s\in\RR$ an induced direct
decomposition of the Sobolev spaces:
\begin{align}\label{for:67}
\mathcal{H}^s=\int_Z\mathcal{H}_z^sd\mu(z)
\end{align}
with respect to the measure $d\mu(z)$.

The existence of the direct integral decompositions
\eqref{for:66}, \eqref{for:67} allows us to reduce our analysis of the
cohomological equation to irreducible unitary representations. This point of view is
essential for our purposes.

Before proceeding further with the proof of Theorem \ref{th:7}, we list some important properties of representation of semidirect product $SL(2,\RR)\ltimes\RR^2$ without non-trivial $\RR^2$-invariant vectors (see \cite{oh}, \cite{zhenqi1} and \cite{Zimmer}) which will be frequently
used in this paper:
\begin{proposition}\label{cor:1}
For any unitary representation $\pi$ of $SL(2,\RR)\ltimes\RR^2$ without non-trivial $\RR^2$-fixed vectors, where $SL(2,\RR)$ acts on $\RR^2$ as the standard representation, then $\pi\mid_{SL(2,\RR)}$ is  tempered, i.e., $\pi\mid_{SL(2,\RR)}$ is weakly contained in the regular representation of $SL(2,\RR)$.
\end{proposition}
The proposition is a special case of Lemma 7.4 in \cite{zhenqi1}, which follows from Mackey's theory and Borel density theorem (see \cite[Theorem 3.2.5]{Zimmer}).
\begin{remark}\label{re:2}It is known that for $SL(2,\RR)$, the discrete series and principal series representations are  tempered, while the complementary series representations are not (see \cite{tan}). The above proposition implies that $\pi\mid_{SL(2,\RR)}$ only contains the principal series and discrete series of $SL(2,\RR)$. If the attached space of $\pi$ is $\mathcal{H}$ and $\mathcal{H}$ is decomposed into a direct integral as described in \eqref{for:1} of Section \ref{sect:SL2R_reps}
\begin{align*}
\mathcal{H}=\int_{\oplus}\mathcal{H}_ud\mu(u).
\end{align*}
then above discussion shows that $\mu((0,1))=0$.
\end{remark}
We end this section by a standard result about coboundary equation:
\begin{lemma}\label{le:14}
Suppose $(\pi,\mathcal{H})$ is a unitary representation for a Lie group $G$ with Lie algebra $\mathfrak{g}$ and $\mathfrak{u}_1,\,\mathfrak{u}_2\in \mathfrak{g}$ with $[\mathfrak{u}_1,\mathfrak{u}_2]=0$. Suppose there is no non-trivial $\mathfrak{u}_2$-invariant vectors (we call $v\in\mathcal{H}$ a $\mathfrak{u}_2$-invariant vector if $\mathfrak{u}_2v=0$). If
$f,\,g\in \mathcal{H}$ satisfy the coboundary equation $\mathfrak{u}_1f=\mathfrak{u}_2g$ and the equations $\mathfrak{u}_1h=g$ has a solution $h\in \mathcal{H}^2$, then $h$ also solves the equation $\mathfrak{u}_2h=f$.
\end{lemma}
\begin{proof}
From $\mathfrak{u}_1h=g$ we have
\begin{align*}
\mathfrak{u}_1\mathfrak{u}_2h=\mathfrak{u}_2(\mathfrak{u}_1h)=\mathfrak{u}_2 g=\mathfrak{u}_1f,
\end{align*}
which implies that $\mathfrak{u}_2h=f$ since there is no non-trivial $\mathfrak{u}_2$-invariant vectors.
\end{proof}

\section{Explicit calculations based on Mackey theory}\label{sec:20}
Recall the context and notations in Section \ref{sec:24}. Let $S$ be locally compact and $\mathcal{N}\subset S$ a
normal abelian subgroup with $\mathcal{N}\cong \RR^n$. Suppose $\pi$ is an irreducible unitary representation of $S$ such that
\begin{align}\label{for:4}
  \pi\mid_{\RR^n}(x)=\int_{\hat{\RR}^n}\chi(x) d\mu(\chi).
\end{align}
We will derive the representations of $S$ from the action of $S$ on $\widehat{\mathcal{N}}\cong \hat{\RR}^n$.

\subsection{Unitary representations of $SL(2, \R) \ltimes \R^2$ without non-trivial $\RR^2$-fixed vectors}\label{sec:102}
 Write $SL(2, \R) \ltimes \R^2$ in the form
$\begin{pmatrix}[cc|c]
  a & b & v_1\\
  c & d & v_2
\end{pmatrix}$, where $\begin{pmatrix}
  a & b \\
  c & d
\end{pmatrix}\in SL(2,\RR)$ and $\begin{pmatrix}
  v_1 \\
  v_2
\end{pmatrix}\in \RR^2$. The action of $SL(2, \R)$ on $\RR^2$ given by usual matrix multiplication. The group composition law is
\begin{align}\label{for:138}
(g_1,v_1)(g_2,v_2)=(g_1g_2,\,g_2^{-1}v_1+v_2).
\end{align}
The description of representations of $SL(2, \R) \ltimes \R^2$  appears in \cite{W1}. Here we just briefly quote the results. For any $h=\begin{pmatrix}[cc]
  a & b\\
c & d
\end{pmatrix}\in SL(2,\RR)$ and $v=\begin{pmatrix}v_1 \\
v_2\\
\end{pmatrix}\in\RR^2$, the dual action $\hat{h}$ on $\hat{\RR}^2\cong\RR^2$ is:
\begin{align*}
 \hat{h}(v)=\begin{pmatrix}av_1+cv_2 \\
bv_1+dv_2\\
\end{pmatrix}.
\end{align*}
This allows us to completely determine the orbits and the representation theory. The orbits fall into 2 classes: the origin and its
complement. Therefore these $SL(2,\RR)$-orbits on $\widehat{\RR^4}$ are locally
closed. Then we can apply Theorem \ref{th:1}. If $\mu$ in \eqref{for:4} is supported on the origin, the corresponding irreducible representation is trivial on $\RR^2$, and hence the representation factors to a representation of $SL(2,\RR)$. If $\mu$ is supported on the complement, we choose a typical vector $\begin{pmatrix}
  0 \\
  1
\end{pmatrix}$, its stabilizer is isomorphic to the Heisenberg group
\begin{align}\label{for:14}
N=\Bigl\{\begin{pmatrix}[cc|c]
  1 & x & v_1\\
  0 & 1 & v_2
\end{pmatrix}:\,x,\,v_1,\,v_2\in\RR\Bigl\}.
\end{align}
Therefore, any irreducible unitary representation of $SL(2,\RR)\ltimes \RR^2$ without non-trivial $\RR^2$-fixed vectors is induced from a representation of $N$:
\begin{lemma}\label{le:1}
The irreducible representations of $SL(2,\RR)\ltimes \RR^2$ without non-trivial $\RR^2$-fixed vectors are parameterized by $t\in\RR$ and the group action is defined by
\begin{gather*}
\rho_t: SL(2,\RR)\ltimes\RR^2\rightarrow \mathcal{B}(\mathcal{H}_t)\\
\rho_{t}(v)f(x,\xi)=e^{(v_2x-v_1\xi)\sqrt{-1}}f(x,\xi),\\
\rho_{t}(g)f(x,\xi)=e^{\frac{bt\sqrt{-1}}{x(dx-b\xi)}}f(dx-b\xi,-cx+a\xi);
 \end{gather*}
 and
 \begin{align*}
  \norm{f}_{\mathcal{H}_t}=\norm{f}_{L^2(\RR^2)},
 \end{align*}
 where $(g,v)=\Big(\begin{pmatrix}a & b \\
c & d\\
 \end{pmatrix},\begin{pmatrix}v_1 \\
v_2\\
 \end{pmatrix}\Big)\in SL(2,\RR)\ltimes \RR^2$.

 We choose a basis for $\mathfrak{sl}(2,\RR)$ as in \eqref{eq:sl-com}
 and a basis of $\RR^2$ to be $Y_1=\begin{pmatrix}1 \\
0\end{pmatrix}$ and $Y_2=\begin{pmatrix}0 \\
1\end{pmatrix}$. Then we get
\begin{gather*}
 X=-x\partial_x+\xi\partial_\xi,\quad U=tx^{-2}\sqrt{-1}-\xi\partial_x,\quad V=-x\partial_\xi\notag\\
 Y_1=-\xi\sqrt{-1},\quad Y_2=x\sqrt{-1}.
 \end{gather*}
For $\rho_t$, if we take the Fourier transformation on $\xi$ ($\hat{f}_y(x,y)=\frac{1}{\sqrt{2\pi}}\int_{\RR}f(x,\xi)e^{-iy \xi}d\xi$), we get the Fourier model:
 \begin{gather}\label{for:139}
 X=-I-x\partial_x-y\partial_y, \qquad U=tx^{-2}\sqrt{-1}-\partial_x\partial_y\sqrt{-1}\notag\\
 V=-xy\sqrt{-1}, \quad Y_1=\partial_y, \quad Y_2=x\sqrt{-1}.
 \end{gather}
\end{lemma}

\subsection{Unitary representations of $SL(2,\RR)\ltimes \RR^4$
without non-trivial $L_1$ or $L_2$-fixed vectors}\label{sec:10}
 We consider the group $SL(2,\RR)\ltimes \RR^4$
which can be expressed in the form $\begin{pmatrix}[cc|cc]
  a & b & u_1& v_1\\
c & d & u_2& v_2
\end{pmatrix}$, where $\begin{pmatrix}[ccc|c]
  a & b \\
c & d
\end{pmatrix}\in SL(2,\RR)$ and $\begin{pmatrix}[ccc|c]
   u_1& v_1\\
 u_2& v_2
\end{pmatrix}\in\RR^4$. Let $L_1=\begin{pmatrix}v_1\\
v_2
\end{pmatrix}$ and $L_2=\begin{pmatrix}u_1\\
u_2
\end{pmatrix}$, which are isomorphic to $\RR^2$. The actions of $SL(2,\RR)$ on $L_1$ and $L_2$ are the standard
representations of $SL(2,\RR)$ on $\RR^2$, as described by \eqref{for:138}. We choose a basis of $\mathfrak{sl}(2,\RR)$ as in \eqref{eq:sl-com}, and a basis of $\RR^4$ to be
\begin{gather*}
Y_1=\begin{pmatrix}1 &0\\
0& 0\\
\end{pmatrix},\quad Y_2=\begin{pmatrix}0 &0\\
1& 0\\
\end{pmatrix}\quad Y_3=\begin{pmatrix}0 &1\\
0& 0\\
\end{pmatrix}\quad Y_4=\begin{pmatrix}0 &0\\
0& 1\\
\end{pmatrix}.
 \end{gather*}
For any $h=\begin{pmatrix}[cc]
  a & b\\
c & d
\end{pmatrix}\in SL(2,\RR)$, $u=\begin{pmatrix}u_1 \\
u_2\\
\end{pmatrix}\in\RR^2$ and $v=\begin{pmatrix}v_1 \\
v_2\\
\end{pmatrix}\in\RR^2$, the dual action $\hat{h}$ on $\hat{\RR}^4\cong\RR^4$ is:
\begin{align}\label{for:81}
 \hat{h}(u,v)=\Bigg(\begin{pmatrix}au_1+cu_2 \\
bu_1+du_2\\
\end{pmatrix},\,\begin{pmatrix}av_1+cv_2 \\
bv_1+dv_2\\
\end{pmatrix}\Bigg).
\end{align}
This allows us to completely determine the orbits and the representation theory. The orbits fall into five classes:
\begin{enumerate}
  \item []$\mathcal{O}_1=\{(0,0)\}$,
  \item []$\mathcal{O}_2=\{(0,v):\,v\neq0\}$,
  \item []$\mathcal{O}_3=\{(u,0):\,u\neq0\}$,
  \item []$\mathcal{O}_4=\{(u,su): u\neq0\}$ with $s\neq0$,
  \item []$\mathcal{O}_5=\{(u,v):\,\text{det}(u,v)=s\}$ with $s\neq0$.
\end{enumerate}
Therefore these $SL(2,\RR)$-orbits on $\widehat{\RR^4}$ are locally
closed. Then we can apply Theorem \ref{th:1}. For $\mathcal{O}_1$, the corresponding irreducible representation is trivial on $\RR^4$, and hence the representation factors to a representation of $SL(2,\RR)$. For $\mathcal{O}_2$ and $\mathcal{O}_3$, the corresponding irreducible representations factor to representations of
$SL(2,\RR)\times\RR^2$: for $\mathcal{O}_2$, corresponding irreducible representation is trivial on $L_2$, and for $\mathcal{O}_3$, the corresponding irreducible representation is trivial on $L_1$. Then we just need to focus on $\mathcal{O}_4$ and $\mathcal{O}_5$.

For $\mathcal{O}_4$, we choose a typical point $\begin{pmatrix}0 & 0 \\
1 & s\\
\end{pmatrix}$, then its stabilizer for the dual action is:
\begin{align*}
N_1=\Bigl\{\begin{pmatrix}[ccc|c]
  1 & x & u_1&v_1\\
  0 & 1 & u_2&v_2
\end{pmatrix}:\,x,\,v_1,\,v_2,\,u_1,\,u_2\in\RR\Bigl\}.
\end{align*}
Compare $N_1$ with the stabilizer $N$ in \eqref{for:14}. It is easy to see that for any irreducible representation determined determined by $\mathcal{O}_4$,
its restrictions on $SL(2,\RR)\ltimes L_i$, $i=1,\,2$ are also irreducible representations without non-trivial $L_i$-fixed vectors. Then by Lemma \ref{le:1}, we get the first class of irreducible representations of $SL(2,\RR)\ltimes \RR^4$ without non-trivial $L_1$ or $L_2$-fixed vectors:
\begin{lemma}\label{le:3}
These representations  are parameterized by $t\in\RR$ and $s\in\RR^*$;  and the group action is defined by
\begin{gather*}
\rho_{t,s}: SL(2,\RR)\ltimes\RR^4\rightarrow \mathcal{B}(\mathcal{H}_t)\\
\rho_{t,s}(u,v)f(x,\xi)=e^{(u_2x-u_1\xi)\sqrt{-1}}e^{(v_2x-v_1\xi)s\sqrt{-1}}f(x,\xi),\\
\rho_{t,s}(g)f(x,\xi)=e^{\frac{bt\sqrt{-1}}{x(dx-b\xi)}}f(dx-b\xi,-cx+a\xi);
 \end{gather*}
 and
 \begin{align*}
  \norm{f}_{\mathcal{H}_{t,s}}=\norm{f}_{L^2(\RR^2)},
 \end{align*}
 where $(g,u,v)=\Big(\begin{pmatrix}a & b \\
c & d\\
 \end{pmatrix},\begin{pmatrix}u_1 \\
u_2\\
 \end{pmatrix},\begin{pmatrix}v_1 \\
v_2\\
 \end{pmatrix}\Big)\in SL(2,\RR)\ltimes \RR^4$.

 Then we get
 \begin{gather*}
 X=-x\partial_x+\xi\partial_\xi,\quad U=tx^{-2}\sqrt{-1}-\xi\partial_x,\quad V=-x\partial_\xi\notag\\
 Y_1=-\xi\sqrt{-1},\quad Y_2=x\sqrt{-1}\quad Y_3=-s\xi\sqrt{-1},\quad Y_4=sx\sqrt{-1}.
 \end{gather*}
 For $\rho_{t,s}$, if we take the Fourier transformation on $\xi$ ($\hat{f}_y(x,y)=\frac{1}{\sqrt{2\pi}}\int_{\RR}f(x,\xi)e^{-iy \xi}d\xi$), we get the Fourier model:
 \begin{gather}\label{for:8}
 X=-I-x\partial_x-y\partial_y, \qquad U=tx^{-2}\sqrt{-1}-\partial_x\partial_y\sqrt{-1}\notag\\
 V=-xy\sqrt{-1}, \quad Y_1=\partial_y, \quad Y_2=x\sqrt{-1},\notag\\
   Y_3=s\partial_y, \quad Y_4=sx\sqrt{-1}.
 \end{gather}
\end{lemma}
For $\mathcal{O}_5$, the stabilizer of a typical point $\begin{pmatrix}0 & s \\
1 & 0\\
\end{pmatrix}$ for the dual action is:
\begin{align*}
N=\Bigl\{\begin{pmatrix}[ccc|c]
  1 & 0 & u_1&v_1\\
  0 & 1 & u_2&v_2
\end{pmatrix}:\,x,\,v_1,\,v_2,\,u_1,\,u_2\in\RR\Bigl\}.
\end{align*}
 Note that $SL(2,\RR)\ltimes\RR^4/N$ is isomorphic to $SL(2,\RR)$. We choose a Borel section $\Lambda:SL(2,\RR)\ltimes\RR^4/N\rightarrow SL(2,\RR)$ given by $\Lambda(x,\xi,z)=\begin{pmatrix}
  x & xz\\
 \xi&  x^{-1}+\xi z\\
 \end{pmatrix}$.  The
action of the group on the cosets is
\begin{align*}
g^{-1}\Lambda(x,\xi,z)&=\Lambda\Big(dx-b\xi,a\xi-cx,z-\frac{bx^{-1}}{dx-b\xi}\Big)(u',v')
\end{align*}
where $g=\begin{pmatrix}
  a & b &u_1&v_1\\
  c & d&u_2&v_2
\end{pmatrix}$, $u'=-\Lambda(x,\xi,z)^{-1}\begin{pmatrix}
  a & b\\
 c&  d\\
 \end{pmatrix}\begin{pmatrix}
  u_1\\
 u_2\\
 \end{pmatrix}$ and $v'=-\Lambda(x,\xi,z)^{-1}\begin{pmatrix}
  a & b\\
 c&  d\\
 \end{pmatrix}\begin{pmatrix}
  v_1\\
 v_2\\
 \end{pmatrix}$.

Then  by using Theorem \ref{th:1} we get the second class of irreducible representations of $SL(2,\RR)\ltimes \RR^4$ without non-trivial $L_1$ or $L_2$-fixed vectors:
\begin{lemma}\label{le:2}
The group action is defined by
\begin{gather*}
\rho_s: SL(2,\RR)\ltimes \RR^4\rightarrow \mathcal{B}(\mathbb{H}_s)\\
\rho_s(u,v)f(x,\xi,z)=e^{(xu_2-\xi u_1+sv_1x^{-1}+sv_1\xi z-sxzv_2)\sqrt{-1}}f(x,\xi,z),\\
\rho_s(g)f(x,\xi,z)=f\Big(dx-b\xi,a\xi-cx,z-\frac{bx^{-1}}{dx-b\xi}\Big);
 \end{gather*}
 and
 \begin{align*}
 \norm{f}_{\mathbb{H}_s}=\norm{f}_{L^2(\RR^3)},
 \end{align*}
  where $g=\begin{pmatrix}a & b \\
c & d
\end{pmatrix}$, $u=\begin{pmatrix}u_1 \\
u_2
 \end{pmatrix}$ and $v=\begin{pmatrix}v_1 \\
v_2\\
 \end{pmatrix}$.

Computing derived representations, we get
 \begin{gather*}
 X=-x\partial_x+\xi\partial_\xi,\qquad V=-x\partial_\xi,\qquad U=-\xi\partial_x-x^{-2}\partial_z,\notag\\
  Y_1=-\xi\sqrt{-1},\quad Y_2=x\sqrt{-1},\notag\\
  Y_3=(sx^{-1}+s\xi z)\sqrt{-1},\quad Y_4=-sxz\sqrt{-1}.
 \end{gather*}
 If we take the Fourier transformation on $\xi$ ($\hat{f}_y(x,y,z)=\frac{1}{\sqrt{2\pi}}\int_{\RR}f(x,\xi,z)e^{-iy \xi}d\xi$), we get the Fourier model:
\begin{gather}\label{for:118}
 X=-I-x\partial_x-y\partial_y,\quad V=-xy\sqrt{-1},\quad U=-\partial_{xy}\sqrt{-1}-x^{-2}\partial_z,\notag\\
  Y_3=-zs\partial_y+\sqrt{-1}sx^{-1},\quad Y_4=-sxz\sqrt{-1} \notag\\
Y_2=x\sqrt{-1},\qquad Y_1=\partial_y.
 \end{gather}
\end{lemma}

\section{Sobolev estimates for solutions of cohomological equations}

\subsection{Coboundary for classical horocycle map}
\label{sec:6}
For the classical horocycle map defined by the $\mathfrak{sl}(2,\RR)$-matrix $U=\begin{pmatrix}
  0 & 1 \\
  0 & 0
\end{pmatrix}$, there is a classification of the obstructions to the solution of the cohomological
equation established by the first author \cite{T1}. That is, for any $F\in \mathcal{H}^\infty$, we know precisely the condition under which the equation
\begin{align}\label{for:18}
 \mathcal{L}_{U}f=F
\end{align}
has a solution $f$. Let
\begin{align*}
    \mathcal{E}_U(\mathcal{H})=\{\mathcal{D}\in \mathcal{E}(\mathcal{H}): \mathcal{L}_U\mathcal{D}=\mathcal{D} \}\quad\text{and}\quad\mathcal{H}_U^{-k}=\{\mathcal{D}\in \mathcal{H}^{-k}: \mathcal{L}_U\mathcal{D}=\mathcal{D} \}.
\end{align*}

\begin{theorem}(\cite{T1})\label{th:2} Suppose $\pi$ has a spectral gap of $u_0$ (defined at the end of Section \ref{sect:SL2R_reps}). For any $s\geq0$ there is a constant $C_{s, u_0} > 0$ such that for all $g\in \mathcal{H}_U^{-(3s + 4)},$ there is a unique solution $f\in \mathcal{H}$
to the cohomological equation $\mathcal{L}_{U}f=F$, which satisfies
\begin{align*}
\|f\|_{s} \leq C_{s, u_0} \| g \|_{3s+4}.
\end{align*}
\end{theorem}
\begin{remark}\label{re:5}
In \cite{T2}, we have refined each step of the argument in \cite{T1}, which improves upon
the estimates in \cite{T1} both with respect to the time step and loss of
regularity. More precisely, we get the following: for any $s>0$ and $\epsilon>0$ there is a constant $C_{s, \epsilon, u_0} > 0$ such that for all $g\in \mathcal{H}_U^{-(2s+1+\epsilon)},$ there is a unique solution $f\in \mathcal{H}$
to the cohomological equation $\mathcal{L}_{U}f=F$, which satisfies
\begin{align}\label{for:34}
\|f\|_{s} \leq C_{s,\epsilon, u_0} \| g \|_{2s+1+\epsilon}.
\end{align}
\end{remark}
Generally, for the cohomological equation \eqref{for:18}, the existence of a bonafide solution for a smooth coboundary doesn't necessarily imply the existence of a smooth solution. The next result shows that
under certain conditions, the solution is automatically smooth.
\begin{theorem}\cite{T1}\label{th:100} Suppose $\pi$ has a spectral gap of $u_0$. Suppose $F\in \mathcal{H}^\infty$ and there is $f\in \mathcal{H}$ such that
$\mathcal{L}_{U}f=F$. Then
\begin{enumerate}
  \item if there is $h\in \mathcal{H}^\infty$ such that $Uh=F$, then $f\in \mathcal{H}^\infty$;

  \smallskip
  \item if $\pi$ only contains the principal series or complementary series, then $f\in \mathcal{H}^\infty$.
\end{enumerate}

\end{theorem}
\begin{remark}\label{re:1}
In fact, the above theorem applies to any irreducible
unitarizable representations of $\mathfrak{sl}(2,\RR)$; that is, those representations that arise
as the derivatives of irreducible unitary representations of some Lie group whose
Lie algebra is $\mathfrak{sl}(2,\RR)$. In fact,
all such representations can be realized from irreducible unitary representations
of some finite cover of $SL(2,\RR)$. In turn, all of these are unitarily equivalent to
irreducible representations of $SL(2,\RR)$ itself \cite{tan}.
\end{remark}

\subsection{Coboundary for unipotent maps in any Lie group $G$} At the beginning of the section we recall the following direct consequence of the well known Howe-Moore theorem
on vanishing of the matrix coefficients at infinity \cite{howe-moore}: if $G$ is a simple Lie group with finite center and $\rho$ is a unitary representation of $G$ without a non-zero $G$-invariant vector and $M$ is a closed non-compact subgroup of $G$, then $\rho$ has
no $M$-invariant vector.

 We take notations in Section \ref{sec:101} and \ref{sec:17}.  We present two technical results in this part, which are suggested by L. Flaminio. Lemma \ref{le:10} and the ``centralizer trick" in Proposition \ref{le:6} will pay a key role in next section.
\begin{lemma}\label{le:10}
Suppose $G$ is a simple Lie group and $(\pi,\mathcal{H})$ contains no non-trivial $G$-fixed vectors. Also suppose
$\{\exp(n\mathfrak{u})\}_{n\in\ZZ}$ is a non-compact subgroup for some $\mathfrak{u}\in \mathfrak{g}$. For any $v_1,\,v_2\in\mathcal{H}$, if there exists $Y\in\mathcal{U}(\mathfrak{g})$, where $\mathcal{U}(\mathfrak{g})$ is the universal enveloping algebra of $\mathfrak{g}$, such that $\langle v_1,\,\mathcal{L}_{\mathfrak{u}} h\rangle=\langle v_2,\,Y\mathcal{L}_{\mathfrak{u}} h\rangle$ for any $h\in\mathcal{H}^\infty$, then $v_1=Y'v_2$, where $Y'$ is the adjoint element of $Y$ in $\mathcal{U}(\mathfrak{g})$.
\end{lemma}
\begin{proof}
Set $H_{\mathfrak{u}}=\{\exp(n\mathfrak{u})\}_{n\in\ZZ}$. Thanks to Howe-Moore, we see that $\pi$ has
no non-trivial $H_{\mathfrak{u}}$-invariant vectors. Since the orthogonal complement of $\mathcal{L}_{\mathfrak{u}}$-coboundary
are the $H_{\mathfrak{u}}$-invariant vectors, which are zero, we see that $v_1=Y'v_2$.
\end{proof}
\begin{definition}\label{de:3}
Suppose $\mathfrak{u}\in \mathfrak{g}$ is a nilpotent element.  The
Jacobson-Morosov theorem asserts the existence of an element $\mathfrak{u}'\in \mathfrak{g}$
such that $\{\mathfrak{u}, \mathfrak{u}', [\mathfrak{u},\mathfrak{u}']\}$ span a three-dimensional Lie algebra $\mathfrak{g}_\mathfrak{u}$ isomorphic to $\mathfrak{sl}(2,\RR)$. Set $G_\mathfrak{u}$ to be the connected subgroup in $G$ with Lie algebra spanned by $\{\mathfrak{u}, \mathfrak{u}', [\mathfrak{u},\mathfrak{u}']\}$.
\end{definition}
Since $G$ has finite center, $G_\mathfrak{u}$ is
isomorphic to a finite cover of $PSL(2,\RR)$. We have the following result which can be viewed as an extension of Theorem \ref{th:2} and Remark \ref{re:5}.
\begin{proposition}\label{le:6}
Suppose there is a spectral gap of $u_0$ for $(\pi\mid_{G_\mathfrak{u}},\,\mathcal{H})$. Suppose $g\in \mathcal{H}^{2s+2}$, $s\geq 0$ and $\mathcal{D}(g)=0$ for all $\mathcal{D}\in\mathcal{H}_{\mathfrak{u}}^{-(2s+2)}$.  Fix a norm $\abs{\,\cdot\,}$ on $\mathfrak{g}$. Set
\begin{align*}
\mathfrak{N}_{\mathfrak{u}}=\{Y\in \mathfrak{g}:\abs{Y}\leq 1\text{ and }[Y,\mathfrak{u}]=0\}.
\end{align*}Then the cohomological equation $\mathcal{L}_{\mathfrak{u}}f=g$ has a solution $f\in \mathcal{H}$ which satisfies
the Sobolev estimates
\begin{align}\label{for:25}
 \norm{Y^mf}_{G_\mathfrak{u},t}\leq C_{u_0, m,s,t}\norm{g}_{2t+m+2},\qquad \forall\,Y\in \mathfrak{N}_\mathfrak{u}
\end{align}
if $2t+m<2s$.
\end{proposition}
\begin{proof} As a direct consequence of Theorem \ref{th:2}, Remark \ref{re:1} and Remark \ref{re:5}, we see that the cohomological equation $\mathcal{L}_{\mathfrak{u}}f=g$ has a solution $f\in \mathcal{H}$ with estimates
\begin{align}\label{for:24}
 \norm{f}_{G_\mathfrak{u},t}\leq C_{s,t,u_0}\norm{g}_{G_\mathfrak{u},2s+2},\qquad \forall \,\,t\leq s.
\end{align}
As a first step to get the Sobolev estimates along $\mathfrak{N}_\mathfrak{u}$, we prove the following fact:
$(*)$ if $\mathcal{D}\in \mathcal{H}_{\mathfrak{u}}^{-k}$ then $Y\mathcal{D}\in \mathcal{H}_{\mathfrak{u}}^{-k-1}$ for any $Y\in\mathfrak{N}_{\mathfrak{u}}$.

By definition $Y\mathcal{D}(h)=-\mathcal{D}(Yh)$ for any $h\in \mathcal{H}^\infty$. Then
\begin{align*}
(Y\mathcal{D})(\mathcal{L}_{\mathfrak{u}}h)&=-\mathcal{D}(Y\mathcal{L}_{\mathfrak{u}}h)=-\mathcal{D}(\mathcal{L}_{\mathfrak{u}}Yh)=0,
\end{align*}
which proves Fact $(*)$.

For any $Y\in\mathfrak{N}_{\mathfrak{u}}$, from Fact $(*)$ we see that $\mathcal{D}(Yg)=0$ for any $\mathcal{D}\in \mathcal{H}_{\mathfrak{u}}^{-(2s+2)+1}$.
Then Theorem \ref{th:2} and Remark \ref{re:1} imply that the equation $\mathcal{L}_{\mathfrak{u}}f_1=Yg$ has a solution $f_1\in \mathcal{H}$ with sobolev estimates
\begin{align}\label{for:22}
\norm{f_1}_{G_\mathfrak{u},t}\leq C_{s,t,u_0}\norm{Yg}_{G_\mathfrak{u},2t+2}\leq C_{s,t,u_0}\norm{g}_{2t+3},\quad \forall\, t<s-\frac{1}{2}.
\end{align}
On the other hand, for any $h\in\mathcal{H}^\infty$ we have
\begin{align*}
    \langle f_1,\mathcal{L}_{-\mathfrak{u}}h\rangle&= \langle \mathcal{L}_{\mathfrak{u}}f_1,h\rangle=\langle Yg,\,h\rangle=-\langle g,\,Yh\rangle=-\langle \mathcal{L}_{\mathfrak{u}}f,\,Yh\rangle\\
    &=-\langle f,\,\mathcal{L}_{-\mathfrak{u}}Yh\rangle=-\langle f,\,Y\mathcal{L}_{-\mathfrak{u}}h\rangle.
\end{align*}
By assumption there is no non-trivial $G$-invariant vectors. By Lemma \ref{le:10}, we get $f_1=Yf$.

From \eqref{for:24} and \eqref{for:22} we have
\begin{align*}
\norm{Yf}_{G_\mathfrak{u},t}=\norm{f_1}_{G_\mathfrak{u},t}\leq C_{s,t,u_0}\norm{g}_{2t+3},\quad \forall\,t<s-\frac{1}{2}.
\end{align*}
Then we just proved \eqref{for:25} when $m=1$. By induction suppose \eqref{for:25} holds when $m\leq k<2s$. Next we will prove the case when  $m=k+1<2s$.
Fact $(*)$ shows that $\mathcal{D}(Y^{k+1}g)=0$ for all $\mathcal{D}\in\mathcal{H}_{\mathfrak{u}}^{-(2s+2)+k+1}$. Then it follows from  Theorem \ref{th:2} and Remark \ref{re:1} that the equation
\begin{align*}
 \mathcal{L}_{\mathfrak{u}}f_{k+1}=Y^{k+1}g
\end{align*}
has a solution $f_{k+1}\in \mathcal{H}$ with Sobolev estimates
\begin{align}\label{for:19}
\norm{f_{k+1}}_{G_\mathfrak{u},t}\leq C_{s,t,u_0}\norm{Y^{k+1}g}_{G_\mathfrak{u},2t+2}\leq C_{s,t,u_0}\norm{g}_{2t+k+3}
\end{align}
if $2t+k+3<2s+2$. On the other hand, for any $h\in\mathcal{H}^\infty$ we have
\begin{align*}
    \langle f_{k+1},\mathcal{L}_{-\mathfrak{u}}h\rangle&= \langle \mathcal{L}_{\mathfrak{u}}f_{k+1},\,h\rangle=\langle Y^{k+1}g,\,h\rangle=(-1)^{k+1}\langle g,\,Y^{k+1}h\rangle\notag\\
    &=(-1)^{k+1}\langle \mathcal{L}_{\mathfrak{u}}f,\,Y^{k+1}h\rangle=(-1)^{k+1}\langle f,\,Y^{k+1}\mathcal{L}_{-\mathfrak{u}}h\rangle\notag\\
    &=-\langle Y^kf,\,Y\mathcal{L}_{-\mathfrak{u}}h\rangle.
    \end{align*}
Here we used the assumption that $Y^kf\in \mathcal{H}$. This shows that
\begin{align*}
 Y^{k+1}f=f_{k+1}
\end{align*}
by Lemma \ref{le:10}. From \eqref{for:24} and \eqref{for:19} we have
\begin{align*}
\norm{Y^{k+1}f}_{G_\mathfrak{u},t}=\norm{f_{k+1}}_{G_\mathfrak{u},t}\leq C_{s,t,u_0}\norm{g}_{2t+k+3}
\end{align*}
if $2t+k+1<2s$. Then we proved the case when $m=k+1$ and thus finish the proof.
\end{proof}

\subsection{Coboundary for the unipotent map in irreducible component of $G=SL(2,\RR)\ltimes\RR^2$}
 In this section we take notations in Section \ref{sec:102}.  In this part, we will prove the following:
\begin{theorem}\label{po:2}
 For any irreducible representation $(\rho_{t},\,\mathcal{H}_{t})$ of $SL(2,\RR)\ltimes \RR^2$, here we consider the Fourier model, then we have:

 \begin{enumerate}
 \item [(a)] if the cohomological equation $\mathcal{L}_{V}f=g$ has a solution $f\in \mathcal{H}_t^s$, $s>6$, then $f$ satisfies the following estimates:
 \begin{align*}
  \norm{f}_{r}\leq C_{r}\norm{g}_{2r+6},\qquad \forall\,\,0\leq r\leq\frac{s-6}{2}.
\end{align*}

\bigskip

 \item [(b)] if $g\in \mathcal{H}_{t}^\infty$ and for any $n\in\ZZ$, $\lim_{y\rightarrow \frac{2n\pi}{x}}g(x,y)=0$ for almost all $x\in\RR$, then the cohomological equation $\mathcal{L}_{V}f=Y_2g$  has a solution $f\in \mathcal{H}_{t}^\infty$ satisfying
\begin{align*}
  \norm{f}_{s}\leq C_{s}\norm{g}_{2s+6},\qquad \forall\,s\geq 0.
\end{align*}

  \medskip
     \item [(c)] if $g\in \mathcal{H}_{t}^\infty$ and  the cohomological equation $\mathcal{L}_{V}f=g$  has a solution $f\in \mathcal{H}_{t}$, then  $f\in(\mathcal{H}_t)^\infty_{SL(2,\RR)}$ and $Y_2f\in \mathcal{H}_{t}^\infty$.

\medskip

     \item [(d)] if $g\in \mathcal{H}_{t}^\infty$ and the cohomological equation $\mathcal{L}_{V}f=g$  has a solution $f\in \mathcal{H}_{t}$, then  $Y_1f\in \mathcal{H}_{t}$ and $Y_2Y_1f\in \mathcal{H}_t^\infty$.

\medskip
     \item [(e)]  if $g\in \mathcal{H}_{t}^\infty$ and the cohomological equation $\mathcal{L}_{V}f=g$  has a solution $f\in \mathcal{H}_{t}$, then  $f\in \mathcal{H}_{t}^\infty$ and satisfies
\begin{align*}
  \norm{f}_{s}\leq C_{s}\norm{g}_{2s+6},\qquad \forall\,s\geq 0.
\end{align*}

 \end{enumerate}
\end{theorem}
\begin{remark}The purpose of Theorem \ref{po:2} and \eqref{for:119} of Lemma \ref{le:15} is a preparation to prove that the solution $f$ in Theorem \ref{th:6} is a smooth vector. As we will see in the next section, $\GG$ is built of subgroups isomorphic to $SL(2,\RR)\ltimes\RR^2$ and $SL(2,\RR)\times\RR$ containing $\{\exp(t\mathfrak{v})\}_{t\in\RR}$.

Also note that $\GG$ is generated by subgroups isomorphic to $SL(2,\RR)\ltimes\RR^2$. Then Corollary \ref{cor:3} shows that $f$ is smooth on these semidirect products. Specially, if $\Gamma$ is cocompact and $\mathcal{H}=L_0^2(\GG/\Gamma)$, the space of square integrable functions on $\GG/\Gamma$ with zero average, then the global smoothness of $f$ is a direct consequence of subelliptic regularity theorem on compact manifolds (see Theorem \ref{th:5}).
\end{remark}

 The subsequent discussion will be devoted to the proof of
this theorem. Recall notations in Section \ref{sec:17}.

\begin{definition}\label{de:1}
For any function $f(x,y)$ on $\RR^2$ and any $x\in\RR$, we associate a function $f_x$ defined on $\RR$ by $f_x(y)=f(x,y)$. Then for any function $f(x_1,\cdots,x_n)$ on
$\RR^n$ and $(x_{k_1},\cdots,x_{k_m})\in\RR^m$, $f_{x_{k_1},\cdots,x_{k_m}}$ is an obviously defined function on $\RR^{n-m}$.
\end{definition}
The following lemma gives the necessary condition under which there exists a solution to the cohomological equation $\mathcal{L}_{V}f=g$ in each irreducible component $(\rho_{t},\,\mathcal{H}_t)$:
\begin{lemma}\label{le:9}
 Suppose $g\in\mathcal{H}_t$ and $Y_1g\in \mathcal{H}_t$. We use the Fourier model. Then:
 \begin{enumerate}
 \item \label{for:109}if the cohomological equation $\mathcal{L}_{V}f=g$ has a solution $f\in \mathcal{H}_t$, then for any $n\in\ZZ$, $\lim_{y\rightarrow \frac{2n\pi}{x}}g(x,y)=0$ for almost all $x\in\RR$.

   \medskip
   \item \label{for:110} if $Y_1^ig\in \mathcal{H}_t$, $i=1,\,2$ and for any $n\in\ZZ$, $\lim_{y\rightarrow \frac{2n\pi}{x}}g(x,y)=0$ for almost all $x\in\RR$, then $f(x,y)=\frac{g(x,y)x}{e^{-xy\sqrt{-1}}-1}\in \mathcal{H}_t$ with the estimate
       \begin{align*}
        \norm{f}\leq C(\norm{g}+\norm{Y_1g}+\norm{Y_1^2g}).
       \end{align*}

 \end{enumerate}
 \end{lemma}
\begin{proof}
\textbf{\emph{Proof of \eqref{for:109}}} For any $h(x,y)\in L^2(\RR^2)$ denote by $\Omega_h\subset\RR$ a full Lebesgue measure set such that $h_x\in L^2(\RR)$ for any $x\in \Omega_h$. Using \eqref{for:139} of Lemma \ref{le:1}, the equation $\mathcal{L}_{V}f=g$ has the expression:
 \begin{align*}
   f(x,y)(e^{-xy\sqrt{-1}}-1)=g(x,y)
 \end{align*}
 which shows that
  \begin{align}\label{for:124}
    f(x,y)=\frac{g(x,y)}{e^{-xy\sqrt{-1}}-1}.
  \end{align}
  The Sobolev imbedding theorem shows that for any $x\in \Omega_g\bigcap \Omega_{Y_1g}$, $g_x$ are continuous functions. Then \eqref{for:124} implies that
for any $x\in \Omega_g\bigcap \Omega_{Y_1g}\backslash 0$ and any $n\in\ZZ$, $\lim_{y\rightarrow \frac{2n\pi}{x}}g(x,y)=0$.

\smallskip
\noindent\textbf{\emph{Proof of \eqref{for:110}}} Let
\begin{align*}
 I_{n,x}=(\frac{2n\pi}{x}-\frac{2\pi}{3|x|},\frac{2n\pi}{x}+\frac{2\pi}{3|x|})\text{  and  }J_{n,x}=[\frac{2n\pi}{x}-\frac{2\pi}{3},\frac{2n\pi}{x}+\frac{2\pi}{3}].
\end{align*}
Set
\begin{align*}
A_1&=\bigl\{(x,y)\in\RR^2: x\neq 0, \,y\in \bigcup_{n\in\ZZ} [\frac{2n\pi}{x}+\frac{2\pi}{3\abs{x}},\frac{2n\pi}{x}+\frac{4\pi}{3|x|}]\bigl\};\\
A_2&=\bigl\{(x,y)\in\RR^2: |x|\geq 1, \,y\in \bigcup_{n\in\ZZ} I_{n,x}\bigl\};\\
A_3&=\bigl\{(x,y)\in\RR^2: 0<|x|< 1, \,y\in \bigcup_{n\in\ZZ} I_{n,x}\backslash J_{n,x}\bigl\};\\
A_4&=\bigl\{(x,y)\in\RR^2: 0<|x|< 1, \,y\in \bigcup_{n\in\ZZ} J_{n,x}\bigl\}.
\end{align*}
For any $(x,y)\in A_1$, then
\begin{align*}
|f(x,y)|&=\frac{|(Y_2g)(x,y)|}{|e^{-xy\sqrt{-1}}-1|}\leq 4|(Y_2g)(x,y)|,
\end{align*}
which implies that
\begin{align}\label{for:133}
\int_{A_1}\abs{f(x,y)}^2dydx\leq 4\int_{A_1}\abs{(Y_2g)(x,y)}^2dydx\leq 4\norm{Y_2g}.
\end{align}
Set $h_{n}(t)=\frac{t-2n\pi}{e^{-t\sqrt{-1}}-1}$ for any $t\in\RR$.  We can write
\begin{align}\label{for:132}
 f(x,y)=\sqrt{-1}\frac{g(x,y)}{y-\frac{2n\pi}{x}}\cdot h_{n}(xy).
\end{align}
If $(x,y)\in A_3$, then there exists $n\in\ZZ$ such that $y\in I_{n,x}\backslash J_{n,x}$, which implies that $\abs{y-\frac{2n\pi}{x}}\geq\frac{2\pi}{3}$. By \eqref{for:132} we have
\begin{align*}
|f(x,y)|&=\bigg|\frac{g(x,y)}{y-\frac{2n\pi}{x}}\bigg|\cdot |h_{n}(xy)|\leq C|g(x,y)|.
\end{align*}
This shows that
\begin{align}\label{for:127}
&\int_{A_3}\abs{f(x,y)}^2dydx\leq C\int_{ A_3}\abs{g(x,y)}^2dydx.
\end{align}
For any $n\in\ZZ$ and $x\in \bigcap_{i=0}^2 \Omega_{Y_1^ig}$ we have
\begin{align}
&f(x,y)=\sqrt{-1}\frac{g(x,y)}{y-\frac{2n\pi}{x}}\cdot h_{n}(xy)\notag\\
&=\sqrt{-1}\frac{g(x,y)-g(x,\frac{2n\pi}{x})}{y-\frac{2n\pi}{x}}\cdot h_{n}(xy)\notag\\
&\overset{\text{(1)}}{=}\sqrt{-1}\int_{0}^{1}\partial_yg\bigl(x,ty+(1-t)\frac{2n\pi}{x}\bigl)dt\cdot h_{n}(xy)\notag\\
&=\sqrt{-1}\int_{0}^{1}(Y_1g)_x\bigl(ty+(1-t)\frac{2n\pi}{x}\bigl)dt\cdot h_{n}(xy)\label{for:130}.
\end{align}
$(1)$ holds since Sobolev embedding theorem implies that $g_x\in C^1(\RR)$ for any $x\in \bigcap_{i=0}^2 \Omega_{Y_1^ig}$. Then for any $x\in \bigcap_{i=0}^2 \Omega_{Y_1^ig}$ and $y\in I_{n,x}$ we have
\begin{align}\label{for:125}
|f_x(y)|&\overset{\text{(1)}}{\leq} C\norm{(Y_1g)_x}_{C^0(I_{n,x})}\overset{\text{(2)}}{\leq} C\norm{(Y_1g)_x}_{W^1(I_{n,x})}.
\end{align}
Here $(1)$ holds since $h_{n}$ are uniformly bounded on $I_{n,x}$ for any $n\in\ZZ$; $(2)$ uses the Sobolev embedding theorem.

Using \eqref{for:125}, we have
\begin{align}\label{for:126}
&\int_{A_2}\abs{f(x,y)}^2dydx=\int_{ \abs{x}\geq 1}\big(\sum_{n\in\ZZ}\int_{I_{n,x}}\abs{f(x,y)}^2dy\big)dx\notag\\
&\leq \int_{\abs{x}\geq 1}\bigl(\sum_{n\in\ZZ}\int_{I_{n,x}}C\norm{(Y_1g)_x}^2_{W^1(I_n)}dy\bigl)dx\notag\\
&\overset{\text{(1)}}{=}\int_{\abs{x}\geq 1}\bigl(\sum_{n\in\ZZ}C|I_{n,x}|\cdot\norm{(Y_1g)_x}^2_{W^1(I_{n,x})}\bigl)dx\notag\\
&\leq\int_{\abs{x}\geq 1}\bigl(\sum_{n\in\ZZ}\frac{4\pi}{3}C\norm{(Y_1g)_x}^2_{W^1(I_{n,x})}\bigl)dx\notag\\
&\leq \sum_{i=0}^2C\norm{Y_1^ig}^2_{L^2(A_2)}.
\end{align}
Here in step $(1)$ $|I_{n,x}|$ denotes the length of interval $I_{n,x}$.

Using \eqref{for:125} again, we also have
\begin{align}\label{for:134}
&\int_{A_4}\abs{f(x,y)}^2dydx=\int_{ \abs{x}<1}\bigl(\sum_{n\in\ZZ}\int_{J_{n,x}}\abs{f(x,y)}^2dy\bigl)dx\notag\\
&\leq \int_{\abs{x}<1}\bigl(\sum_{n\in\ZZ}\int_{J_{n,x}}C\norm{(Y_1g)_x}^2_{W^1(J_{n,x})}dy\bigl)dx\notag\\
&=\int_{\abs{x}< 1}\bigl(\sum_{n\in\ZZ}\frac{4\pi}{3}C\norm{(Y_1g)_x}^2_{W^1(J_{n,x})}\bigl)dx\notag\\
&\leq C(\sum_{i=0}^2\norm{Y_1^ig}^2_{L^2(A_4)}).
\end{align}
It is clear that $\RR^2\backslash (\bigcup_{i=1}^4 A_i)$ is a $0$-measure set with respect to the Lebesgue measure. Then \eqref{for:133}, \eqref{for:126}, \eqref{for:127} and \eqref{for:134} imply the conclusion.

\end{proof}

Let
$G'$ denote the subgroup $\begin{pmatrix}[cc|c]
  a & 0 & v_1\\
  c & a^{-1} & v_2
\end{pmatrix}$, where $a\in\RR^+$ and $c,\,v_1,\,v_2\in\RR$. Then the Lie algebra of $G'$ is generated by $X$, $V$, $Y_1$ and $Y_2$. The next result is a crucial step in proving Theorem \ref{po:2}. We list two facts which will be used in the proof.
\begin{fact} \label{re:3}
(Lemma 6.9 of \cite{W1})
For any irreducible component $(\rho_{t},\,\mathcal{H}_{t})$ of $SL(2,\RR)\ltimes \RR^2$ under the Fourier model, if $g\in(\mathcal{H}_t)_{G'}^s$, $s>10$ and $\lim_{y\rightarrow 0}g(x,y)=0$ for almost all $x\in\RR$, then the cohomological equation $\mathcal{L}_{V}f=Y_2g$ has a solution $f\in (\mathcal{H}_t)^{s-10}_{G'}$ with the estimate
  \begin{align*}
  \norm{f}_{G',r}\leq C_{s}\norm{g}_{G',r+3}, \qquad 0\leq r\leq s-13.
\end{align*}
\end{fact}
\begin{fact}\label{fact:1}
It is easy to check the following fact: for any $s>0$, any Schwartz function $h\in \mathcal{S}(\RR)$ and any $q\in (\mathcal{H}_t)_{G',s}$, we have
\begin{align*}
 \big\|q(x,y)h(xy)\big\|_{G',s}\leq C_{h,s}\big\|q(x,y)\big\|_{G',s}.
\end{align*}
\end{fact}
\begin{lemma}\label{le:8}
 For any irreducible component $(\rho_{t},\,\mathcal{H}_{t})$ of $SL(2,\RR)\ltimes \RR^2$ under the Fourier model, if $g\in(\mathcal{H}_t)_{G'}^s$, $s>10$ and for any $n\in\ZZ$, $\lim_{y\rightarrow \frac{2n\pi}{x}}g(x,y)=0$ for almost all $x\in\RR$, then the cohomological equation $\mathcal{L}_{V}f=Y_2g$ has a solution $f\in (\mathcal{H}_t)^{s-10}_{G'}$ with the estimate
  \begin{align*}
  \norm{f}_{G',r}\leq C_{s}\norm{g}_{G',r+3},\qquad 0\leq r\leq s-13.
\end{align*}
Furthermore, if $g\in(\mathcal{H}_t)_{G'}^\infty$, then $f\in (\mathcal{H}_t)^{\infty}_{G'}$.

\end{lemma}

\begin{proof}
Let $f(x,y)=\frac{g(x,y)\cdot x\sqrt{-1}}{e^{-xy\sqrt{-1}}-1}.$ \eqref{for:110} of Lemma \ref{le:9} shows that $f\in \mathcal{H}_t$. Recall relations in \eqref{for:139}. We see that $f$ is a solution of the equation $\mathcal{L}_{V}f=Y_2g$. Next, we will give the Sobolev estimates of the solution on $G'$. We assume notations in proof of Lemma \ref{le:9}.

\smallskip
\noindent\textbf{Sobolev estimates along $Y_2$ and $V$}. Note that for any $m\leq s-2$,
\begin{gather*}
(Z^mf)(x,y)=\frac{(Z^mg)(x,y)\cdot x\sqrt{-1}}{e^{-xy\sqrt{-1}}-1}
\end{gather*}
where $Z$ stands for $Y_2$ or $V$; and for any $n\in\ZZ$, $\lim_{y\rightarrow \frac{2n\pi}{x}}(Z^mg)(x,y)=0$ for almost all $x\in\RR$. Then it follows from \eqref{for:110} of Lemma \ref{le:9} that
\begin{align}\label{for:58}
    \norm{Z^mf}\leq C\norm{Z^mg}_{G',2}\leq C\norm{g}_{G',m+2},\qquad \forall\,m\leq s-2.
 \end{align}

\noindent\textbf{Sobolev estimates along $Y_1$}. Note that
\begin{align}\label{for:11}
Y_1^nf(x,y)&=\sum_{i=0}^n(n-i)!\binom{n}{i}\frac{(Y_2^{n-i+1}Y_1^ig)(x,y)}{(e^{-xy\sqrt{-1}}-1)^{n-i+1}}
\end{align}
Then for any $(x,y)\in A_1$ and $n\leq s-1$, we have
\begin{align*}
&|Y_1^nf(x,y)|\leq \sum_{0\leq i\leq n}C_{n}|(Y_2^{n-i+1}Y_1^{i}g)(x,y)|.
\end{align*}
This shows that
\begin{align}\label{for:2}
 \int_{A_1}|Y_1^nf(x,y)|^2dydx&\leq C_n\norm{g}^2_{G',n+1}.
\end{align}
We also note that
\begin{align*}
Y_1^nf(x,y)&=\sum_{l+i+j=n}(-1)^{i+j}(\sqrt{-1})^{i+1}x^{i+1}j!\binom{n}{l,i,j}\frac{(Y_1^{l}g)(x,y)}{(y-\frac{2m\pi}{x})^{j+1}}\cdot h^{(i)}_{m}(xy)\\
&=\sum_{l+i+j=n}(-1)^{i+j}j!\binom{n}{l,i,j}\frac{(Y_2^{i+1}Y_1^{l}g)(x,y)}{(y-\frac{2m\pi}{x})^{j+1}}\cdot h^{(i)}_{m}(xy).
\end{align*}
If $(x,y)\in A_3$, there exist some $m\in\ZZ$, such that $y\in I_{m,x}\backslash J_{m,x}$. Then for any $n\leq s-1$, we have
\begin{align*}
&|Y_1^nf(x,y)|\leq \sum_{0\leq i+l\leq n}C_{n}|(Y_2^{i+1}Y_1^{l}g)(x,y)|.
\end{align*}
We get this estimate by noting that for any $i$, $h^{(i)}_{m}$ are uniformly bounded on $I_{m,x}$ for any $m\in\ZZ$.

This shows that if $n\leq s-1$, then
\begin{align}\label{for:10}
&\int_{A_3}Y_1^nf(x,y)|^2dydx\leq C_n\norm{g(x,y)}_{G',n+1}^2.
\end{align}
Using \eqref{for:130}, for any $n<s-2$, $m\in\ZZ$, $x\in \bigcap_{i=0}^{n+2} \Omega_{Y_1^ig}$ and $y\in I_{m,x}$ we also have
\begin{align}\label{for:23}
Y_1^nf(x,y)&\overset{\text{(1)}}{=}\sum_{0\leq i\leq n}\binom{n}{i}\int_{0}^{1}(Y_1^{n-i+1}g)_x\bigl(ty+(1-t)\frac{2m\pi}{x}\bigl)\cdot t^{n-i}dt\notag\\
&\cdot h^{(i)}_{m}(xy)\cdot x^i\notag\\
&=\sum_{0\leq i\leq n}\binom{n}{i}\int_{0}^{1}(Y_2^{i}Y_1^{n-i+1}g)_x\bigl(ty+(1-t)\frac{2m\pi}{x}\bigl)\cdot t^{n-i}dt\notag\\
&\cdot (\sqrt{-1})^{-i}h^{(i)}_{m}(xy).
\end{align}
The differentiation under the integral sign in $(1)$ is justified by the fact that $(Y_1^{n-i+1}g)_x\in C^0(\RR)$ for any $0\leq i\leq n$, which is guaranteed by
Sobolev embedding theorem.

Then for any $y\in I_{m,x}$ we have
\begin{align}\label{for:131}
|Y_1^nf(x,y)|&\leq C_{n}\Big\|\sum_{0\leq i\leq n}(Y_2^{i}Y_1^{n-i+1}g)_x\Big\|_{C^0(I_{m,x})}\notag\\
&\overset{\text{(1)}}{\leq} C_{n}\Big\|\sum_{0\leq i\leq n}(Y_2^{i}Y_1^{n-i+1}g)_x\Big\|_{W^1(I_{m,x})}.
\end{align}
Here $(1)$ is from the Sobolev embedding theorem.

By using \eqref{for:131}, if $\abs{x}\geq1$ (resp. $\abs{x}<1$), in \eqref{for:126} (resp. \eqref{for:134}) substituting $f$ with $Y_1^nf$ and $(Y_1g)_x$ with $\sum_{0\leq i\leq n}(Y_2^{i}Y_1^{n-i+1}g)_x$ we get
\begin{align*}
&\int_{A_l}\abs{Y_1^nf(x,y)}^2dydx\leq C_{n}\sum_{0\leq j\leq 1}\sum_{0\leq i\leq n}\norm{Y_1^jY_2^{i}Y_1^{n-i+1}g}^2_{L^2(A_l)}
\end{align*}
where $l=2$ or $4$ for any $n<s-2$.

The above estimates together with \eqref{for:2}, \eqref{for:10} imply that
\begin{align}\label{for:26}
    \norm{Y_1^nf}\leq C_n\norm{g}_{G',n+2},\qquad \forall\,n\leq s-2.
 \end{align}

\smallskip
\noindent\textbf{Sobolev estimates along $X$}. Let $p$ be a smooth function on $\RR$ satisfying:  $p(x)=1$ if $\abs{x}\leq \pi$ and $p(x)=0$ if $\abs{x}\geq \frac{3\pi}{2}$. We note that
\begin{align*}
V\big(f(x,y)p(xy)\big)=-\sqrt{-1}Y_2g(x,y)\cdot h_{0}(xy)p(xy).
\end{align*}
It follows from  Fact \eqref{fact:1} that
\begin{align*}
 \big\|g(x,y)h_{0}(xy)p(xy)\big\|_{G',t}\leq C_n\big\|g(x,y)\big\|_{G',t},\qquad 0\leq t\leq s.
\end{align*}
By Fact \ref{re:3} and above estimates, for any $0\leq n\leq s-13$ we get
 \begin{align*}
 \Big\| X^n\big(f(x,y)p(xy)\big)\Big\|&\leq C_{n}\big\|g(x,y)h_{0}(xy)p(xy)\big\|_{G',n+3}\\
 &\leq C_n\big\|g(x,y)\big\|_{G',n+3}.
\end{align*}
Hence we immediately have
\begin{align}\label{for:7}
&\int_{\abs{xy}\leq\frac{\pi}{2}}\abs{X^nf(x,y)}^2dydx\leq C_n\big\|g(x,y)\big\|_{G',n+3}
\end{align}
for any $0\leq n\leq s-13$.

If $\abs{xy}>\frac{\pi}{2}$, by the change of variable $(\omega,\,Y)=(yx^{-1},\,yx)$, the new model is $\mathcal{H}_t=L^2(\RR^2,\mu)$, where $d\mu=\abs{2\omega}^{-1}d\omega dY$. Note that $\abs{xy}>\frac{\pi}{2}$ implies that $\abs{Y}>\frac{\pi}{2}$. Set $A=\{(\omega,Y):\,\abs{Y}>\frac{\pi}{2}\}$.  The vector fields in the new model are:
\begin{gather}\label{for:137}
 V=-Y\sqrt{-1},\qquad X=-I-2Y\partial_Y.
 \end{gather}
Set $I_n=\{Y\in\RR:\abs{Y-2\pi n}<\frac{\pi}{2}\}$ and $Z=Y\partial_Y$. Let $J=\RR\backslash \bigcup I_n$. Note that in the new model, for any $n\leq s-1$, we have
 \begin{align*}
  (Z^nf)(\omega,Y)&=\sum_{i=0}^n(-1)^{n-i}(n-i)!\binom{n}{i}\frac{(V^{n-i}Z^iY_2g)(\omega,Y)}{(e^{-Y\sqrt{-1}}-1)^{n-i+1}}.
 \end{align*}
Then it follows that
 \begin{align}\label{for:3}
&\int_{J}|(Z^nf)(\omega,Y)|^2d\mu\leq C_n\norm{g}_{G',n+1}^2,
\end{align}
for any $n\leq s-1$.

Note that on any set $B\subseteq A$, we have
\begin{align}\label{for:21}
 \int_{B}\abs{\partial_Y^ng(\omega,Y)}^2d\omega dY\leq C\big\|Z^ng\big\|
\end{align}
for any $0\leq n\leq s$. The similar to \eqref{for:130}, we can write
\begin{align*}
f(\omega,Y)&=\frac{Y_2g(\omega,Y)}{e^{-Y\sqrt{-1}}-1}=\frac{Y_2g(\omega,Y)}{Y-2\pi m}\cdot h_m(Y)\\
&=\int_{0}^{1}\partial_Y(Y_2g)\bigl(\omega,tY+(1-t)2m\pi\bigl)dt\cdot h_{m}(Y),
\end{align*}
for $(\omega,Y)\in A$ and any $m\neq 0$.  Set $l_{m,t,Y}=tY+(1-t)2m\pi$.

Using \eqref{for:21}, similar to \eqref{for:23}, for any $n<s-2$, $m\neq 0$, $\omega\in \bigcap_{i=0}^{n+2} \Omega_{Z^ig}$ and $Y\in I_{m}$ we also have
 Then we have
 Then
\begin{align*}
&(Z^nf)(\omega,Y)=\sum_{i=0}^n\binom{n}{i}\int_{0}^{1}\partial^{i+1}_Y(Y_2g)\bigl(\omega,l_{m,t,Y}\bigl)t^iY^ndt\cdot h^{(n-i)}_{m}(Y)\\
&=\sum_{i=0}^n(\sqrt{-1})^{n-i}\binom{n}{i}\int_{0}^{1}(V^{n-i}Z^{i+1}Y_2g)\bigl(\omega,l_{m,t,Y}\bigl)t^i\frac{Y^n}{(l_{m,t,Y})^{n+1}}dt\cdot h^{(n-i)}_{m}(Y).
\end{align*}
Then it follows that
\begin{align*}
|(Z^nf)(\omega,Y)|&\overset{\text{(1)}}{\leq} C_{n}\sum_{0\leq i\leq n}\Big\|(V^{n-i}Z^{i+1}Y_2g)_\omega\Big\|_{C^0(I_{m})}\notag\\
&\overset{\text{(2)}}{\leq} C_{n}\sum_{0\leq i\leq n}\Big\|(V^{n-i}Z^{i+1}Y_2g)_\omega\Big\|_{W^1(I_{m})}\notag\\
&\overset{\text{(3)}}{\leq}C_{n}\sum_{0\leq i\leq n}\sum_{0\leq j\leq 1}\Big\|(Z^jV^{n-i}Z^{i+1}Y_2g)_\omega\Big\|.
\end{align*}
$(1)$ holds since for any $i$ and $n$, $h^{i}_{m}(Y)$ and $\frac{Y^n}{(l_{m,t,Y})^{n+1}}$ are uniformly bounded on $I_m$ for any $m\neq0$. In $(2)$ we use
Sobolev imbedding theorem again. In $(3)$ we use \eqref{for:21}.

Hence, for any $m\neq0$ we get
\begin{align*}
  \int_{I_m}|(Z^nf)(\omega,Y)|^2d\mu&\leq C_{n}\abs{I_m}\sum_{0\leq i\leq n}\sum_{0\leq j\leq 1} \int_{I_m}\Big|(Z^jV^{n-i}Z^{i+1}Y_2g)(\omega,Y)\Big|^2d\mu.
  \end{align*}
This implies that
\begin{align}\label{for:5}
  &\int_{\bigcup_{m\neq 0} I_m}|Z^nf(\omega,Y)|^2d\mu\notag\\
  &\leq C_{n}\sum_{0\leq i\leq n}\sum_{0\leq j\leq 1} \int_{\bigcup_{m\neq 0} I_m}\Big|(Z^jV^{n-i}Z^{i+1}Y_2g)(\omega,Y)\Big|^2d\mu\notag\\
  &\leq C_n\norm{g}_{G',n+2}
\end{align}
for any $0\leq n<s-2$.
Then \eqref{for:3} and \eqref{for:5} imply that
\begin{align*}
&\int_{\abs{Y}\geq\frac{\pi}{2}}|X^nf(X,Y)|^2d\mu\leq C_n\norm{g}_{G',n+2}
\end{align*}
for any $n<s-2$.

The above estimates together with  \eqref{for:7} show that
\begin{align}\label{for:28}
    \norm{X^nf}\leq C_n\norm{g}_{G',n+3}
 \end{align}
 for any $0\leq n\leq s-13$.

 Thus the results follow immediately from \eqref{for:58}, \eqref{for:26} and \eqref{for:28}.
\end{proof}
We are now in a position to proceed with the proof of Theorem \ref{po:2}. We list a fact which will be used in the proof.
\begin{fact}\label{fact:2}
(Theorem 6.6 of \cite{W1})
For any irreducible component $(\rho_{t},\,\mathcal{H}_{t})$ of $SL(2,\RR)\ltimes \RR^2$ (see Lemma \ref{le:1}), if $g\in \mathcal{H}_{t}^\infty$ and the cohomological equation $Vf=g$  has a solution $f\in \mathcal{H}_{t}$, then  $f\in \mathcal{H}_{t}^\infty$ and satisfies
\begin{align*}
  \norm{f}_{s}\leq C_{s}\norm{g}_{s+5},\qquad \forall\,s\geq 0.
\end{align*}
Furthermore, if $t\neq0$, and if $\lim_{y\rightarrow 0}g(x,y)=0$ for almost all $x\in\RR$, then the cohomological equation $Vf=g$ has a solution $f\in \mathcal{H}_t$; if $t=0$, then $\rho_{0}$ only contains principal series.
\end{fact}

\subsection{Proof of Theorem \ref{po:2}} \emph{\textbf{Proof of $(a)$}}. Since $f\in\mathcal{H}_{t}^s$, immediately we see that $\mathcal{D}(g)=0$ for all $\mathcal{D}\in (\mathcal{H}_t)_{\mathfrak{u}}^{-s+1}$. Note that $[V,Y_2]=0$. By using Proposition \ref{le:6} we get the estimates
\begin{align}\label{for:20}
    \norm{Y_2^mf}_{SL(2,\RR),s}<C_{s,m}\norm{g}_{2s+m+2},\qquad \forall \,\,s\geq0.
\end{align}
Note that the constants $C_{s,m}$ are independent of the parameter $t$ since all $\rho_t\mid_{SL(2,\RR)}$ are outside a fixed neighborhood the trivial representation in the sense of Fell topology  by Remark \ref{re:2}.

Since $f\in\mathcal{H}_{t}$ and $g\in\mathcal{H}_{t}^{s}$, $s>1$, it follows from \eqref{for:109} of Lemma \ref{le:9} that
for any $n\in\ZZ$, $\lim_{y\rightarrow \frac{2n\pi}{x}}g(x,y)=0$ for almost all $x\in\RR$. Note that
\begin{align*}
 Y_2(\mathcal{L}_{V}f)=\mathcal{L}_{V}(Y_2f)=Y_2g.
\end{align*}
By using \eqref{for:26} we get the estimates
\begin{align*}
  \norm{Y_2f}_{G',r}\leq C_{s}\norm{Y_2g}_{G',r+2}\leq C_{s}\norm{g}_{G',r+3},
\end{align*}
for $0\leq r\leq s-3$. From above estimates and \eqref{for:20}, by using Theorem \ref{th:4} we see that  $Y_2f$ satisfies the estimates
\begin{align}\label{for:65}
    \norm{Y_2f}_{r}\leq C_{r}\norm{g}_{2r+4},
\end{align}
for any $r\leq\frac{s-4}{2}$. Note that
\begin{align*}
 Y_1\rho_t(\exp V)f=\rho_t(\exp V)(Y_1-Y_2)f.
\end{align*}
Then for any $m\in\NN$:
\begin{align*}
Y_1^m(\mathcal{L}_{V}f)&=\rho_t(\exp V)(Y_1-Y_2)^mf-Y_1^mf\\
&=\mathcal{L}_{V}(Y_1^mf)+\sum_{1\leq j\leq m}(-1)^j\binom{m}{j}\rho_t(\exp V)(Y_1^{m-j}Y_2^jf).
\end{align*}
This shows that
\begin{align}\label{for:27}
 \mathcal{L}_{V}(Y_1^mf)=Y_1^mg-\sum_{1\leq j\leq m}(-1)^j\binom{m}{j}\rho_t(\exp V)(Y_1^{m-j}Y_2^jf).
\end{align}
This gives the Sobolev estimates of $f$ along $Y_1$:
\begin{align}\label{for:60}
\norm{Y_1^mf}&\overset{\text{(1)}}{\leq} C\parallel Y_1^mg-\sum_{1\leq j\leq m}(-1)^j\binom{m}{j}\rho_t(\exp V)(Y_1^{m-j}Y_2^jf)\parallel_2\notag\\
&\leq C\norm{g}_{m+2}+C_m\norm{Y_2f}_{m+1}\notag\\
&\overset{\text{(2)}}{\leq} C_m\norm{g}_{2m+4},
\end{align}
for any $m\leq\frac{s-4}{2}$. $(1)$ follows from Theorem \ref{th:2} and Remark \ref{re:5}; and $(2)$ holds because of \eqref{for:65}. As an immediate consequence of \eqref{for:20}, \eqref{for:60} and Theorem \ref{th:4} we get
\begin{align*}
    \norm{f}_{r}\leq C_r\norm{g}_{2r+6},
\end{align*}
for any $r\leq\frac{s-6}{2}$,  which proves $(a)$.

\medskip

\emph{\textbf{Proof of $(b)$}}. It follows from Lemma \ref{le:8} that $f\in (\mathcal{H}_t)^\infty_{G'}$. Note that the assumption implies that $\lim_{y\rightarrow 0} g(x,y)=0$ for almost all $x\in\RR$. Then Fact \ref{re:3} shows that equation $Vh=Y_2g$ has a solution $h\in \mathcal{H}_{t}^\infty$. Hence by Theorem \ref{th:100}, $f\in(\mathcal{H}_t)^\infty_{SL(2,\RR)}$. Since the linear span of the Lie algebras of $G'$ and $SL(2,\RR)$ cover the Lie algebra of $SL(2,\RR)\ltimes\RR^2$, by using Theorem \ref{th:4} we see that $f\in \mathcal{H}_{t}^\infty$. The Sobolev estimates of $f$ follow from $(a)$ immediately. This proves $(b)$.

\medskip

\emph{\textbf{Proof of $(c)$}}. When $t=0$, by Fact \ref{fact:2}, we see that $\rho_0$ only contains the principal series. Then by Theorem \ref{th:100} we see that
$f\in(\mathcal{H}_0)^\infty_{SL(2,\RR)}$. When $t\neq 0$, it follows from Fact \ref{fact:2} that if the equation $Vh=g$ has a solution $h\in \mathcal{H}_{t}^\infty$,
by noting that the assumption implies that $\lim_{y\rightarrow 0} g(x,y)=0$ for almost all $x\in\RR$. Then by Theorem \ref{th:100} again, we also get that $f\in(\mathcal{H}_t)^\infty_{SL(2,\RR)}$.

\eqref{for:109} of Lemma \ref{le:9} and $(b)$ show that the equation $\mathcal{L}_Vh=Y_2g$ has a solution $h\in \mathcal{H}_{t}^\infty$. For any $\omega\in \mathcal{H}_t^\infty$, we have
\begin{align*}
    \langle h,\,\mathcal{L}_{-V}\omega\rangle&=\langle Y_2g,\,\omega\rangle=-\langle g,\,Y_2\omega\rangle=-\langle \mathcal{L}_Vf,\,Y_2\omega\rangle=\langle f,\,Y_2\mathcal{L}_{-V}\omega\rangle.
    \end{align*}
This shows that $Y_2f=h$ by Proposition \ref{cor:1} and Lemma \ref{le:10}. Then $Y_2f\in \mathcal{H}_t^\infty$.
\medskip

\emph{\textbf{Proof of $(d)$}}. From $(c)$ we see that $f\in(\mathcal{H}_t)^\infty_{SL(2,\RR)}$ and $Y_2f\in \mathcal{H}_t^\infty$. Since $Vf$ is the solution of the equation
\begin{align*}
 \mathcal{L}_V(Vf)=Vg,
\end{align*}
$Y_2Vf\in \mathcal{H}_t^\infty$ from $(c)$.

By commutator relation $XY_2+Y_2=Y_2X$, for any $h\in \mathcal{H}_{t}^\infty$ we have
\begin{align*}
 \langle Xf,\,Y_2h\rangle&=-\langle f,\,XY_2h\rangle=-\langle f,\,(Y_2X-Y_2)h\rangle\overset{\text{(1)}}{=}-\langle XY_2f+Y_2f,\,h\rangle.
\end{align*}
In $(1)$ we used that $Y_2f\in \mathcal{H}_t^\infty$. Hence we get that $Y_2Xf=XY_2f+Y_2f\in \mathcal{H}_t^\infty$.

Note that
\begin{align*}
 U\rho_t(\exp V)f&=\rho_t(\exp V)(U+X-V)f,\qquad\text{and}\\
 X\rho_t(\exp V)f&=\rho_t(\exp V)(X-2V)f.
\end{align*}
Then
\begin{align*}
  Ug=U(\mathcal{L}_Vf)&=\rho_t(\exp V)(U+X-V)f-Uf\\
  &=\mathcal{L}_V(Uf)+\rho_t(\exp V)(X-2V)f+V\rho_t(\exp V)f\\
  &=\mathcal{L}_V(Uf)+X\rho_t(\exp V)f+V\rho_t(\exp V)f\\
  &=\mathcal{L}_V(Uf)+(Xf+Xg)+(Vf+Vg)
\end{align*}
This shows that $Uf$ satisfies the equation
\begin{align*}
 \mathcal{L}_V(Uf)=Ug-Vg-Xg-Vf-Xf.
\end{align*}
By noting that $Y_2Ug-Y_2Vg-Y_2Xg-Y_2Vf-Y_2Xf\in \mathcal{H}_t^\infty$ by previous arguments, it follows from \eqref{for:109} of Lemma \ref{le:9} and Lemma \ref{le:8} that the equation
\begin{align*}
 \mathcal{L}_Vh=Y_2\big(Y_2Ug-Y_2Vg-Y_2Xg-Y_2Vf-Y_2Xf\big)
\end{align*}
show that the above equation has a solution $h\in \mathcal{H}_{t}$; furthermore, from $(b)$ we get that $h\in \mathcal{H}_t^\infty$.

For any $\omega\in \mathcal{H}_t^\infty$, we have
\begin{align*}
    \langle h,\,\mathcal{L}_{-V}\omega\rangle&=\langle \mathcal{L}_Vh,\,\omega\rangle=\langle Y_2^2\big(Ug-Vg-Xg-Vf-Xf\big),\,\omega\rangle\\
    &=\langle Ug-Vg-Xg-Vf-Xf,\,Y_2^2\omega\rangle=\langle \mathcal{L}_V(Uf),\,Y_2^2\omega\rangle\\
    &=\langle Uf,\,Y_2^2\mathcal{L}_{-V}\omega\rangle.
    \end{align*}
This shows that $h=Y_2^2(Uf)\in \mathcal{H}_t^\infty$ by Proposition \ref{cor:1} and Lemma \ref{le:10}. Since $Y_2^m(Uf)\in \mathcal{H}_t$ for any $m\geq2$ It follows from Theorem \ref{th:4} that $Y_2(Uf)\in \mathcal{H}_t$.

Using the relation  $UY_2-Y_2U=Y_1$, for any $h\in \mathcal{H}_{t}^\infty$ we have
\begin{align*}
 \langle f,\,Y_1h\rangle&=\langle f,\,(UY_2-Y_2U)h\rangle=\langle (Y_2U-UY_2)f,\,h\rangle.
\end{align*}
In the last equation we used the facts that $Y_2f\in \mathcal{H}_t^\infty$ and $Y_2(Uf)\in \mathcal{H}_t$.
This shows that $Y_1f=(Y_2U-UY_2)f\in \mathcal{H}_t$ and $Y_2Y_1f=(Y_2^2U-Y_2UY_2)f\in \mathcal{H}_t^\infty$.

\medskip

\emph{\textbf{Proof of $(e)$}}. Using \eqref{for:27} for $m=1$, we have
\begin{align*}
  \mathcal{L}_V(Y_1f)=Y_1g+Y_2\rho_t(\exp V)f=Y_1g+Y_2(f+g).
\end{align*}
Note that $Y_2f\in \mathcal{H}_t^\infty$. Then it follows from $(d)$ that $Y_1^2f\in \mathcal{H}_t$ and $Y_2Y_1^2f\in \mathcal{H}_t^\infty$.

Inductively, suppose for any $m\leq k$, $Y_1^m\in \mathcal{H}_t$ and $Y_2Y_1^mf\in \mathcal{H}_t^\infty$; furthermore, suppose $Y_1^m$ satisfies the equation
\begin{align}\label{for:61}
  \mathcal{L}_V(Y_1^mf)=f_m
\end{align}
where $f_m\in \mathcal{H}_t^\infty$. Then $(d)$ shows that $Y_1^{k+1}f\in \mathcal{H}_t$ and $Y_2Y_1^{k+1}f\in \mathcal{H}_t^\infty$.

By using \eqref{for:61} for $m=k$ and \eqref{for:27} for $m=1$ we see that
\begin{align*}
 \mathcal{L}_V(Y_1^{k+1}f)=Y_1f_m+Y_2(f_m+Y_1^kf).
\end{align*}
By assumption, $Y_1f_m+Y_2(f_m+Y_1^kf)\in \mathcal{H}_t^\infty$. Then we proved the case for $m=k+1$. This shows that $Y_1^nf\in \mathcal{H}_t$ for any $n\in \NN$. This and together with  $(c)$ implies that $f\in \mathcal{H}_t^\infty$ by Theorem \ref{th:4}. The estimates of $f$ follows from $(a)$ immediately.

\begin{remark}\label{re:6}
For the flow, we have tame solution to the cohomological
equation, see Fact \ref{fact:2}. From the proof of Theorem \ref{po:2}, we see that the solution to the discrete horocycle map loses tame estimates in every direction except $Y_2$. This is because the estimates of the solution in the proof are based on the no-tame estimates in Theorem \ref{th:2}.
\end{remark}

\subsection{Coboundary for the unipotent map in irreducible component of $G=SL(2,\RR)\ltimes\RR^4$} In this section we prove an important result, which plays an essential role proving Theorem \ref{th:101}. In this part, we assume notations in Section \ref{sec:10}. Let
$G''$ denote the subgroup
\begin{align*}
\begin{pmatrix}[ccc|c]
  a & 0 & u_1& v_1\\
c & a^{-1} & u_2& v_2
\end{pmatrix}
\end{align*}
where $a\in\RR^+$ and $c,\,v_1,\,v_2,\,u_1,\,u_2\in\RR$. Then the Lie algebra of $G''$ is generated by $X$, $V$, and $Y_i$, $1\leq i\leq4$.

From Section \ref{sec:10} we see that there are two classes of irreducible representations of $SL(2,\RR)\ltimes\RR^4$ without $L_1$ or $L_2$ fixed vectors. For simplicity, we use unified notation $(\rho_\delta,\,\mathcal{H}_\delta)$, where $\delta$ stands for $(t,s)$ (see Lemma \ref{le:3}) or $s$ (see Lemma \ref{le:2}). In the latter case, $\mathcal{H}_s=\mathbb{H}_s$.
\begin{lemma}\label{le:4}
For any irreducible component $(\rho_\delta,\,\mathcal{H}_\delta)$ of $SL(2,\RR)\ltimes\RR^4$ without $L_1$ or $L_2$ fixed vectors, if $g\in(\mathcal{H}_{\delta})_{G''}^\infty$ and the cohomological equation $\mathcal{L}_{V}f=Y_2g$ has a solution $f\in \mathcal{H}_\delta$. Then
  \begin{align}\label{for:12}
  \norm{f}_{G'',r}\leq C_{r}\norm{g}_{G'',r+4},\qquad \forall\,r\geq 0.
\end{align}
Furthermore, if $g\in(\mathcal{H}_\delta)_{G''}^\infty$, then $f\in (\mathcal{H}_\delta)^{\infty}_{G''}$.
\end{lemma}
\begin{proof} Let $H$ be the subgroup generated by
$X$, $V$ and $Y_i$, $1\leq i\leq2$. Then $H$ is isomorphic to $SL(2,\RR)\ltimes\RR^2$.

If $\rho_\delta=\rho_{t,s}$, where $\rho_{t,s}$ is defined in Lemma \ref{le:3}, then $\pi_\delta|_H=\rho_t$, where $\rho_t$ is as defined in Lemma \ref{le:1}. We use the Fourier model, see \eqref{for:8}. Note that $Y_3=sY_1$ and $Y_4=sY_2$. If $\abs{s}\leq1$, then estimates \eqref{for:12} follows immediately from Lemma \ref{le:8}. If $\abs{s}>1$, by the change of variable $(\omega,\,Y)=(sx,\,s^{-1}y)$ we get $\mathcal{H}_{t,s}=L^2(\RR^2,\mu)$, where $d\mu=d\omega dY$. Computing derived representations, we get
\begin{gather*}
V=-\omega Y\sqrt{-1}, \quad Y_1=s^{-1}\partial_Y, \quad Y_2=s^{-1}\omega\sqrt{-1},\notag\\
   Y_3=\partial_Y, \quad Y_4=\omega\sqrt{-1}.
 \end{gather*}
Then $f(\omega,Y)=\frac{g(\omega,Y)\cdot s^{-1}\omega\sqrt{-1}}{e^{-\omega Y\sqrt{-1}}-1}$. Compare the above vector fields with \eqref{for:139}, then similar to the proof of Lemma \ref{le:8}, we also get \eqref{for:12}.

If $\rho_\delta=\rho_s$, where $\rho_s$ is defined in Lemma \ref{le:2}, we use the Fourier model, see \eqref{for:118}. For any $h\in \mathbb{H}_s$ we can write
\begin{align*}
  h(x,y,z)=\frac{1}{\sqrt{2\pi}}\int \hat{h}(x,y,\xi)e^{-\xi z i}d\xi
\end{align*}
where $\hat{h}(x,y,\xi)=\frac{1}{\sqrt{2\pi}}\int_{\RR}h(x,y,z)e^{iz \xi}dz$. Furthermore, we have
\begin{align*}
  Zh(x,y,z)=\frac{1}{\sqrt{2\pi}}\int Z\big(\hat{h}(x,y,\xi)\big)e^{\xi z i}d\xi,
\end{align*}
for any $Z\in\text{Lie}(G')$. This shows that $\mathbb{H}_{s}$ is a direct integral of the representation of $G'$. Then by discussion in Section \ref{sec:3} and
Lemma \ref{le:8}, we get
\begin{align}\label{for:144}
  \norm{f}_{G',s}\leq C_{s}\norm{g}_{G',s+3},\qquad \forall \, s\geq 0.
\end{align}
We note that $\mathbb{H}_{s}\mid_{SL(2,\RR)}$ is outside a fixed neighborhood the trivial representation in the sense of Fell topology  by Remark \ref{re:2}. Then it follows Proposition \ref{le:6} that
\begin{align}\label{for:145}
  \norm{Y_4^mf}\leq C_{m}\norm{g}_{m+3}\qquad \forall \,m\in\mathbb{N}.
\end{align}
Next, we will show how to obtain estimates along $Y_3$. We assume notations in proof of Lemma \ref{le:9}.  If $\abs{sz}\leq 1$, by noting that
\begin{align*}
\mathcal{L}_{V}\big(f(x,y,z)\cdot s^mx^{-m}\big)&=Y_2g(x,y,z)\cdot s^mx^{-m}\qquad \forall\,m\in\NN,
\end{align*}
and $\sqrt{-1}sx^{-1}=Y_3+zsY_2$, for any $m\in\NN$ we have:
\begin{align*}
&\int_{\abs{z}\leq \abs{s}^{-1}}|f(x,y,z)\cdot s^mx^{-m}|^2dxdydz\\
&\overset{\text{(1)}}{\leq} C\sum_{i=0}^2\int_{\abs{z}\leq \abs{s}^{-1}}|Y_1^iY_2g(x,y,z)\cdot s^mx^{-m}|^2dxdydz\\
&=C\sum_{i=0}^2\int_{\abs{z}\leq \abs{s}^{-1}}|(Y_3+zsY_2)^mY_1^iY_2g(x,y,z)|^2dxdydz\\
&\leq C_m\norm{g}^2_{m+3}.
\end{align*}
Here $(1)$ follows from Lemma \ref{le:9}.

Hence for any $m\in\NN$ we have
\begin{align}\label{for:142}
&\int_{\abs{z}\leq \abs{s}^{-1}}|Y_3^mf(x,y,z)|^2dxdydz\notag\\
&=\int_{\abs{z}\leq \abs{s}^{-1}}|(\sqrt{-1}sx^{-1}-szY_2)^mf(x,y,z)|^2dxdydz\notag\\
&\leq C_m\int_{\abs{z}\leq \abs{s}^{-1}}|f(x,y,z)\cdot s^mx^{-m}|^2dxdydz\notag\\
&+C_m\int_{\abs{z}\leq \abs{s}^{-1}}|Y_2^mf(x,y,z)|^2dxdydz\notag\\
&\leq C_m\norm{g}^2_{m+3}.
\end{align}
If $\abs{sz}>1$, by the change of variable $(\omega,\,Y)=(xzs,\,y(zs)^{-1})$, we define a unitary isomorphism $\mathcal{F}:L^2(\RR^3,dxdydz)\rightarrow L^2(\RR^3,d\omega dYdz)$ as follows:
$(\mathcal{F}h)(\omega,Y,z)=h(\omega (sz)^{-1},Ysz,z)e^{-\sqrt{-1}\omega^{-1}Ys^2z}$. Computing derived representations for the new model, we get
\begin{gather}\label{for:140}
 V=-\omega Y\sqrt{-1},\qquad Y_3=-\partial_Y,\qquad Y_4=-\omega\sqrt{-1} \notag\\
 Y_2=\omega(zs)^{-1}\sqrt{-1}
 \end{gather}
 Then we have
\begin{gather*}
f(\omega,Y,z)=\frac{g(\omega,Y,z)\cdot \omega(sz)^{-1}\sqrt{-1}}{e^{-\omega Y\sqrt{-1}}-1}
\end{gather*}
Compare \eqref{for:140} with \eqref{for:118}, it follows from Lemma \ref{le:8} that for any $z\in\RR$
\begin{align*}
  &\int|Y_3^mf(\omega,Y,z)|^2d\omega dY\leq C_m\sum_{i=0}^m \sum_{j=0}^1\int|Y_3^iY_2^{m-i+1+j}g(\omega,Y,z)\cdot (sz)^{-1}|^2d\omega dY
\end{align*}
 Hence for any $m\in\NN$ we have
\begin{align}\label{for:141}
&\int_{\abs{z}\geq\abs{s}^{-1}}|Y_3^mf(\omega,Y,z)|^2d\omega dYdz\leq C_m\norm{g}^2_{m+3}.
\end{align}
It follow from \eqref{for:142} and \eqref{for:141} that
\begin{align}\label{for:143}
&\norm{Y_3^mf(x,y,z)}\leq C_m\norm{g}_{m+3}.
\end{align}
Then the estimates follows from \eqref{for:144}, \eqref{for:145}, \eqref{for:143} and Theorem \ref{th:4}.

\end{proof}
\subsection{Global coboundary for the unipotent map in $G=SL(2,\RR)\ltimes\RR^2$ or $SL(2,\RR)\ltimes\RR^4$}\label{sec:7}
Let $(\pi,\mathcal{H})$ be a unitary representation of $SL(2,\RR)\ltimes\RR^2$ without non-trivial $\RR^2$-invariant vectors; or a unitary representation of $SL(2,\RR)\ltimes\RR^4$ without non-trivial $L_1$ or $L_2$-invariant vectors (see Section \ref{sec:10}).

We now discuss how to obtain a global solution from the solution which exists in each irreducible component of $\mathcal{H}$. By general arguments in Section \ref{sec:3} there is a direct decomposition of
$\mathcal{H}=\int_Z \mathcal{H}_zd\mu(z)$ of irreducible unitary representations of $G$ for some measure space $(Z,\mu)$. If $\pi$ has no non-trivial $\RR^2$-invariant vectors or no non-trivial $L_1$ or $L_2$-invariant vectors, then for almost
all $z\in Z$, $\pi_z$ has no non-trivial $\RR^2$-invariant vectors or no non-trivial $L_1$ or $L_2$-invariant vectors. Hence we can apply
Theorem \ref{po:2} and Lemma \ref{le:4} to prove the following:
\begin{corollary}\label{cor:3}
Let $(\pi,\mathcal{H})$ be a unitary representation of $SL(2,\RR)\ltimes\RR^2$ without non-trivial $\RR^2$-invariant vectors. If $g\in \mathcal{H}^\infty$ and the cohomological equation $\mathcal{L}_{V}f=g$  has a solution $f\in \mathcal{H}$, then  $f\in \mathcal{H}^\infty$ and satisfies
\begin{align*}
  \norm{f}_{t}\leq C_{t}\norm{g}_{2t+6},\qquad \forall\,t\geq 0.
\end{align*}
\end{corollary}
\begin{corollary}\label{cor:2}
Let $(\pi,\,\mathcal{H})$ be a unitary representation $SL(2,\RR)\ltimes\RR^4$ without $L_1$ or $L_2$ fixed vectors. If $g\in\mathcal{H}_{G''}^\infty$ and the cohomological equation $\mathcal{L}_{V}f=Y_2g$ has a solution $f\in \mathcal{H}$. Then
  \begin{align*}
  \norm{f}_{G'',r}\leq C_{r}\norm{g}_{G'',r+4},\qquad \forall\,r\geq 0.
\end{align*}
\end{corollary}

\subsubsection{Coboundary for the unipotent map in $G=SL(2,\RR)\times\RR$}\label{sec:14} Before proceeding further with the proof of Theorem \ref{th:6}, we prove
certain technical results which are very useful for the discussion.
\begin{lemma}\label{le:15}
Suppose $G=H\times \RR$ and the Lie algebra of $H$ is $\mathfrak{sl}(2,\RR)$, and suppose $(\pi,\mathcal{H})$ is a unitary representation of $G$ such that there is a spectral gap of $u_0$ for $(\pi\mid_{H},\,\mathcal{H})$. Let $\mathfrak{u}=\begin{pmatrix}
  0 & 1 \\
  0 & 0
\end{pmatrix}\in \mathfrak{sl}(2,\RR)$ and $\chi=1\in\text{Lie}(\RR)$.
\begin{enumerate}
  \item \label{for:119} Suppose $p\in \mathcal{H}^\infty$. If the cohomological equation $\mathcal{L}_{\mathfrak{u}}h=p$ has a solution $h\in \mathcal{H}_{H}^\infty$, then $h\in \mathcal{H}^\infty$ with estimates $\norm{h}_s\leq C_{s,u_0}\norm{p}_{2s+4}$ for any $s\geq0$.

      \medskip
  \item \label{for:120}Suppose there is no non-trivial $\exp(\chi)$-invariant vectors for $\pi$; and  suppose $p,\,\psi\in \mathcal{H}^\infty$. Then the cohomological equation $\mathcal{L}_{\mathfrak{u}}\psi=\mathcal{L}_{\chi} p$ has a common solution $\rho\in \mathcal{H}^\infty$ with estimates
      \begin{align*}
        \norm{\rho}_s\leq C_s\max\{\norm{p}_{2s+4},\,\norm{\psi}_{2s+4}\}
      \end{align*}
      for any $s\geq0$.
\end{enumerate}

\end{lemma}
\begin{proof}
Irreducible unitary representations of $G$ are of the form $(\pi^\delta_\nu\otimes \zeta_v,\,\mathcal{H}^\delta_\nu)$, where $(\pi^\delta_\nu,\,\mathcal{H}^\delta_\nu)$, $\delta=0,\,\pm$ is an irreducible unitary
representation of $H$ described in Section \ref{sect:SL2R_reps}; and $\zeta_v$, $v\in\RR$ is an irreducible unitary representation of $\RR$ defined as follows: $\zeta_v(x)=e^{\sqrt{-1}xv}$, $x\in\RR$. The discussion in Section \ref{sec:3} allows us to reduce our analysis of the cohomological and cocycle equations to each irreducible component $\pi^\delta_\nu\otimes \zeta_v$ that appears in $\pi$.

\noindent\textbf{\emph{Proof of \eqref{for:119}}} In $\pi^\delta_\nu\otimes \zeta_v$, write $h_{\nu,\delta,v}=h_{\nu,\delta}$ and $p_{\nu,\delta,v}=p_{\nu,\delta}$ where $h_{\nu,\delta}$ and $p_{\nu,\delta}$ are in $(\mathcal{H}^\delta_\nu)_H^\infty$. Note that
\begin{align*}
    \chi (k_{\nu,\delta})=v\sqrt{-1}k_{\nu,\delta},\qquad\text{ where }k=h\text{ or }p.
\end{align*}
This shows that $h_{\nu,\delta}\in (\mathcal{H}^\delta_\nu)^\infty$. For any $n\in\NN$ we have
\begin{align*}
  \mathcal{L}_\mathfrak{u}\bigl(\chi^n (h_{\nu,\delta})\bigl)&=\mathcal{L}_\mathfrak{u}(v^n\sqrt{-1}^nh_{\nu,\delta})=v^n\sqrt{-1}^n\mathcal{L}_\mathfrak{u}(h_{\nu,\delta})\\
  &=v^n\sqrt{-1}^np_{\nu,\delta}=\chi^n (p_{\nu,\delta}).
\end{align*}
Since $h_{\nu,\delta}\in (\mathcal{H}^\delta_\nu)^\infty$, $\mathcal{D}\bigl(\chi^n (p_{\nu,\delta})\bigl)=0$ for any $n\in\NN$ and any $\mathcal{D}\in\mathcal{E}_{\mathfrak{u}}(\mathcal{H}^\delta_\nu)$. Theorem \ref{th:2} and Remark \ref{re:5} show that
\begin{align}\label{for:13}
    \norm{\chi^n (h_{\nu,\delta})}&\leq C\norm{\chi^n (p_{\nu,\delta})}_2\leq C\norm{p_{\nu,\delta}}_{n+2},\\
    \norm{h_{\nu,\delta}}_{H,s}&\leq C_{s,u_0}\norm{p_{\nu,\delta}\otimes 1}_{H,2s+2}\notag
\end{align}
for any $s\geq0$ and $n\in \NN$. Then it follows from Theorem \ref{th:4} that
\begin{align*}
    \norm{h_{\nu,\delta}}_{s}&\leq C_{s,u_0}\norm{p_{\nu,\delta}}_{2s+4},\qquad \forall\,s\geq0,
\end{align*}
which gives the estimates of $f$ immediately.

\medskip
\noindent\textbf{\emph{Proof of \eqref{for:120}}} By assumption we only need to consider $\pi^\delta_\nu\otimes \zeta_v$, $v\neq 2n\pi$, $n\in\ZZ$. Write $\psi_{\nu,\delta,v}=\psi_{\nu,\delta}$ and $p_{\nu,\delta,v}=p_{\nu,\delta}$, where $\psi_{\nu,\delta}$ and $p_{\nu,\delta}$ are in $(\mathcal{H}^\delta_\nu)^\infty$. Then we have
\begin{align*}
  \mathcal{L}_{\mathfrak{u}}\psi_{\nu,\delta}&=\mathcal{L}_{\chi} p_{\nu,\delta}=p_{\nu,\delta}e^{\sqrt{-1}v}-p_{\nu,\delta}.
\end{align*}
which immediately gives
\begin{align*}
  \mathcal{L}_{\mathfrak{u}}\big(\frac{\psi_{\nu,\delta}}{e^{\sqrt{-1}v}-1}\big)&=p_{\nu,\delta},
\end{align*}
and
\begin{align*}
  \mathcal{L}_{\chi}\big(\frac{\psi_{\nu,\delta}}{e^{\sqrt{-1}v}-1}\big)&=\psi_{\nu,\delta}.
\end{align*}
Let $\rho_{\nu,\delta}=\frac{\psi_{\nu,\delta}}{e^{\sqrt{-1}v}-1}$. Results in previous part shows that
\begin{align*}
    \norm{\rho_{\nu,\delta}}_{t}&\leq C_{s,u_0}\max\{\norm{p_{\nu,\delta}}_{2s+4},\,\norm{\psi_{\nu,\delta}}_{2s+4}\},\qquad \forall\,s\geq 0
\end{align*}
which gives the existence and estimates of $\rho$ immediately.
\end{proof}
\begin{remark}\label{re:4}
If $\mathfrak{u}$ and $\chi$ imbed in a Lie algebra isomorphic to $sl(2,\RR)\ltimes\RR^2$, then the cohomological equation $\mathcal{L}_{\mathfrak{u}}\psi=\mathcal{L}_{\chi} p$ probably fail to have a common solution. For example, we consider the irreducible component $(\rho_t,\,\mathcal{H}_t)$ at $t=0$. If $g=h_1(x)h_2(y)$ where $h_i\in C_0^\infty[-1,1]$, $i=1,\,2$ and satisfies: $h_1=1$ on $[-\frac{1}{2},\frac{1}{2}]$ and $h_2(0)=0$. Let
\begin{align*}
f(x,y)&=\frac{h_1(x)(e^{x\sqrt{-1}}-1)h_2(y)}{e^{-xy\sqrt{-1}}-1}\\
&=h_1(x)\cdot\frac{h_2(y)}{y}\cdot\frac{(e^{x\sqrt{-1}}-1)}{x}\cdot\frac{xy}{e^{-xy\sqrt{-1}}-1}.
\end{align*}
It is easy to check that $f\in \mathcal{H}_0$ is a solution to the cohomological equation $\mathcal{L}_Vf=\mathcal{L}_{Y_2}g$. If the cohomological equation has a common solution $h\in \mathcal{H}_0$, then $\mathcal{L}_Vh=g$. Hence we have
\begin{align*}
h(x,y)&=\frac{h_1(x)h_2(y)}{e^{-xy\sqrt{-1}}-1}=\frac{h_2(y)}{y}\cdot\frac{h_1(x)}{x}\cdot\frac{xy}{e^{-xy\sqrt{-1}}-1}.
\end{align*}
 This implies that $h_1(x)\cdot x^{-1}\in L^2(\RR)$, which is a contradiction.
\end{remark}

\section{Proof of Theorem \ref{th:6} and \ref{th:101}}\subsection{Proof of \eqref{for:35} of Theorem \ref{th:6}}
To apply Proposition \ref{le:6} or Theorem \ref{th:2}, it suffices to prove that $\pi\mid_{G_{\mathfrak{v}}}$ has a spectral gap (see Definition \ref{de:3}).
By Howe-Moore, $\pi\mid_{G'_{\mathfrak{v}}}$ has no non-trivial $\RR^2$-invariant vectors where $G'_{\mathfrak{v}}$ is a subgroup in $\GG$ containing $G_{\mathfrak{v}}$ and isomorphic to $SL(2,\RR)\ltimes\RR^2$. Remark \ref{re:2} shows that $\pi\mid_{G_{\mathfrak{v}}}$ is outside a fixed neighborhood of trivial representation in the Fell topology, which proves the claim. Then the result follows immediately from Proposition \ref{le:6} or Theorem \ref{th:2}.

\subsection{Proof of \eqref{for:33} of Theorem \ref{th:6}}\label{sec:8}

 Since $\Phi$ is a reduced irreducible root system, for any $\psi\in \Phi$, $\psi\neq \pm\phi$, either $\phi\pm\psi\notin\Phi$, or one of $\phi\pm\psi$
 belongs to $\Phi$. In the first case, we know that $\mathfrak{u}_\phi$ and $\mathfrak{u}_\psi$ imbed in a Lie algebra isomorphic to $sl(2,\RR)\times\RR$. Then \eqref{for:13} of Lemma \ref{le:15} that for any $t\geq0$,
\begin{align}\label{for:1138}
 \norm{f}_{U_\psi,t}\leq C_{t}\norm{g}_{t+2}.
\end{align}
 For the second case,  consider $\Psi=\{i\phi+j\psi\in\Phi\mid\,i,\,j\in\ZZ\}$. Then $\Psi$ is a reduced irreducible root system of rank $2$. Denote by $G$
the closed subgroup
of $\GG$ with its Lie algebra generated by the root sub-subgroups $\mathfrak{u}_\varphi$, $\varphi\in\Psi$. Note that each $\mathfrak{u}_\varphi$, $\varphi\in\Psi$ is one-dimensional. Then $G$ is of type $A_2$. This shows that $\mathfrak{u}_\phi$ and $\mathfrak{u}_\psi$ imbed in a Lie algebra isomorphic to $sl(2,\RR)\ltimes\RR^2$.
 It follows from Corollary \ref{cor:3} that
 \begin{align}\label{for:53}
  \norm{f}_{G,t}&\leq C_{t}\norm{g}_{G,2t+6}.
 \end{align}
For the Cartan subalgebra $\mathcal{C}$, we can find a basis $\{\mathcal{C}_i\}$ of $\mathcal{C}\cap\mathfrak{G}^1$, such that $\mathcal{C}_1\in [\mathfrak{u}_\phi, \mathfrak{u}_{-\phi}]\subset\text{Lie}(G)$ and $[\mathcal{C}_i,\mathfrak{u}_\phi]=0$, $i\neq1$. Then
\eqref{for:13} of Lemma \ref{le:15} that for any $m\in\NN$,
\begin{align}\label{for:38}
 \norm{\mathcal{C}_i^mf}\leq C_{m}\norm{g}_{m+2},\qquad i\neq1.
\end{align}
Then the estimates follow immediately from \eqref{for:1138}, \eqref{for:53} and \eqref{for:38} and Theorem \ref{th:4}.

\subsection{Proof of Theorem \ref{th:101} when $\GG=SL(n,\RR)$, $n\geq4$}\label{sec:2}
Firstly, we list a fact which will be used in the next proofs.
\begin{fact}\label{fact:3}
(Corollary 6.11 and Lemma 6.13 of \cite{W1}) Let $(\pi,\mathcal{H})$ be a unitary representation of $SL(2,\RR)\ltimes\RR^2$ without non-trivial $\RR^2$-invariant vectors. If $g\in \mathcal{H}^\infty$ and the cohomological equation $Vf=g$  has a solution $f\in \mathcal{H}$, then  $f\in \mathcal{H}^\infty$ and satisfies
\begin{align*}
  \norm{f}_{t}\leq C_{t}\norm{g}_{t+5},\qquad \forall\,t\geq 0.
\end{align*}
We also have: suppose $(\pi,\mathcal{H})$ is a unitary representation of $SL(2,\RR)\ltimes\RR^2$ such that $\pi\mid_{SL(2,\RR)}$ only contains the principal and
discrete series. If $g\in \mathcal{H}^\infty$ and the cohomological equation $Vf=g$ has a solution $f\in \mathcal{H}_{SL(2,\RR)}^\infty$, then $f\in \mathcal{H}^\infty$ with estimates
\begin{align*}
  \norm{f}_{t}\leq C_{t}\norm{g}_{t+2},\qquad \forall\,t\geq 0.
\end{align*}
\end{fact}
Fix $3\leq m\neq l\leq n$. Set $F=\mathfrak{u}_{1,m}\mathfrak{u}_{l,2}f$. Note that
\begin{align*}
  \mathcal{L}_{\mathfrak{u}_{1,2}}F=\mathfrak{u}_{1,m}(\mathfrak{u}_{l,2}g).
\end{align*}
For $k\geq 3$, let $G'_k$ be the closed subgroup in $SL(n,\RR)$ generated by $U_{1,2}$, $X_{1,2}$, $U_{1,k}$, $U_{2,k}$, $U_{1,m}$ and $U_{2,m}$. Then for any $k\neq m$, the subgroup generated by $G'_k$ and $U_{2,1}$ are isomorphic to $SL(2,\RR)\ltimes\RR^4$. Thanks to Howe-Moore, it follows from Corollary \ref{cor:2} that
\begin{align}\label{for:29}
  \norm{F}_{G'_k,t}&\leq C_{t}\norm{\mathfrak{u}_{l,2}g}_{G'_k,t+4}\leq C_{t}\norm{g}_{G'_k,t+5},\qquad t\geq0,
\end{align}
for all $3\leq k\leq n$.

We also note that
\begin{align*}
  \mathcal{L}_{\mathfrak{u}_{1,2}}F=\mathfrak{u}_{l,2}(\mathfrak{u}_{1,m}g).
\end{align*}
Let $H'_k$ in $SL(n,\RR)$ be the subgroup generated by $U_{1,2}$, $X_{1,2}$, $U_{l,1}$, $U_{l,2}$, $U_{k,1}$, and $U_{k,2}$. By similar arguments, we have
\begin{align}\label{for:30}
  \norm{F}_{H'_k,t}&\leq C_{t}\norm{\mathfrak{u}_{1,m}g}_{H'_k,t+4}\leq C_{t}\norm{g}_{H'_k,t+5},\qquad t\geq0,
\end{align}
for all $3\leq k\leq n$.

Let $Z$ stand for $U_{i,j}$ or $X_{k,l}$, where $i,\,j\geq 3$ and $3\leq k\neq l\leq n$. Then it follows from \eqref{for:13} of Lemma \ref{le:15} that for any $m\in\NN$,
\begin{align}\label{for:146}
 \norm{Z^mF}&\leq C_{m}\norm{\mathfrak{u}_{1,m}\mathfrak{u}_{l,2}g}_{m+2}\leq C_{m}\norm{g}_{m+4}.
\end{align}
Denote by $\mathcal{A}$ the subspace of of $sl(n,\RR)$ spanned by the linear algebra of $G_k'$, $H_k'$, $3\leq k\leq n$, together with $U_{i,j}$, $X_{k,l}$, where $i,\,j\geq 3$ and $3\leq k\neq l\leq n$. Then $\mathcal{A}$ is of codimension $1$; and the linear span of $\mathcal{A}$ and $\mathfrak{u}_{2,1}$ is $sl(n,\RR)$.
Then it follows from \eqref{for:29}, \eqref{for:30} and \eqref{for:146} and Theorem \ref{th:4} that for subgroup $S$ of $SL(n,\RR)$ with $\text{Lie}(S)\subseteq \mathcal{A}$ we have
\begin{align}\label{for:31}
 \norm{F}_{S,s}&\leq C_{s}\norm{g}_{s+5},\qquad \forall\,s\geq 0.
\end{align}
Note that
\begin{align*}
  \mathfrak{u}_{1,m}(\mathfrak{u}_{l,2}f)=F.
\end{align*}
Let $P_m$ denote the subgroup generated by $\mathfrak{u}_{1,m}$, $\mathfrak{u}_{m,1}$, $X_{1,m}$. It is clear that $P_m$ is isomorphic to $SL(2,\RR)$.
Let $Z$ stand for $\mathfrak{u}_{l,2}$, $\mathfrak{u}_{2,l}$, $X_{2,l}$, $\mathfrak{u}_{2,k}$ and $\mathfrak{u}_{l,k}$, $3\leq k\neq m$. Let $P_k'$ be the subgroup generated by $P_m$ and $\{\exp(tZ)\}_{t\in\RR}$. Then $P_k'$ are isomorphic to $SL(2,\RR)\times\RR$. Since $\text{Lie}(P_k')\subseteq \mathcal{A}$, it follows from Fact \ref{fact:3}, Remark \ref{re:2} and \eqref{for:31} that
\begin{align}\label{for:42}
 \norm{\mathfrak{u}_{l,2}f}_{P_k',s}&\leq C_{s}\norm{F}_{P_k', s+2}\leq C_{s}\norm{g}_{s+7},\qquad \forall\,s\geq 0.
\end{align}
for $3\leq k\neq m$.

Note that the subgroup generated by $\mathfrak{u}_{l,2}$, $\mathfrak{u}_{2,l}$, $X_{2,l}$, $\mathfrak{u}_{2,k}$ and $\mathfrak{u}_{l,k}$, $3\leq k\neq m$, which are denoted by $L_k'$, are isomorphic to $SL(2,\RR)\ltimes\RR^2$, then above estimates, together with Howe-Moore and Fact \ref{fact:3} imply that
\begin{align}\label{for:43}
 \norm{f}_{L_k',s}&\leq C_{s}\norm{\mathfrak{u}_{l,2}f}_{P_k', s+5}\leq C_{s}\norm{g}_{s+12},\qquad \forall\,s\geq 0.
\end{align}
Hence we get
\begin{align}\label{for:32}
   \max\{\norm{\mathfrak{u}_{2,k}^sf},\,\norm{X_{2,l}^sf}\}&\leq C_{s}\norm{g}_{s+12},\qquad \forall\,s\geq 0,
\end{align}
for any $3\leq k\neq m$. Note that we also have
\begin{align*}
  \mathfrak{u}_{l,2}(\mathfrak{u}_{1,m}f)=F.
\end{align*}
Then similar arguments shows that
\begin{align}\label{for:37}
   \max\{\norm{\mathfrak{u}_{k,2}^sf},\,\norm{X_{1,m}^sf}\}&\leq C_{s}\norm{g}_{s+12},\qquad \forall\,s\geq 0,
\end{align}
for any $3\leq k\neq l$.

Since $m$ and $l$ are chosen arbitrarily. Then the estimates follow immediately from \eqref{for:1138}, \eqref{for:32}, \eqref{for:37} and Theorem \ref{th:4}.

\subsection{Proof of Theorem \ref{th:101} for other cases} Since $\Phi$ is a reduced irreducible root system, there exists
a root $\psi\in \Phi$ such that $\psi-\phi$ belongs to $\Phi$. Set $\Psi=\{i\phi+j\psi\in\Phi\mid\,i,\,j\in\ZZ\}$. For any $\beta\in \Phi$, but $\beta\notin\Psi$, either
$\phi\pm\beta\notin\Phi$, or one of $\phi\pm\beta$ belongs to $\Phi$.

For the first case,  we see that $\mathfrak{u}_\phi$ and $\mathfrak{u}_\beta$ imbed in a Lie algebra isomorphic to $sl(2,\RR)\times\RR$. Then \eqref{for:13} of Lemma \ref{le:15} that for any $t\geq0$,
\begin{align}\label{for:54}
 \norm{f}_{U_\beta,t}\leq C_{t}\norm{g}_{t+2}.
\end{align}
For the second case,  we consider $\Psi_1=\{i\phi+j\psi+\ell\beta\in\Phi\mid\,i,\,j,\,\ell\in\ZZ\}$. Then $\Psi_1$ is a reduced irreducible root system of rank $3$. Denote by $G$ the closed subgroup
of $\GG$ with its Lie algebra generated by the root sub-subgroups $\mathfrak{u}_\varphi$, $\varphi\in\Psi_1$. Note that each $\mathfrak{u}_\varphi$, $\varphi\in\Psi$ is one-dimensional. Then $G$ is of type $A_3$. Results in Section \ref{sec:2} show that
\begin{align}\label{for:47}
 \norm{f}_{U_\beta,s}&\leq C_{s}\norm{g}_{s+12},\qquad \forall\,s\geq 0
\end{align}
and
\begin{align*}
 \norm{\mathcal{C}_1^mf}&\leq C_{m}\norm{g}_{m+12},\qquad \forall\,m\geq 0
\end{align*}
where $\mathcal{C}_1\in[\mathfrak{u}_\phi, \mathfrak{u}_{-\phi}]\cap\mathfrak{G}^1$.

For the Cartan subalgebra $\mathcal{C}$, we can find a basis $\{\mathcal{C}_i\}$ of $\mathcal{C}\cap\mathfrak{G}^1$, such that $\mathcal{C}_1$ is as described above and $[\mathcal{C}_i,\mathfrak{u}_\phi]=0$, $i\neq1$. Then
\eqref{for:13} of Lemma \ref{le:15} that for any $m\in\NN$,
\begin{align*}
 \norm{\mathcal{C}_i^mf}\leq C_{m}\norm{g}_{m+2},\qquad i\neq1.
\end{align*}
Then estimates follows immediately from above estimates,  \eqref{for:54}, \eqref{for:47} and Theorem \ref{th:4}.

\section{Proof of Theorem \ref{th:8}}\label{sec:22}
Note that $\mathfrak{u}_\phi$ and $\mathfrak{u}_\psi$ can imbed in $\mathfrak{sl}(2,\RR)\times\RR$. \eqref{for:120} of Lemma \ref{le:15} implies that there is a common solution $h\in\mathcal{H}$ to the cocycle equation $\mathcal{L}_\mathfrak{u}f=\mathcal{L}_\mathfrak{v}g$. Then it follows from  \eqref{for:33} of Theorem \ref{th:6} that $h\in \mathcal{H}^\infty$. Since $\mathcal{L}_\mathfrak{u}h=g$, it follows from Theorem \ref{th:101} that
\begin{align}\label{for:48}
 \norm{\nu^nh}\leq C_n\norm{g}_{n+12}
\end{align}
for any $\nu\in\mathfrak{u}_\mu\bigcap\mathfrak{G}^1$, where $\mu\neq -\phi$; and
\begin{align}\label{for:49}
 \norm{Y^nh}\leq C_n\norm{g}_{n+3}
\end{align}
for any $Y\in \mathcal{C}\bigcap\mathfrak{G}^1$.

On the other hand, we also have $\mathcal{L}_\mathfrak{v}h=f$. Then by Theorem \ref{th:101} again we have
\begin{align*}
 \norm{\nu^nf}\leq C_n\norm{f}_{n+12}
\end{align*}
for any $\nu\in\mathfrak{u}_\mu\bigcap\mathfrak{G}^1$, where $\mu\neq -\psi$; especially, we have
\begin{align}\label{for:50}
 \norm{\nu^nh}\leq C_n\norm{f}_{n+12}
\end{align}
for any $\nu\in\mathfrak{u}_\mu\bigcap\mathfrak{G}^1$, where $\mu=-\phi$.

Then the estimates follow immediately from \eqref{for:48}, \eqref{for:49}, \eqref{for:50} and Theorem \ref{th:4}.

\section{Proof of Theorem \ref{th:7}}\label{sec:21}
\subsection{Unitary representations of $SL(n, \R)$, $n \geq 3$}\label{sect:SLnR-Rep}
The detailed study of the representation $\text{Ind}_P^{SL(n,\RR)}( \lambda_{t}^{\pm})$ are given in \cite{W1}. In this part, we just list the results.
Let $g_{i,j}^t=\exp(t\mathfrak{u}_{i,j})$ and $h_i^t=\exp(tX_i)$, $t\in\RR$. The realization of the representation $\text{Ind}_P^{SL(n,\RR)}( \lambda_{t}^{\pm})$
on $L^2(\RR^{n-1},\,dx)$ can be formulated as follows:
\begin{align*}
&\text{Ind}_P^{SL(n,\RR)}( \lambda_{t}^{\pm})(h_{i}^s)f(x_1,\cdots,x_{n-1})\\
&=\left\{\begin{aligned} &e^{sn/2}e^{ts\sqrt{-1}}f(e^{2s}x_1,e^{s}x_{2}\cdots,e^{s}x_{n-1}),&\quad &i=1\\
&f(x_1,\cdots,e^{-s}x_{i-1},e^{s}x_{i}\cdots,x_{n-1}),&\quad& i\geq 2;
\end{aligned}
 \right.
\end{align*}
and has the following expressions
\begin{align*}
&\text{Ind}_P^{SL(n,\RR)}( \lambda_{t}^{\pm})(g_{i,j}^s)f(x_1,\cdots,x_{n-1})\\
    &=\left\{\begin{aligned} &\abs{1-x_{j-1}s}^{-n/2-t\sqrt{-1}}\varepsilon^{\pm}(1-x_{j-1}s)\\
    &\cdot f(\frac{x_1}{1-x_{j-1}s},\cdots,\frac{x_{n-1}}{1-x_{j-1}s}),&\quad &i=1,\,\,j\geq 2,\\
&f(x_1,\cdots,x_{i-1}-sx_{j-1},\cdots,x_{n-1}),&\quad &i\geq 2,\,\,j\neq 1,\\
&f(x_1,\cdots,x_{i-1}-s,\cdots,x_{n-1}),&\quad &i\geq 2,\,\,j= 1.
\end{aligned}
 \right.
\end{align*}
Since the one-parameter subgroups $g_{i,j}^t$ and $h_\ell^t$ generate $SL(n,\RR)$, the actions of these subgroups determine the group action of $SL(n,\RR)$. Computing derived representations, we get
\begin{align}\label{for:9}
X_i=&\left\{\begin{aligned} &(\frac{n}{2}+t\sqrt{-1})+2x_1\partial_{x_{1}}+\sum_{k=2}^{n-1}x_k\partial_{x_{k}},&\qquad \qquad &i=1,\\
&-x_{i-1}\partial_{x_{i-1}}+x_{i}\partial_{x_{i}},&\qquad \qquad &i\geq 2,
\end{aligned}
 \right.
\end{align}
and
\begin{align}\label{for:105}
\mathfrak{u}_{i,j}=&\left\{\begin{aligned} &(\frac{n}{2}+t\sqrt{-1})x_{j-1}+\sum_{k=1}^{n-1}x_{j-1}x_k\partial_{x_{k}},&\qquad \qquad &i=1,\,\,j\geq 2,\\
&-x_{j-1}\partial_{x_{i-1}},&\qquad \qquad &i\geq 2, \,\,j\neq 1,\\
&-\partial_{x_{i-1}},&\qquad \qquad &i\geq 2, \,\,j= 1.
\end{aligned}
 \right.
\end{align}
We are now in a position to proceed with the proof of Theorem \ref{th:7}.
 Noting that the Weyl group is the symmetric group $S_n$ which operates simply transitively on the set of Weyl chambers, we may assume that one element in the set $\{\mathfrak{u}_{i_k,j_k}$, $1\leq k\leq m\}$ is $\mathfrak{u}_{2,1}$. By assumption, the other ones either from the set $\{\mathfrak{u}_{2,j}:j\geq 3\}$ or from the set
$\{\mathfrak{u}_{j,1}:j\geq 3\}$.
\begin{case} The other elements are $\mathfrak{u}_{2,i_k}$, $i_k\geq 3$ and $2\leq k\leq m$.

If we take the Fourier transformation on $x_1$, i.e.,
\begin{align*}
  &\hat{f}(x_1,\cdots,x_{n-1})=\frac{1}{\sqrt{2\pi}}\int_{\RR}f(\xi_1,x_2,\cdots,x_{n-1})e^{-ix_1\xi_1}d\xi_1
\end{align*}
 for any $f\in L^2(\RR^{n-1},dx)$ we get the Fourier model:
\begin{align*}
\mathfrak{u}_{2,j}=&\left\{\begin{aligned}
&-x_{j-1}x_{1}\sqrt{-1},&\qquad \qquad &j\neq 1,\\
&-x_{1}\sqrt{-1},&\qquad \qquad &j= 1.
\end{aligned}
 \right.
\end{align*}
Let $h(x_1,\cdots,x_{n-1})=p(x_1)p(x_2)\cdots p(x_{n-1})$ with $p\in C_0^\infty[-1,1]$ and $p=1$ on $[-\frac{1}{2},\frac{1}{2}]$. Let $f_1=h$ and
\begin{align*}
  f_{k}&=h\cdot \frac{e^{-x_1x_{i_k-1}\sqrt{-1}}-1}{e^{-x_1\sqrt{-1}}-1}\\
  &=h\cdot x_{i_k-1}\cdot\frac{e^{-x_1x_{i_k-1}\sqrt{-1}}-1}{x_1x_{i_k-1}}\cdot\frac{x_1}{e^{-x_1\sqrt{-1}}-1},
\end{align*}
$2\leq k\leq m$. From relations in \eqref{for:9} and \eqref{for:105}, it is easy to check that $f_k$, $1\leq k\leq m$ are smooth vectors in $\text{Ind}_P^{SL(n,\RR)}( \lambda_{t}^{\pm})$. Using relations described above for the Fourier model we have
\begin{align*}
 \mathcal{L}_{\mathfrak{u}_{2,1}}f_{k}&=(e^{-x_1\sqrt{-1}}-1)f_{k}=(e^{-x_1x_{i_k-1}\sqrt{-1}}-1)f_1=\mathcal{L}_{\mathfrak{u}_{2,i_k}}f_1
\end{align*}
and
\begin{align*}
 \mathcal{L}_{\mathfrak{u}_{2,i_\ell}}f_{k}&=(e^{-x_1x_{i_\ell-1}\sqrt{-1}}-1)f_{k}=(e^{-x_1x_{i_k-1}\sqrt{-1}}-1)f_{\ell}
 =\mathcal{L}_{\mathfrak{u}_{2,i_k}}f_{\ell}
\end{align*}
where $i_k,\,i_\ell\geq 3$. Then $f_i$, $1\leq i\leq m$ satisfy all cocycle equations.

If $\mathcal{L}_{\mathfrak{u}_{2,1}}\omega=f_1$ where $\omega\in L^2(\RR^{n-1},dx)$, then $(e^{-x_1\sqrt{-1}}-1)\omega=h$. Then
\begin{align*}
\omega=\frac{p(x_1)}{e^{-x_1\sqrt{-1}}-1}\cdot p(x_2)\cdots p(x_{n-1})
\end{align*}
This implies that $\frac{p(x_1)}{e^{-x_1\sqrt{-1}}-1}\in L^2(\RR)$. Then we get a contradiction. Note that all the other equations $\mathcal{L}_{\mathfrak{u}_{2,i_k}}\omega=f_{k}$, $i_k\geq 3$ are equivalent to the earlier equation $(e^{-x_1\sqrt{-1}}-1)\omega=h$. Thus we proved the claim.
\end{case}
\begin{case}The other elements are $\mathfrak{u}_{i_k,1}$, $i_k\geq 3$ and $2\leq k\leq m$.

If we take the Fourier transformation, i.e.,
\begin{align}\label{for:52}
  &\hat{f}(x_1,\cdots,x_{n-1})\notag\\
  &=\frac{1}{(2\pi)^{\frac{n}{2}}}\int_{\RR}f(\xi_1,\xi_2,\cdots,\xi_{n-1})e^{-ix_1\xi_1}\cdots e^{-ix_{n-1}\xi_{n-1}}
  d\xi_1\cdots d\xi_{n-1}
\end{align}
 for any $f\in L^2(\RR^{n-1},dx)$ we get the Fourier model:
\begin{align}\label{for:51}
\mathfrak{u}_{i,1}=-x_{i-1}\sqrt{-1}, \qquad i\geq2.
\end{align}
We set $i_1=2$. By Theorem 2.5 in \cite{W1}, we can find smooth vectors $f_{j}$, $2\leq j\leq n-1$ and $j\neq i_2$ in $\text{Ind}_P^{SL(n,\RR)}( \lambda_{t}^{\pm})$ satisfying the following:
\begin{align*}
\mathfrak{u}_{i_\ell,1}f_j=\mathfrak{u}_{i_j,1}f_\ell,
\end{align*}
but non of the equations:
\begin{align*}
  \mathfrak{u}_{i_j,1}h=f_j
\end{align*}
has a solution $h\in L^2(\RR^{n-1})$.

By \eqref{for:51}, the above assumption implies that $f_j=f_1\cdot\frac{x_{i_j-1}}{x_1}$,
but $f_j\cdot\frac{1}{x_{i_j-1}}=\frac{f_1}{x_1}\notin L^2(\RR^{n-1})$. Suppose $r(x)\in C_0^\infty([-1,1])$ satisfying $r(x)=1$ if $x\in [-\frac{1}{2},\frac{1}{2}]$. Set $f_1'=f_1\cdot r(x_1)$ and
\begin{align*}
  f_j'&=f_1'\cdot\frac{e^{-x_{i_j-1}\sqrt{-1}}-1}{e^{-x_1\sqrt{-1}}-1}\\
  &=f_1'\cdot\frac{x_{i_j-1}}{x_1}\cdot\frac{e^{-x_{i_j-1}\sqrt{-1}}-1}{x_{i_j-1}}\cdot\frac{x_1}{e^{-x_1\sqrt{-1}}-1}\\
  &=f_j\cdot r(x_1)\cdot\frac{x_{i_j-1}}{x_1}\cdot\frac{e^{-x_{i_j-1}\sqrt{-1}}-1}{x_{i_j-1}}\cdot\frac{x_1}{e^{-x_1\sqrt{-1}}-1}
   \end{align*}
It is easy to check that the following fact: for any smooth vector $f$ in the Fourier model, $f\cdot r(x_1)$ is also smooth. This fact shows that $f_j'$, $1\leq j\leq n-1$ and $j\neq i_2$ are smooth vectors in $\text{Ind}_P^{SL(n,\RR)}( \lambda_{t}^{\pm})$. We also have
\begin{align*}
\mathcal{L}_{\mathfrak{u}_{i_\ell,1}}f'_j=\mathcal{L}_{\mathfrak{u}_{i_j,1}}f'_\ell,
\end{align*}
and
\begin{align*}
  &f'_j\cdot\frac{1}{e^{-x_{i_j-1}\sqrt{-1}}-1}= f_1'\cdot\frac{1}{x_1}\cdot\frac{x_1}{e^{-x_1\sqrt{-1}}-1}\notin L^2(\RR^{n-1}).
 \end{align*}
Hence we can prove the claim.

\end{case}
\section{Proof of Theorem \ref{th:10}}
The proof is standard and similar arguments appeared in \cite{Spatzier1}, \cite{Mieczkowski} and \cite{Ramirez}. Let $\beta$ be a cocycle over the $V$-action on $\GG/\Gamma$. Restricted to the $U$-action on $\GG/\Gamma$, $\beta$ is also a cocycle.

It follows from Theorem \ref{th:8} that there is a smooth transfer function $p$ that satisfies
\begin{align*}
    \beta(u,x)=p(u\cdot x)+c(u)-p(x)
\end{align*}
for any $u\in U$ and $x\in \GG/\Gamma$, where $c(u)=\int_{\GG/\Gamma}\beta(u,x)dx$ is a constant cocycle. For any $v\in V$, let
\begin{align*}
    \beta^*(v,x)=\beta(v,x)-p(v\cdot x)+p(x).
\end{align*}
Using the definition of cocycle, we see that $\beta^*$ is also a cocycle over $V$-action. Then
\begin{align*}
 \beta^*(v,x)&= \beta^*(uv,x)-\beta^*(u,v\cdot x)=\beta^*(vu,x)-\beta^*(u,v\cdot x)\\
 &=\beta^*(v,u\cdot x)+\beta^*(u,x)-\beta^*(u,v\cdot x)\\
 &=\beta^*(v,u\cdot x)
\end{align*}
is a $U$-invariant smooth function on $\GG/\Gamma$ for every $v\in V$. By Howe-moore, it is constant. Therefore, setting
$c'(v)=\beta(v,x)-p(v\cdot x)+p(x)$, we have shown that $p$ satisfies
\begin{align*}
\beta(v,x)-p(v\cdot x)+p(x)=c'(v)
\end{align*}
for all $v\in V$ and $x\in \GG/\Gamma$. It is clear that $c'=c$ on $U$.

\section{Proof of Theorems~\ref{main_thm2-lower}}


Let $n \geq 2$.  We will do all computations in irreducible, unitary models of $\SL(n, \R)$, which
have Hilbert space norm $L^2(\R^{n-1})$, where the formulas for the
hyperbolic and unipotent elements in $\sl(n, \R)$ are given in \eqref{for:9} and \eqref{for:105}.
For $n = 2$, these computations will take place
in the line model of the principal series.
We first prove a lower bound for the Sobolev norm of the
solution of the twisted equation
\begin{equation}\label{eq:twist-cohomological}
(\mathfrak{u}_{ij} + \lambda \sqrt{-1}) f = g\,,
\end{equation}
in irreducible, unitary representations of $SL(n, \RR)$,
where $i, j \in \N$, $i > j$ and $\lambda \in \R^*$.
In particular, we show there are non-tame Sobolev estimates for the
solution to the above equation.

\begin{theorem}\label{main_thm1}
Let $n \geq 2$, and let $i > j \geq 1$.
For any $s \geq 0$, for any $\sigma \in [0, s+1/2)$,
and for any $\lambda \in \R^*$, the following holds.
For any constant $C > 0$, there is a constant $\delta > 0$
such that for any $\vert t \vert > \delta$,
there is a smooth vector $g \in Ind_P^{SL(n, \RR)}(\lambda_t^\pm)$ with a smooth solution $f \in Ind_P^{SL(n, \RR)}(\lambda_t^\pm)$ to the equation
\eqref{eq:twist-cohomological}
such that
\[
\Vert (I - \mathfrak{u}_{j, i}^2)^{s/2}f \Vert > C \Vert g \Vert_{s+\sigma}\,.
\]
\end{theorem}

As a consequence, we prove Theorem~\ref{main_thm2-lower},
namely, that non-tame Sobolev estimates
exist for the solution to the cohomological equation
\begin{equation}\label{eq:coeqn_general}
f \circ \exp(L \mathfrak{u}_{ij}) - f = g\,,
\end{equation}
where $L > 0$ and $\mathfrak{u}_{ij}$ is a unipotent vector of $\sl(n, \R)$.


\subsection{Proof of Theorem~\ref{main_thm1}}\label{sect:lower_setup}

Now fix $i, j$ as in Theorem~\ref{main_thm1}.
The estimate is based on an analysis with respect to the $x_{i - 1}$ variable.
For convenience, set
\[
a := i - 1\,, \ \ \ \nu := t \sqrt{-1}\,,
\]
and because we can take $\vert \nu \vert$ arbitrarily large, we assume
\[
\vert \nu \vert \geq 4\,.
\]
Also set
\[
n_{j-1} := \left\{\begin{array}{ll}
n - 2 & \text{ if } j = 1\,, \\
n - 3 & \text{ otherwise}\,.
\end{array}
\right.
\]

For any $(x_1, \cdots, x_{n-1}) \in \R^{n-1}$, define
\[
\begin{aligned}
& \mathbf{x} := (x_1, \cdots, x_{j-2}, x_j, \cdots x_{a-1}, x_{a+1}, \cdots, x_{n-1}) \in \R^{n_{j-1}}\,,\\
& f_{\mathbf{x}}(x_{j-1}, x_{a}) := f(x_1, \cdots, x_{n-1})\,.
\end{aligned}
\]
Notice that if $j = 1$, then $x_{j-1} = 1$, so
\[
\mathbf{x} := (x_1, \cdots x_{a-1}, x_{a+1}, \cdots, x_{n-1}) \in \R^{n_{j-1}}\,,
\]
and for convenience, we define
\[
f_{\mathbf{x}}(x_{j-1}, x_{a}) := f_{\mathbf{x}}(x_{a}) := f(x_1, \cdots, x_{n-1})\,.
\]

The Fourier transform
with respect to the $x_{a}$ variable is
\[
\hat f_{\mathbf x}(x_{j-1}, \omega)  = \int_\R f_{\mathbf x}(x_{j-1}, r) e^{-i r \omega} dr\,.
\]
The Fourier transform of the above vector fields \eqref{for:9} and \eqref{for:105} are
\begin{equation}\label{eq:Xhat}
\hat X_l = \left\{
\begin{array}{lll}
(\frac{n}{2}-2 + \nu) - 2\omega \partial \omega + \sum_{k = 2}^{n - 1} x_k \partial k \,,  & l = 1\,, a = 1\,, \\
 (\frac{n}{2} - 1 + \nu) + 2 x_1 \partial {x_1} - \omega \partial \omega +
 \sum_{
 \substack{
 2 \leq k \leq n - 1\\
 k \neq a
 }}
 x_k \partial {x_k}\,, & l = 1\,, a \neq 1 \,,  \\
1 + \omega \partial \omega + x_l \partial {x_l}\,,  & l \geq 2\,, a = l - 1 \,, \\
- (1 + x_{l-1} \partial {x_{l-1}} + \omega \partial \omega) \,,  & l \geq 2\,, a = l \,.
\end{array}
\right.
\end{equation}
For the unipotent vector fields $(\mathfrak{u}_{l, m})$, when $l = 1$, we have
\begin{equation}\label{eq:uhat1}
\begin{aligned}
& \hat{u}_{1, m} = \sqrt{-1} \\
& \times \left\{
\begin{array}{lll}
 [(\frac{n}{2} -2 + \nu) \partial \omega - \omega \partial^2 \omega] + \partial \omega \sum_{\substack{1 \leq k \leq n - 1 \\
 k \neq a}}  x_k \partial {x_k} & \text{ if } a = m - 1 \,,\\
 (\frac{n}{2} - 1 + \nu - \omega \partial \omega)x_{m - 1} + \sum_{
 \substack{
 1 \leq k \leq n - 1\\
 k \neq a
 }} x_{m - 1}  x_k \partial {x_k} & \text{ if } a \neq m - 1 \,.
\end{array}
\right.
\end{aligned}
\end{equation}
and when $l > 1$, we have
\begin{equation}\label{eq:uhat2}
\hat{\mathfrak{u}}_{l, m} = - \sqrt{-1} \times \left\{
\begin{array}{lll}
 \partial {\omega} \partial x_{l - 1} & \text{ if } a = m - 1\,, \\
  \omega x_{m- 1}& \text{ if } a = l - 1 \,, \\
  -\sqrt{-1} x_{m - 1} \partial {x_{l -1}} & \text{ otherwise}\,.
\end{array}
\right.
\end{equation}

Then in our model, the cohomological equation
\eqref{eq:twist-cohomological} is equivalently
\[
- \sqrt{-1} ( x_{j - 1} \omega - \lambda) \hat f_{\mathbf{x}}(x_{j-1}, \omega) = \hat g_{\mathbf{x}}(x_{j-1}, \omega)\,,
\]
which means
\begin{equation}\label{eq:coeqn-twist-gen_lambda}
  \hat f_{\mathbf{x}}(x_{j-1}, \omega) = \sqrt{-1} \frac{\hat g_{\mathbf{x}}(x_{j-1}, \omega)}{(x_{j - 1}\omega - \lambda)}\,.
\end{equation}
As in formula (36) of \cite{FFT}, we define $\hat g_{\mathbf{x}, \lambda}(x_{j-1}, \omega) := \hat g_{\mathbf{x}} (x_{j-1}, \lambda \omega)$ and study the following equivalent equation
for a solution $\hat f_{\mathbf{x}, \lambda}$ given by
\begin{align}
  \hat f_{\mathbf{x},\lambda}(x_{j-1}, \omega) & = \sqrt{-1} \frac{\hat g_{\mathbf{x},\lambda}(x_{j-1}, \omega)}{\lambda(x_{j - 1}\omega - 1)} \notag \\
 & = \frac{\sqrt{-1}}{\lambda x_{j-1}} \left(\frac{\hat g_{\mathbf{x},\lambda}(x_{j-1}, \omega)}{\omega - x_{j-1}^{-1}}\right)  \,. \label{eq:twist_eqn-lambda=1}
\end{align}
In what follows, we simplify notation and set
\begin{equation}\label{eq:f-flambda}
\hat g_{\mathbf{x}}:= \hat g_{\mathbf{x}, \lambda}\,, \quad \hat f_{\mathbf{x}} := \hat f_{\mathbf{x}, \lambda}\,.
\end{equation}

Define
\[
\begin{aligned}
I^{(j-1)} :=
\left\{
\begin{aligned}
& [\frac{4}{5}, \frac{5}{4}] & \text{ if } j = 1\,,\\
& [\frac{4}{5}, \frac{5}{4}]^2 & \text{ if } j > 1\,,
\end{aligned}
\right.  \quad \quad
I^{(j-1), e} :=
\left\{
\begin{aligned}
& [\frac{3}{4}, \frac{4}{3}] & \text{ if } j = 1\,,\\
& [\frac{3}{4}, \frac{4}{3}]^2 & \text{ if } j > 1\,.
\end{aligned}
\right.
\end{aligned}
\]
Recall that when $j = 1$, $x_{j-1} = 1$.
So for $j = 1$, it will be convenient to use the notation
\[
(x_{j-1}, \omega) \in I_\beta^{(0)} \text{ if and only if } \omega \in I_\beta^{(0)}\,.
\]

Now we will define the coboundary and transfer function
that will be used in the proof theorem.
Then let
\[
\hat g \in C_c^\infty([\frac{3}{4}, \frac{4}{3}]^{n-1})\,,
\]
and define
\begin{equation}\label{eq:fdef}
\hat g_{\mathbf{x}}(x_{j-1}, \omega) := \hat G_1 (x_{j-1}, \omega) G_2(\mathbf{x}) \,,
\end{equation}
where
\[
\begin{aligned}
& \hat G_1(x_{j-1}, \omega) = \hat q(x_{j-1}, \omega) (\omega^{\nu} - x_{j-1}^{-\nu})\,, \\
& q \in C_c^\infty(I_\beta^{(j-1), e})\,, \ \ q \equiv 1 \text{ on } I^{(j-1)}\,,
\end{aligned}
\]
and $G_2 := G_{2, j-1}$ satisfies
\[
G_2 \in C^\infty([\frac{3}{4}, \frac{4}{3}]^{n_{j-1}}) \,, \ \ \ G_2 \equiv 1 \text{ on }[\frac{4}{5}, \frac{5}{4}]^{n_{j-1}}\,.
\]

Notice that $g \in H^\infty$ and $\hat G_1(\omega^{-1}, \omega) = 0$.
Then by Theorem~\ref{th:6},
the equation \eqref{eq:twist-cohomological} has a smooth solution $g$,
which is given by \eqref{eq:twist_eqn-lambda=1}.  That is,
\[
\hat f_{\mathbf{x}} (x_{j-1}, \omega) = \hat F(x_{j-1}, \omega) G_2(\mathbf{x})\,,
\]
where
\begin{equation}\label{eq:Fdef}
\hat F(x_{j-1}, \omega)
:=  \frac{\sqrt{-1}}{\lambda x_{j-1}}
\left(
\frac{\hat G_1(x_{j-1}, \omega)}{\omega - x_{j-1}^{-1}}
\right)\,.
\end{equation}
\smallskip

Write
\[
\hat{\mathfrak{u}}_{1, i} = \hat V_0 + \hat V_c\,,
\]
where
\[
\begin{aligned}
&\hat V_0 = - \sqrt{-1} [(2 - \frac{n}{2} - \nu) \partial \omega + \omega \partial^2 \omega] \,, \\
&\hat V_c = \sqrt{-1} \sum_{k = 1}^{n - 1}  x_k \partial {x_k} \partial \omega\,.
\end{aligned}
\]
\smallskip

As a first step toward proving Theorem~\ref{main_thm1},
we decompose the operator $\hat{\mathfrak{u}}_{1, i}$
into partial derivatives and $\hat V_0$,
which compels us to use some notation.
For any $m \in \N\setminus\{0\}$
and for any integer vector $\boldsymbol{\alpha} = (\alpha_k)_{k = 1}^m \subset \N^m$,
define
\[
\vert \boldsymbol{\alpha} \vert := \sum_{k = 1}^m \alpha_k\,.
\]
For any $ l \leq \beta$,
consider the collection of integer vectors
\[
\mathcal Q^{(\beta-l)} := \left\{\boldsymbol{\alpha} = (\alpha_1, \cdots, \alpha_{a - 1}, \alpha_{a + 1}, \cdots, \alpha_{n - 1}) \in \N^{n - 2} : \vert \boldsymbol{\alpha}\vert = \beta - l\right\}\,.
\]
Furthermore, for any $ l \leq \beta$ and $j > 1$,
let $\boldsymbol{\alpha}_{\beta-l,j-1} \in \mathcal Q^{(\beta -l)}$
be given by
\[
\boldsymbol{\alpha}_{\beta-l,j-1} := \left\{
\begin{aligned}
& (\beta-l, \mathbf 0^{n-3}) & \text{ if } j = 2\,, \\
& (\mathbf 0^{j-2}, \beta-l, \mathbf 0^{n-j}) & \text{ if } j > 2\,,
\end{aligned}
\right.
\]
where for $i \in \N$, $\mathbf 0^i$ is a vector of $i$ 0's.
Also, for any $l \leq \beta$, define
\[
\Pi_\beta^{(l)}
\]
to be the sum of all products of $l$ operators $\hat V_0$
and $\beta - l$ operators $\partial \omega$.

\begin{lemma}\label{lemm:lower_twist}
For any $\beta \in \N$ and for any $l \leq \beta$,
there are $(c_{m, k}^{(\beta - l, \boldsymbol{\alpha})}) \subset \N \setminus\{0\}$ such that
\[
\hat{\mathfrak{u}}_{1, i}^\beta
=  \sum_{
\substack{ l \leq \beta \\
{\boldsymbol{\alpha} \in  \mathcal Q_{j-1}^{(\beta-l)}}}} (\sqrt{-1})^{\beta - l} \Pi_{\beta}^{(l)}
\prod_{\substack{1 \leq k \leq n - 1 \\ k \neq a}}
\sum_{m = 1}^{\alpha_k} c_{m, k}^{(\beta - l, \boldsymbol{\alpha})} x_k^m \partial^m x_k \,,
\]
where for $j > 1$, $c_{\beta-l, j-1}^{(\beta - l,\boldsymbol{\alpha}_{\beta-l,j-1})} = 1.$
\end{lemma}
\begin{proof}
Write
\begin{equation}\label{eq:u_1i-beta}
\hat{\mathfrak{u}}_{1, i}^\beta = (\hat V_0 + \hat V_c)^\beta  = \sum_{l \leq \beta} \Pi_{\beta}^{(l)}\left(\sqrt{-1}\sum_{\substack{1 \leq k \leq n - 1 \\ k \neq a}}
x_k \partial x_k\right)^{\beta - l} \,.
\end{equation}

Then for any $\mathbf{x} \in [\frac{4}{5}, \frac{5}{4}]^{n_{j-1}}$,
there are coefficients $(c_{\boldsymbol{\alpha}}^{(\beta - l)})\subset \N \setminus\{0\}$
such that
\begin{align}
\left(\sum_{\substack{1 \leq k \leq n - 1 \\ k \neq a}}
x_k \partial x_k\right)^{\beta - l}
& =  \sum_{\boldsymbol{\alpha} \in \mathcal Q^{(\beta -l)}}
c_{\boldsymbol \alpha}^{(\beta - l)}
\prod_{\substack{1 \leq k \leq n - 1 \\ k \neq a}}
(x_k \partial {x_k})^{\alpha_k}  \,. \label{eq:decompose-sumxk:2}
\end{align}
For any $\alpha_k \in \mathbb N\setminus\{0\}$, by induction we get
 $(c_{m, k}^{(\alpha_k)}) \subset \N \setminus\{0\}$ such that
\begin{equation}\label{eq:xpartialx}
(x_k \partial x_k)^{\alpha_k} = \sum_{m = 1}^{\alpha_k - 1}c_{m,k}^{(\alpha_k)} x_k^m \partial^m x_k\,.
\end{equation}
Hence, there are positive integers $c_{m, k}^{(\beta - l,\boldsymbol{\alpha})} = c_{\boldsymbol \alpha}^{(\beta - l)}  c_{m,k}^{(\alpha_k)}$ such that
\[
\begin{aligned}
\eqref{eq:decompose-sumxk:2}
& = \sum_{\boldsymbol{\alpha} \in  \mathcal Q_{j-1}^{(\beta-l)}}
\prod_{\substack{1 \leq k \leq n - 1 \\ k \neq a}}
\sum_{m = 1}^{\alpha_k} c_{m, k}^{(\beta - l, \boldsymbol{\alpha})} x_k^m \partial^m x_k \,,
\end{aligned}
\]
which implies
\[
\eqref{eq:u_1i-beta} = \sum_{\substack{ l \leq \beta \\
\boldsymbol{\alpha} \in  \mathcal Q_{j-1}^{(\beta-l)}}}
\Pi_{\beta}^{(l)}(\sqrt{-1})^{\beta-l} \prod_{\substack{1 \leq k \leq n - 1 \\ k \neq a}}
\sum_{m = 1}^{\alpha_k} c_{m, k}^{(\beta - l, \boldsymbol{\alpha})} x_k^m \partial^m x_k\,.
\]
This proves the decomposition formula for $\hat u^\beta$.

It remains to prove $c_{\beta-l, j-1}^{(\beta - l,\boldsymbol{\alpha}_{\beta-l,j-1})} = 1$, when $j > 1$.
Then for any $\mathbf{x} \in [\frac{4}{5}, \frac{5}{4}]^{n_{j-1}}$,
there are coefficients $(c_{\boldsymbol{\alpha}}^{(\beta - l)})\subset \N \setminus\{0\}$
such that
\begin{align}
\left(\sum_{\substack{1 \leq k \leq n - 1 \\ k \neq a}}
x_k \partial x_k\right)^{\beta - l}
& =  (x_{j-1} \partial x_{j-1})^{\beta-l} \notag \\
& + \sum_{\boldsymbol{\alpha} \in \mathcal Q^{(\beta -l)} \setminus\{\boldsymbol\alpha_{\beta, a}\}}
c_{\boldsymbol \alpha}^{(\beta - l)}
\prod_{\substack{1 \leq k \leq n - 1 \\ k \neq a}}
(x_k \partial {x_k})^{\alpha_k}  \,, \label{eq:decompose-sumxk:1}
\end{align}
where, as in \eqref{eq:xpartialx},
 $(c_{m, j-1}^{(\alpha_k)}) \subset \N \setminus\{0\}$ are such that
\[
(x_{j-1} \partial x_{j-1})^{\beta-l} = x_{j-1}^{\beta-l} \partial^{\beta-l} x_{{j-1}} +  \sum_{m = 1}^{\beta-l - 1}c_{m,{j-1}}^{(\alpha_{j-1})} x_{j-1}^m \partial^m x_{j-1}\,.
\]
So
\[
\begin{aligned}
\eqref{eq:decompose-sumxk:1}
& = \sum_{\boldsymbol{\alpha} \in  \mathcal Q_{j-1}^{(\beta-l)}}
\prod_{\substack{1 \leq k \leq n - 1 \\ k \neq a}}
\sum_{m = 1}^{\alpha_k} c_{m, k}^{(\beta - l, \boldsymbol{\alpha})} x_k^m \partial^m x_k \,,
\end{aligned}
\]
where $c_{\beta-l, j-1}^{(\beta - l,\boldsymbol{\alpha}_{\beta-l,j-1})} = 1$.
The lemma follows from this.
\end{proof}

Next, we prove an upper bound for the Sobolev norm of $g$.
We will use the following important identity that shows that even though $\hat V_0$ is a second
order differential operator, $\vert \hat V_0 \omega^{\nu+r}\vert$ grows only linearly
in $\nu$ for any real number $r$.  Indeed, for any $r \in \R$,
\begin{equation}\label{eq:V_0-omeganu}
 \hat V_0 \omega^{\nu+r} = -\sqrt{-1} \ (\nu+r) (r + 1 -\frac{n}{2}) \omega^{\nu+r-1}\,.
 \end{equation}
\begin{lemma}\label{lemm:f-bound}
For any $s \geq 0$, there is a constant $C_{s, n}^{(0)} > 0$ such that
\[
\Vert g \Vert_{s} \leq C_{s, n}^{(0)} \vert \nu \vert^{s}\,.
\]
\end{lemma}
\begin{proof}
Let $s \in \N$,
and recall that $g \in C_c^\infty([\frac{3}{4}, \frac{4}{3}]^{n-1})$, so it is
supported away from zero.
For each integer $\beta \leq s$, let $\mathcal B^{(\beta)}$ be the sum of all
products of $\beta$ vector fields in
$\{\hat X_j\}_{j = 1}^n \cup \{\hat{\mathfrak{u}}_{l, k}\}_{1 \leq l, k \leq n} \setminus\{\hat{\mathfrak{u}}_{1, i}, \hat{\mathfrak{u}}_{j, i} \}$.
Using the commutation relations and the triangle inequality,
we get a constant $C_{s,n} > 0$ such that
\begin{equation}\label{eq:g-B-u}
\Vert g \Vert_{s}  \leq C_{s,n} \sum_{\substack{\beta \leq s \\ \beta_1 + \beta_2 \leq \beta}}
\Vert \mathcal B^{(s-\beta)} \hat{\mathfrak{u}}_{j, i}^{\beta_1} \hat{\mathfrak{u}}_{1, i}^{\beta_2} \hat g\Vert\,.
\end{equation}
By Lemma~\ref{lemm:lower_twist} and the triangle inequality,
there is a constant $C_{s, n} > 0$ such that
\begin{align}
&\eqref{eq:g-B-u} \notag \\
& \leq C_{s,n} \sum_{\substack{\beta \leq s \\ \beta_1 + \beta_2 \leq \beta \\
l \leq \beta_2 }} \sum_{\substack{1 \leq k \leq n - 1 \\ k \neq a \\ \sum_{k \neq a} m_k = \beta_2-l}}
\Vert \mathcal B^{(s-\beta)} (\partial \omega \partial x_{j-1})^{\beta_1} \Pi_{\beta_2}^{(l)}
\prod_{(m_k)}
x_k^{m_k} \partial^{m_k} x_k \hat g\Vert \notag \\
& = C_{s,n} \sum_{\substack{\beta \leq s \\ \beta_1 + \beta_2 \leq \beta \\
l \leq \beta_2 }} \sum_{\substack{1 \leq k \leq n - 1 \\ k \neq a \\ \sum_{k \neq a} m_k = \beta_2-l}}
\Vert \mathcal B^{(s-\beta)} \prod_{(m_k)}
\left( x_k^{m_k} \partial^{m_k} x_k G_2 \right) \notag \\
& \times \left(\partial^{\beta_1}\omega\Pi_{\beta_2}^{(l)} \partial^{\beta_1} x_{j-1}(x_{j-1}^{m_{j-1}} \partial^{m_{j-1}} x_{j-1}) \hat G_1)\right)\Vert  \notag \\
& \leq C_{s, n} \sum_{\substack{\beta \leq s \\ \beta_1 + \beta_2 \leq \beta}}
\sum_{\substack{v_1 + v_2 = \beta_1 \\
l \leq \beta_2 }}
\sum_{\substack{1 \leq k \leq n - 1 \\ k \neq a \\ \sum_{k \neq a} m_k = \beta_2-l}}
\Vert \mathcal B^{(s-\beta)} \prod_{(m_k)}
\left( x_k^{m_k} \partial^{m_k} x_k G_2 \right) \notag \\
& \times \left(\partial^{\beta_1}\omega\Pi_{\beta_2}^{(l)} x_{j-1}^{m_{j-1}-v_1} \partial^{m_{j-1}+v_2} x_{j-1} \hat G_1)\right)\Vert\,,  \label{eq:f-est:1}
\end{align}
where in the last step, the term $x_{j-1}^{m_{j-1}-v_1}$ is bounded, because $x_{j-1} \in [\frac{3}{4}, \frac{4}{3}]$.

Because
\[
[\hat V_0\,, \partial \omega] = \sqrt{-1} \partial\omega^{2}\,,  \quad [\Pi_{\beta_2}^{(l)}, x_{j-1}^{m} \partial^{m} x_{j-1}] = 0
\]
we have
\begin{align}
\Vert \mathcal B^{(s-\beta)}  \prod_{(m_k)}&
\left( x_k^{m_k} \partial^{m_k} x_k G_2 \right) \left(\partial^{\beta_1}\omega\Pi_{\beta_2}^{(l)} x_{j-1}^{m_{j-1}-v_1} \partial^{m_{j-1}+v_2} x_{j-1} \hat G_1)\right)\Vert \notag \\
& \leq C_{s,n}
\sum_{\substack{\alpha_1 + \alpha_2\leq \beta_2 \\  \alpha_2 \leq l}}
\Vert \mathcal B^{(s-\beta)}  \prod_{(m_k)}
\left( x_k^{m_k} \partial^{m_k} x_k G_2 \right) \notag  \\
&\times \left(x_{j-1}^{m_{j-1}-v_1} \partial^{m_{j-1}+v_2} x_{j-1}  \partial^{\alpha_1+\beta_1}\omega
\hat V_0^{\alpha_2} \hat G_1\right) \Vert \,. \label{eq:F1-upper3}
\end{align}

Note that $\hat G_1$ is a product of two functions.
From formula~(39) and Lemma~3.7 of \cite{FFT}, there is a
Leibniz-type formula for the operator
$\hat V_0$.
Specifically, for $l = 1, 2$, there are universal coefficients
 $(b^{(\alpha_2)}_{qzkm})$ such that for any
 $x_{j-1}$ and any pair of functions
 $h_l := h_{l,x_{j-1}}$ in the variable $\omega$, we have
 \begin{equation}\label{eq:Leibnitz}
 \hat V_0^{\alpha_2}(h_1 h_2)  = \sum_{
\substack{
 z_0+z_1+ z_3\leq \alpha_2 \\
 z_2 \leq z_3
  }}
 b^{(\alpha_2)}_{z_0z_1z_2z_3} [\partial^{z_3}\omega \hat V_0^{z_0} h_1]
 [(\omega \partial\omega)^{z_2} \hat V_0^{z_1} h_2]\,.
\end{equation}
Then set
\[
h_1 := \omega^{\nu} - x_{j-1}^{-\nu}\,, \quad h_2 := \hat q\,.
\]
so
\[
\hat G_1 = h_{1} \cdot h_{2} \,.
\]

Notice that the term $(\omega \partial\omega)^{z_2} \hat V_0^{z_1} h_1$ is under control,
because $\hat q$ is smooth and compactly supported.
Moreover, by \eqref{eq:V_0-omeganu},
there is a complex polynomial $P$ of degree at most $z_3+z_0$ in $\nu$
such that
\[
\partial^{z_3}\omega \hat V_0^{z_0} h_1(x_{j-1}, \omega) = P(\nu) \omega^{\nu - (z_3+z_0)} - \delta_{z_3+z_0,  0} \,,
\]
where
\[
\delta_{z_3+z_0,0} :=
\left\{
\begin{aligned}
& 1 & \text{ if }z_1 + z_3 = 0\,, \\
& 0 & \text{ otherwise}\,.
\end{aligned}
\right.
\]
Similarly, for non-negative integers $(m_1, m_2) \neq (0, 0)$, notice that for any $r \in \N$,
\[
\begin{aligned}
\partial^{m_1} & x_{j-1} \partial^{m_2} \omega (\omega^{\nu - r} - x_{j- 1}^{-\nu})  \\
& :=
\left\{
\begin{aligned}
&\prod_{j = 0}^{m_2-1} (\nu - r-j) \omega^{\nu - r- m_2}  & \text{ if } m_1 = 0 \land m_2 \neq 0\,, \\
&-\prod_{j = 0}^{m_1-1} (\nu - r-j) x_{j-1}^{-\nu - m_1}  & \text{ if } m_1 \neq 0 \land m_2 = 0 \,, \\
& 0 & \text{ otherwise}\,.
\end{aligned}
\right.
\end{aligned}
\]
Because $z_1 + z_3 \leq \alpha_2$,
there is a constant $C_s > 0$ such that
\begin{equation}\label{eq:partial-x-omega-V}
\begin{aligned}
\vert\partial^{m_{j-1}+v_2} x_{j-1}  \partial^{\alpha_1+\beta_1}\omega \hat V_0^{\alpha_2} \hat G_1(x_{j-1}, \omega)\vert
& \leq C_s \vert \nu \vert^{\max\{m_{j-1}+v_2\,, \, \alpha_1+\beta_1 + \alpha_2\}} \\
& \leq C_s \vert \nu \vert^{\beta_1 + \beta_2} \\
& \leq C_s \vert \nu \vert^{\beta} \,.
\end{aligned}
\end{equation}

Finally, from \eqref{eq:Xhat}, \eqref{eq:uhat1} and \eqref{eq:uhat2}, observe that the
vector fields in $\{\hat X_j\}_{j = 1}^n \cup \{\hat{\mathfrak{u}}_{l, k}\}_{1 \leq l, k \leq n} \setminus\{\hat{\mathfrak{u}}_{1, i}, \hat{\mathfrak{u}}_{j, i}\}$
are first order differential operators with respect to $x_{j - 1}, \omega$.
Moreover, $G_2 \in C_c^\infty([\frac{3}{4}, \frac{4}{3}]^{n_j})$ is fixed
independent of $\nu$, so
we conclude that there is a constant $C_{s, n} > 0$ such that
for each $\beta \leq s$,
\begin{align}
\eqref{eq:F1-upper3} & \leq C_{s, n} \vert \nu \vert^{s - \beta +\beta} \notag \\
& \leq C_{s, n} \vert \nu \vert^{s} \,. \label{eq:F_1-B}
\end{align}

When $j = 1$, all vector fields except $\hat{\mathfrak{u}}_{1, i}$ are first order in $\omega$.
So this time, for each integer $\beta \leq s$,
we let $\mathcal B^{(\beta)}$ be the sum of all
products of $\beta$ vector fields in
$\{\hat X_j\}_{j = 1}^n \cup \{\hat{\mathfrak{u}}_{l, k}\}_{1 \leq l, k \leq n} \setminus\{\hat{\mathfrak{u}}_{1, i}\}$.
Hence, as in \eqref{eq:F1-upper3}, we have
\begin{align}
& \Vert g \Vert_{s} \leq C_{s,n} \sum_{\beta \leq s} \Vert \mathcal B^{(s-\beta)} \hat{\mathfrak{u}}_{1, i}^{\beta} \hat g\Vert \notag \\
& \leq C_{s, n} \sum_{\substack{\beta \leq s \\ l \leq \beta}}
\sum_{\substack{1 \leq k \leq n - 1 \\ k \neq a \\ \sum_{k \neq a} m_k = \beta_2-l}}
\Vert \mathcal B^{(s-\beta)} \prod_{(m_k)}
\left( x_k^{m_k} \partial^{m_k} x_k G_2 \right) \Pi_{\beta_2}^{(l)}  \hat G_1\Vert \notag \\
& \leq C_{s, n} \sum_{\substack{\beta \leq s \\ l \leq \beta}}
\sum_{\substack{1 \leq k \leq n - 1 \\ k \neq a \\ \sum_{k \neq a} m_k = \beta_2-l}} \sum_{\substack{\alpha_1 + \alpha_2\leq \beta_2 \\  \alpha_2 \leq l}} \Vert \mathcal B^{(s-\beta)} \prod_{(m_k)}
\left( x_k^{m_k} \partial^{m_k} x_k G_2 \right) \partial^{\alpha_1}\omega \hat V_0^{\alpha_2} \hat G_1\Vert\,. \label{eq:j=1-g-up}
\end{align}
Then as in \eqref{eq:partial-x-omega-V},
we get a constant $C_{s, n} > 0$ such that
\[
\begin{aligned}
\eqref{eq:j=1-g-up} & \leq C_{s, n} \vert \nu \vert^{s - \beta + \alpha_1 + \alpha_2} \\
& \leq C_{s, n} \vert \nu \vert^{s}\,.
\end{aligned}
\]

The lemma now follows from the above estimate, \eqref{eq:F_1-B} and interpolation.
\end{proof}

Now we focus on a lower bound for the $\hat{\mathfrak{u}}_{1, i}^\beta$ norm of $\hat f$.
\begin{theorem}\label{thm:lower-bound-uf}
For any $s \geq 0$,
there are constants $c_s^{(0)}, \alpha_s^{(0)} > 0$
such that the following holds.
For any $\vert \nu \vert \geq \alpha_s^{(0)}$,
for any $j \geq 1$,
we have
\[
\Vert (I - \mathfrak{u}_{j, i}^2)^{s/2}  f \Vert > \frac{c_s^{(0)}}{\vert\lambda\vert} \vert \nu \vert^{2s + 1/2}\,.
\]
\end{theorem}

The first step is to write a pointwise decomposition for $\hat{\mathfrak{u}}_{j, i}^\beta \hat f$.
\begin{lemma}\label{lemm:g-G}
For any $j \geq 1$, for any $\mathbf x \in [\frac{4}{5}, \frac{5}{4}]^{n_{j-1}}$
and for any $(x_{j-1}, \omega) \in I^{(j-1)}$,
\[
 \hat{\mathfrak{u}}_{j, i}^\beta \hat f_{\mathbf{x}}(x_{j-1}, \omega) =
 \left\{
 \begin{aligned}
 & \left(\hat V_0^{\beta} \hat F_1(1, \omega)\right) G_2(\mathbf x) & \text{ if } j = 1\,, \\
 & (-\sqrt{-1})^\beta \left(\partial^{\beta}x_{j-1} \partial^\beta\omega \hat F_1(x_{j-1}, \omega)\right) G_2(\mathbf x) & \text{ otherwise}\,.
 \end{aligned}
 \right.
\]
\end{lemma}
\begin{proof}
The formula for $j > 1$ is immediate from the
definition of $\hat{\mathfrak{u}}_{j, i}$ (see \eqref{eq:uhat2}).

For $j = 1$, by Lemma~\ref{lemm:lower_twist},
\begin{align}
\hat{\mathfrak{u}}_{1, i}^\beta =
\sum_{
\substack{ l \leq \beta \\
{\boldsymbol{\alpha} \in  \mathcal Q_{j-1}^{(\beta-l)}}}} \Pi_{\beta}^{(l)}
\prod_{\substack{1 \leq k \leq n - 1 \\ k \neq a}}
\sum_{m = 1}^{\alpha_k} c_{m, k}^{(\beta - l, \boldsymbol{\alpha})} (\sqrt{-1})^{\beta - l} x_k^m \partial^m x_k\,.  \label{eq:decompose-sumxk:3}
\end{align}

Notice that for any $l < \beta$ and
for any $\boldsymbol \alpha \in \mathcal Q^{(\beta-l)}$,
there is some $k_0$ such $\alpha_{k_0} \geq 1$.
Then because $G_2 \equiv 1$ on $[\frac{4}{5}, \frac{5}{4}]^{n - 2}$,
we get that for any $1 \leq m \leq \alpha_{k_0}$,
\[
\begin{aligned}
x_{k_0}^m \partial^m x_{k_0} \hat f_{\mathbf{x}}(1, \omega) & = \hat F(1, \omega) \left(x_{k_0}^m \partial^m x_{k_0} G_2(\mathbf x)\right)  \\
& = 0\,.
\end{aligned}
\]
So for such $\boldsymbol\alpha$ we get
\begin{equation}\label{eq:vanishing_else:1}
\sum_{m = 1}^{\alpha_{k_0}}
 c_{m, k_0}^{(\boldsymbol{\alpha}, \beta - l)} (\sqrt{-1})^{\beta - l} x_{k_0}^m \partial^m x_{k_0} f_{\mathbf x}(x_{j-1}, \omega)  = 0\,.
\end{equation}
Hence,
\[
\sum_{\substack{
l < \beta \\ \boldsymbol{\alpha} \in  \mathcal Q_{j-1}^{(\beta-l)}}} \Pi_{\beta}^{(l)}
\prod_{\substack{1 \leq k \leq n - 1 \\ k \neq a}} \sum_{m = 1}^{\alpha_k}
\tilde c_{m, k}^{(\beta - l, \boldsymbol{\alpha})} x_k^m \partial^m x_k f_{\mathbf{x}}(x_{j-1}, \omega) = 0\,.
\]
Recall that $\Pi_{\beta}^{(l)}$ is the sum of all products of $l$ operators in $\hat V_0$ and
$\beta-l$ operators in $\partial \omega$.  So we conclude
\[
\begin{aligned}
\hat{\mathfrak{u}}_{1, i}^\beta \hat f_{\mathbf{x}}(1, \omega)
 &  = G_2(\mathbf x) \Pi^{(\beta)} \hat F(1, \omega) \\
 & =  G_2(\mathbf x) \hat V_0^\beta \hat F(1, \omega)\,.
  \end{aligned}
  \]
\end{proof}

For any $r \in [0, 1]$ and for any
$(\omega, x_{j-1}) \in I^{(j-1)}$,  set
\[
\omega_{r,j}:= r (\omega-x_{j-1}^{-1}) + x_{j-1}^{-1}\,,
\]
and recall from \eqref{eq:Fdef} that
\[
\hat F(x_{j-1}, \omega)
= \frac{\sqrt{-1}}{\lambda x_{j-1}}
\left(
\frac{\hat G_1(x_{j-1}, \omega)}{\omega - x_{j-1}^{-1}}
\right)\,.
\]

Now using the above lemma, we will re-write
$\hat{\mathfrak{u}}_{j, i}^\beta \hat f_{\mathbf{x}}(x_{j-1}, \omega)$
as an operator on $\hat G_1$.
For each $l \leq \beta$, define $\mathcal L^{(\beta)}$ by
\[
\mathcal L^{(\beta)} := \{ \mathbf{l} = (l_0, l_1) \in \N^2 : l_0 +l_1 = \beta\}\,.
\]
Next, for any $\mathbf{l} \in \mathcal L^{(\beta)}$,
let $\mathcal L_{\mathbf{l}}^{(\beta)}$ be the set of all
sequences of length $\beta$ that contain $l_k$ elements
of the integer $k$, for each $k = 0, 1$.
For example, for $\beta = 5$,
\[
(1, 0, 1, 1, 0) \in \mathcal L_{(2, 3)}^{(5)}\,,
\]
because it contains two element "0", three elements "1".

Let
\begin{equation}\label{eq:W_i}
W_0 = \hat V_0\,,  \quad W_1 = \partial^2 \omega\,.
\end{equation}
Also, for any $r \in [0, 1]$ and for any
$(\omega, x_{j-1}) \in I^{(j-1)}$, set
\[
\omega_{r,j}:= r (\omega-x_{j-1}^{-1}) + x_{j-1}^{-1}\,.
\]
\begin{lemma}\label{lemm:PilG-F}
For any $j \geq 1$, for any $(x_{j-1}, \omega) \in I^{(j-1)}$ and for any $l \leq \beta$,
the following holds.
If $j = 1$, then
\[
\begin{aligned}
\hat V_0^{\beta} \hat F(1, \omega) = \frac{\sqrt{-1}}{\lambda} \sum_{\substack{\mathbf{l} \in \mathcal L^{(\beta)} \\ (s_k) \in \mathcal L_{\mathbf{l}}^{(\beta)}}} \int_0^1 r^\beta (\sqrt{-1}(1-r))^{l_1}  \prod_{k =1}^\beta W_{s_k} \partial\omega \hat G_1] (1, \omega) dr\,.
\end{aligned}
\]

If $j > 1$, then
\[
\begin{aligned}
\partial^{\beta}x_{j-1} \partial^\beta\omega \hat F_1(x_{j-1}, \omega) & = \frac{\sqrt{-1}}{\lambda} \sum_{\substack{m_1+m_3 = \beta}} c^{(\beta)}_{m_1 m_3} x_{j-1}^{-2m_1-m_3-1} \notag \\
&
 \times \int_0^1 r^{\beta} (1-r)^{m_1}
[\partial^{\beta+m_1+1}\omega \hat G_1] (x_{j-1}, \omega_{r, j}) dr\,,
\end{aligned}
\]
where $c^{(\beta)}_{\beta, 0} = (-1)^\beta$.
\end{lemma}
\begin{proof}
From the definition of $I^{(j-1)}$,
\[
(x_{j-1}, \omega_{r, j}) \in I^{(j-1)}\,, \quad \text{for } r \in \{0, 1\}\,.
\]
Because $I^{(j-1)}$ is convex
and $\{\omega_{r, j} : r \in [0, 1]\}$ is a line in $\R$,
we conclude that for any $r \in [0, 1]$,
$(x_{j-1}, \omega_{r, j}) \in I^{(j-1)}\,.$

Because for each $(x_{j-1}, \omega) \in I^{(j-1)}$,
$\hat G_1(x_{j-1}, x_{j-1}^{-1}) = 0$,
the fundamental theorem of calculus
shows that for all $r \in \R$,
\[
\begin{aligned}
\hat G_1(\omega, x_{j-1}^{-1}) & = \int_0^1 \frac{d}{dr} \hat G_1(x_{j-1}, r(\omega - x_{j-1}^{-1})+x_{j-1}^{-1}) dr \\
& = (\omega - x_{j-1}^{-1}) \int_0^1 [\partial\omega \hat G_1](x_{j-1}^{-1}, r(\omega - x_{j-1}^{-1})+x_{j-1}^{-1}) dr\,.
\end{aligned}
\]
Hence, as in Lemma~3.5 of \cite{FFT}, we get
\[
\hat F(x_{j-1},\omega) = \sqrt{-1} \frac{x_{j-1}^{\nu-1}}{\lambda}  \int_0^1 [\partial\omega \hat G_1](x_{j-1}, \omega_{r, j}) dr\,,
\]

Next, by a short calculation, as in formula (52) of \cite{FFT}, we get
\begin{align}
\hat V_0^\beta & \hat F(x_{j-1}, \omega)  = \sqrt{-1} \frac{x_{j-1}^{\nu-1}}{\lambda} \int_0^1 \hat V_0^{\beta} [\partial\omega \hat G_1] (x_{j-1}, \omega_{r, j}) dr \notag \\
& = \sqrt{-1} \frac{x_{j-1}^{\nu-1}}{\lambda} \int_0^1 r^\beta [(\hat V_0 +\sqrt{-1} x_{j-1}(1-r) \partial \omega^2)^{\beta} \partial\omega \hat G_1] (x_{j-1}, \omega) dr\,.  \label{eq:VXder}
\end{align}
By an induction argument,
\[
(\hat V_0 +\sqrt{-1} x_{j-1}^{-1}(1-r)  \partial^2 \omega)^\beta = \sum_{\substack{\mathbf{l} \in \mathcal L^{(\beta)} \\ (s_k) \in \mathcal L_{\mathbf{l}}^{(\beta)}}} (\sqrt{-1}x_{j-1}^{-1}(1-r))^{l_1}  \prod_{k =1}^l W_{s_k}\,.
\]
Hence, for $j = 1$ and $l = \beta$,
\[
\hat V_0^\beta \hat F(1, \omega) = \sqrt{-1} \sum_{\substack{\mathbf{l} \in \mathcal L^{(\beta)} \\ (s_k) \in \mathcal L_{\mathbf{l}}^{(\beta)}}} \frac{1}{\lambda} \int_0^1 r^\beta (\sqrt{-1}(1-r))^{l_1}  \prod_{k =1}^\beta W_{s_k} \partial\omega \hat G_1] (1, \omega) dr\,,
\]
which proves the first estimate in the lemma.

Now we prove the second, so $j > 1$.  It is clear that
\begin{equation}\label{eq:partial-omegaF}
\partial^{\beta}\omega \hat F(x_{j-1}, \omega)
= \frac{\sqrt{-1}}{\lambda x_{j-1}} \int_0^1 r^{\beta} [\partial^{\beta+1}\omega \hat G_1](x_{j-1}, \omega_{r, j}) dr\,.
\end{equation}
Next, notice that
\[
\partial x_{j-1} w_{r, j} = -(1-r) x_{j-1}^{-2}\,.
\]
By induction there are coefficients $(c_{m_1 m_3}^{(\beta)}) \subset \Z$
such that for any $H \in C^\infty(\R^2)$ supported
away the coordinate axes
\begin{align}
\partial^\beta x_{j-1}(x_{j-1}^{-1} \int_{0}^1 H(x_{j-1}, & \omega_{r, j})) dr =
\sum_{\substack{m_1+m_2 +m_3 = \beta}} c^{(\beta)}_{m_1 m_3} x_{j-1}^{-2m_1-m_3-1} \notag \\
& \times \int_0^1 (1-r)^{m_1} [\partial^{m_2} x_{j-1} \partial^{m_1}\omega H](x_{j-1}, \omega_{r, j}) dr\,,\notag
\end{align}
where
\begin{equation}\label{eq:cbeta-0}
c_{\beta 0}^{(\beta)} = (-1)^{\beta}\,.
\end{equation}
Then by the above equality and \eqref{eq:partial-omegaF},
we have
\begin{align}
\partial^\beta x_{j-1} \partial^\beta\omega & \hat F(x_{j-1}, \omega)  = \frac{\sqrt{-1}}{\lambda} \sum_{\substack{m_1+m_2 +m_3 = \beta}} c^{(\beta)}_{m_1 m_3} x_{j-1}^{-2m_1-m_3-1} \notag \\
&
 \times \int_0^1 r^{\beta} (1-r)^{m_1}
[\partial^{m_2} x_{j-1} \partial^{\beta+m_1+1}\omega \hat G_1] (x_{j-1}, \omega_{r, j}) dr\,.   \label{eq:decompF:2}
\end{align}
Finally, observe that $\partial \omega \hat G_1$ is a function of $\omega$ alone.
Hence, whenever $m_2 > 0$, we have
\[
\partial^{m_2} x_{j-1} \partial^{\beta+m_1+1}\omega \hat G_1 = 0\,,
\]
which means
\[
\begin{aligned}
\eqref{eq:decompF:2} & = \frac{\sqrt{-1}}{\lambda} \sum_{\substack{m_1+m_3 = \beta}} c^{(\beta)}_{m_1 m_3} x_{j-1}^{-2m_1-m_3-1} \notag \\
&
 \times \int_0^1 r^{\beta} (1-r)^{m_1}
[\partial^{\beta+m_1+1}\omega \hat G_1] (x_{j-1}, \omega_{r, j}) dr\,,
\end{aligned}
\]
where $c^{(\beta)}_{\beta 0}$ is given by \eqref{eq:cbeta-0}.
\end{proof}

As a consequence, we have
\begin{corollary}\label{coro:u1i-beta}
For any $j \geq 1$ and for any $(x_{j-1}, \omega) \in I^{(j-1)}$,
the following holds.
If $j = 1$, then
\[
\begin{aligned}
\hat{\mathfrak{u}}_{1, i}^\beta \hat f_{\mathbf{x}}(1, \omega) & =  \frac{\sqrt{-1}}{\lambda} G_2(\mathbf x)\sum_{\substack{\mathbf{l} \in \mathcal L^{(\beta)} \\ (s_k) \in \mathcal L_{\mathbf{l}}^{(\beta)}}} \int_0^1 r^\beta (\sqrt{-1}(1-r))^{l_1}  \prod_{k =1}^\beta W_{s_k} \partial\omega \hat G_1] (1, \omega) dr\,.
\end{aligned}
\]

If $j > 1$, then
\[
\begin{aligned}
\hat{\mathfrak{u}}_{j, i}^\beta \hat f_{\mathbf{x}}(1, \omega) & = -\frac{(-\sqrt{-1})^{\beta+1}}{\lambda} G_2(\mathbf x) \sum_{m_1+m_3 = \beta} c^{(\beta)}_{m_1 m_3} x_{j-1}^{-2m_1-m_3-1} \notag \\
&
 \times \int_0^1 r^{\beta} (1-r)^{m_1}
[\partial^{\beta+m_1+1}\omega \hat G_1] (x_{j-1}, \omega_{r, j}) dr\,,
\end{aligned}
\]
where $c^{(\beta)}_{\beta 0} = (-1)^\beta$.
\end{corollary}
\begin{proof}
This is immediate from the above lemma and Lemma~\ref{lemm:g-G}.
\end{proof}

Now for any $j \geq 1$ and $\mathbf x \in [\frac{4}{5}, \frac{5}{4}]^{n_{j-1}}$,
$(x_{j-1}, \omega) \in I^{(j-1)}$, write
\[
\hat{\mathfrak{u}}_{1, i}^\beta  \hat f_{\mathbf{x}}(x_{j-1}, \omega) =
T_{\mathbf x}^{(j-1)}(x_{j-1}, \omega) + B_{\mathbf x}^{(j-1)}(x_{j-1}, \omega)\,,
\]
where $T^{(j-1)}$ and $B^{(j-1)}$ are defined on $[\frac{4}{5}, \frac{5}{4}]^{n-1}$ as follows.
For $j = 1$, set
\[
T_{\mathbf x}^{(0)}(1, \omega) := G_2(\mathbf x) \frac{(\sqrt{-1})^{\beta+1}}{\lambda}  \int_0^1 r^{\beta} (1-r)^{\beta}
[\partial^{2 \beta + 1}\omega\hat G_1] (1, \omega_{r, j}) dr\,, \\
\]
and for $j > 1$,
\[
\begin{aligned}
T_{\mathbf x}^{(j-1)}(x_{j-1}, \omega) & :=
\frac{(\sqrt{-1})^{\beta+1}}{\lambda} G_2(\mathbf x) x_{j-1}^{-(2\beta+1)} \\
& \times \int_0^1 r^{\beta} (1-r)^{\beta}
[\partial^{2\beta+1}\omega \hat G_1] (x_{j-1}, \omega_{r, j}) dr\,.
\end{aligned}
\]
Then $B^{(j-1)}$ is given by
\[
B_{\mathbf x}^{(j-1)}(x_{j-1}, \omega) := \hat{\mathfrak{u}}_{j, i}^\beta  \hat f_{\mathbf{x}}(x_{j-1}, \omega) - T_{\mathbf x}^{(j-1)}(x_{j-1}, \omega)\,.
\]
We will show that $T^{(j-1)}$ is the term with the largest
power of $\nu$ so that $B_{\mathbf x}^{(j-1)}(x_{j-1}, \omega)$
will be a remainder term.
\begin{lemma}\label{lemm:B}
There is a constant $C_\beta^{(0)} > 0$ such that,
for any $j \geq 1$ and for any $(x_{j-1}, \omega) \in I^{(j-1)}$,
\[
\vert B_{\mathbf x}^{(j-1)}(x_{j-1}, \omega) \vert \leq \frac{C_\beta^{(0)}}{\vert \lambda \vert} \vert \nu \vert^{2\beta}\,.
\]
\end{lemma}
\begin{proof}
We first consider the case $j = 1$.
If $l_0 > 0$ in the pair $(l_0, l_1)$,
then $\prod_{k =1}^\beta W_{s_k}$ contains
at least one term $\hat V_0$ in place of $\partial^2 \omega$.
From \eqref{eq:V_0-omeganu},
for any $r \in \R$,
we have
\[
\begin{aligned}
& \hat V_0 \omega^{\nu + r}  = -\sqrt{-1} \ (\nu+r) (r + 1 -\frac{n}{2}) \omega^{\nu+r-1} \\
& \partial \omega (\omega^{\nu + r}) = (\nu + r) \omega^{\nu + r - 1}\,.
\end{aligned}
\]
Notice in particular that $\vert \hat V_0 \omega^{\nu + r}\vert$ grows linearly in $\vert \nu\vert$.
So we get a constant $C_{\beta} > 0$ such that for any $m \leq \beta-l$,
\begin{align}
\vert \int_0^1 r^\beta (\sqrt{-1}(1-r))^{l_1}  \prod_{k =1}^\beta W_{s_k} \partial\omega \hat G_1] (1, \omega) dr \vert
& \leq C_{\beta} \vert \nu \vert^{l_0 + 2l_1 + 1}\notag \\
& \leq  C_{\beta} \vert \nu \vert^{1 + 2(\beta-1) +1} \notag \\
& \leq  C_{\beta} \vert \nu \vert^{2\beta}\,. \notag
\end{align}
The by Corollary~\ref{coro:u1i-beta} and because $G_2$ is bounded,
this implies the estimate when $j = 1$.

For $j > 1$, if $m_1 < \beta$, then clearly,
there is a constant $C_\beta> 0$ such that
\[
\begin{aligned}
\vert \int_0^1 r^{\beta} (1-r)^{m_1}
[\partial^{\beta+m_1+1}\omega \hat G_1] (x_{j-1}, \omega_{r, j}) dr\vert & \leq C_\beta \vert \nu \vert^{\beta+m_1+1} \\
& \leq C_\beta \vert \nu \vert^{2\beta} \,.
\end{aligned}
\]
Because $x_{j-1}$ is bounded away from zero, the estimate follows as in the case $j = 1$.
\end{proof}

For any $\vert \nu \vert \geq 4$, define
\begin{equation}\label{eq:I_nu-def}
I_{\nu} := [1, 1 + \frac{1}{\vert \nu \vert}]\,. \\
\end{equation}
\begin{lemma}\label{lemm:lower-bound:4}
There is a constant $c_{\beta}^{(1)} > 0$ such that
for any $\vert \nu \vert \geq 4$ and
for any $(x_{j - 1}, \omega) \in I^{(j-1)}$
such that $x_{j-1} \omega \in I_{\nu}$,
we have
\[
\vert \int_0^1 r^{\beta} (1 - r)^\beta
[\partial^{2\beta +1}\omega\hat G_1] (x_{j-1}, \omega_{r, j}) dr \vert > c_{\beta}^{(1)} \vert \nu \vert^{2\beta + 1}\,.
\]
\end{lemma}
\begin{proof}
Clearly,
\[
\partial^{2\beta +1}\omega\hat G_1 = \prod_{j = 0}^{2\beta} (\nu - j) \omega^{\nu - (\beta + 1)}\,.
\]
Because $\nu \in i \R$, we have
\begin{equation}\label{eq:prodnu}
\vert \prod_{j = 0}^{2\beta} (\nu-j)\vert \geq \vert \nu \vert^{2\beta+1}\,,
\end{equation}
which means
\begin{equation}\label{eq:max-lower:2}
\vert \int_0^1 r^{\beta} (1 - r)^\beta
[\partial^{\beta +1}\omega\hat G_1] (x_{j-1}, \omega_{r, j}) dr \vert  \geq \vert \nu \vert^{\beta+1} \vert
\int_0^1 r^\beta  (1 - r)^\beta \omega_{r, j}^{\nu - (\beta+1)} dr\vert\,.
\end{equation}

Notice
\begin{align}
\omega_{r, j} & = r(\omega - x_{j-1}^{-1}) + x_{j-1}^{-1} \notag \\
& = x_{j-1}^{-1} \left( r(\omega x_{j-1} - 1) + 1\right)\,, \label{eq:xomega:1}
\end{align}
so
\[
\omega_{r, j}^{\nu - (\beta + 1)} = x_{j-1}^{-\nu + \beta+1} \left( r(\omega x_{j-1} - 1) + 1\right)^{\nu - (\beta + 1)}\,.
\]
Hence,
\begin{equation}\label{eq:max-lower:3}
\eqref{eq:max-lower:2} = \vert \nu \vert^{\beta+1}  x_{j-1}^{\beta+1} \vert
\int_0^1 r^\beta (1 - r)^\beta \left( r(\omega x_{j-1} - 1) + 1\right)^{\nu - (\beta + 1)} dr\vert\,.
\end{equation}

Now say $\nu \in i \R^+$.
Then we can write
\[
\begin{aligned}
& \omega_{r, j}^\nu = \exp\left(\sqrt{-1} \vert \nu\vert \log (r(\omega x_{j-1} - 1) + 1)\right) \\
& = \cos\left(\vert \nu \vert \log (r(\omega x_{j-1} - 1) + 1)\right) + \sqrt{-1} \sin\left(\vert \nu \vert \log (r(\omega x_{j-1} - 1) + 1)\right)\,.
\end{aligned}
\]
So taking the real part of the integral  \eqref{eq:max-lower:3}, we get
\begin{align}
\eqref{eq:max-lower:3}
& \geq \vert \nu \vert^{\beta+1}x_{j-1}^{\beta+1} \vert
\int_0^1 r^\beta (1 -r)^\beta \omega_{r, j}^{- (\beta+1)} \cos\left(\vert \nu \vert \log (r(\omega x_{j-1} - 1) + 1)\right) dr\vert  \,.\label{eq:max-lower:4}
\end{align}

For any $r \in [0, 1]$,
define
\[
\phi(r, j, \omega) := \log(1+r(\omega x_{j - 1}-1)) - r(\omega x_{j- 1}-1)\,.
\]
Because $\vert \nu \vert \geq 4$ and $\omega x_{j-1} \in I_{\nu}$, we have
\[
\begin{aligned}
\vert \phi(r, j, \omega) \vert \leq \frac{1}{\vert \nu \vert^2}\,.
\end{aligned}
\]
So for any $r \in [0, 1]$,
\begin{align}
\cos(\vert \nu \vert \log (1+r(\omega x_{j - 1} -1) &)) = \cos\left(\vert \nu \vert (r(\omega x_{j - 1}-1) + \phi(r, \omega))\right)  \notag \\
& \geq \cos(1 + \frac{1}{\vert \nu \vert}) > \frac{1}{20}\,. \label{eq:cos-bound}
\end{align}
Hence, the integrand in \eqref{eq:max-lower:4} is positive.

Finally, we use \eqref{eq:xomega:1} and $x_{j-1} \omega \in I_{\nu}$ to get for any $r \in [0, 1]$,
\[
\vert \omega_{r, j}\vert > x_{j-1}^{-1} \,.
\]
Then it follows from the above estimate and \eqref{eq:cos-bound}
that
\[
\begin{aligned}
\eqref{eq:max-lower:4} & \geq \frac{\vert \nu \vert^{2\beta+1}}{2} \vert
\int_0^1 r^\beta (1-r)^\beta dr\vert  \\
& \geq \vert \nu \vert^{2\beta+1} 2^{-(4\beta+1)} \,.
\end{aligned}
\]
\end{proof}

As a consequence, we have
\begin{corollary}\label{lemm:T_1-lower}
There is a constant $c_{\beta}^{(2)} \in (0, 1)$ such that for any $\vert \nu \vert \geq 4$,
for any $j \geq 1$,
for any $\mathbf x \in [\frac{4}{5}, \frac{5}{4}]^{n_{j-1}}$,
and for any $(x_{j - 1}, \omega) \in I^{(j-1)}$ such that
$x_{j-1}\omega \in I_{\nu}$, we have
\[
\vert T_{\mathbf x}^{(j-1)}(x_{j-1}, \omega) \vert > c_\beta^{(2)} \vert \nu \vert^{2\beta + 1}\,.
\]
\end{corollary}
\begin{proof}
This is immediate by the above lemma, by the definition of $G_2$,
and because $x_{j-1}$ is bounded.
\end{proof}

We can now prove Theorem~\ref{thm:lower-bound-uf}.
\begin{proof}[Proof of Theorem~\ref{thm:lower-bound-uf}]
Let $C_\beta^{(0)}$ and $c_\beta^{(2)}$ be the constants from Lemma~\ref{lemm:B} and Corollary~\ref{lemm:T_1-lower},
and let $(\alpha_\beta^{(0)})$ be a sequence such that for each $\beta \in \N$, $\alpha_\beta^{(0)} \geq 2$ and
\[
\begin{aligned}
& \alpha_{\beta + 1}^{(0)} > \alpha_{\beta}^{(0)} \\
& c_\beta^{(2)} - C_\beta^{(0)} (\alpha_\beta^{(0)})^{-1} > \frac{c_\beta^{(2)}}{2}\,.
\end{aligned}
\]
Then by Lemma~\ref{lemm:B} and Corollary~\ref{lemm:T_1-lower}
and by the triangle inequality,
for any $\vert \nu \vert > \alpha_\beta^{(0)}$, we have
\[
\begin{aligned}
\vert \hat{\mathfrak{u}}_{1, i}^\beta  \hat f_{\mathbf{x}}(x_{j-1}, \omega) \vert & \geq
\left| \vert T(x_{j-1}, \omega) \vert - \vert B(x_{j-1}, \omega) \vert \right| \\
& \geq \frac{\vert \nu \vert^{2\beta + 1}}{\vert\lambda\vert} ( c_\beta^{(2)} - \frac{C_\beta^{(0)}}{ \vert \nu \vert}) \\
& > \frac{\vert \nu \vert^{2\beta + 1}}{\vert\lambda\vert} \frac{c_\beta^{(2)}}{2}\,.
\end{aligned}
\]

Define
\begin{equation}\label{eq:Inuj}
I_{\nu, j - 1} := \{(x_{j-1}, \omega) \in I^{(j-1)} : x_{j-1} \omega \in I_\nu\}\,.
\end{equation}
Therefore, there is a constant $c_\beta^{(0)} > 0$ such that
\[
\begin{aligned}
\Vert (I - \mathfrak{u}_{j, i}^2)^{\beta/2} f \Vert & \geq \Vert (I - \mathfrak{u}_{1, i}^2)^{\beta/2} f \Vert_{L^2([\frac{4}{5}, \frac{5}{4}]^{n_{j-1}} \times I_{\nu, j - 1})} \\
& >  \frac{c_\beta^{(0)}}{\vert\lambda\vert} \vert \nu \vert^{2\beta + 1/2} \,.
\end{aligned}
\]

Now let $s \geq 0$, and let $\beta = \lfloor s \rfloor$.
Then let $\vert \nu \vert > \alpha_{\beta+1}^{(0)}$.
So because $\alpha_{\beta+1}^{(0)} > \alpha_{\beta}^{(0)}$ the above estimate gives
\begin{equation}\label{eq:interpolate-beta}
\begin{aligned}
& \Vert (I - \mathfrak{u}_{j, i}^2)^{\beta/2} f \Vert  > \frac{c_\beta^{(0)}}{\vert\lambda\vert} \vert \nu \vert^{2\beta + 1/2} \\
& \Vert (I - \mathfrak{u}_{j, i}^2)^{(\beta+1)/2} f \Vert > \frac{c_\beta^{(0)}}{\vert\lambda\vert} \vert \nu \vert^{2(\beta+1) + 1/2}\,.
\end{aligned}
\end{equation}
By interpolation, we conclude that
\[
\Vert (I - \mathfrak{u}_{j, i}^2)^{s/2} f \Vert > \frac{c_{\lfloor s+1\rfloor}^{(0)}}{\vert\lambda\vert} \vert \nu \vert^{2s + 1/2}\,.
\]
\end{proof}

\begin{proof}[Proof of Theorem~\ref{main_thm1}]
Let $s \geq 0$, $C > 0$, $\lambda \in \R^*$ and let $\sigma \in [0, s+1/2)$.
Let $C_{s+\sigma, n}^{(0)} > 0$ and $\alpha_s^{(0)}, c_s^{(0)} > 0$ be the constants from
Lemma~\ref{lemm:f-bound} and Theorem~\ref{thm:lower-bound-uf}, respectively.
Let $\vert \nu\vert$ be large enough that
\begin{equation}\label{eq:nu-large-enough}
\left\{
\begin{aligned}
& \vert \nu \vert \geq \alpha_s^{(0)}\,,  \\
&  c_s^{(0)} \vert \lambda \vert^{s-1/2} \vert \nu \vert^{2s+1/2}  > C C_{s+\sigma, n}^{(0)} (\vert \lambda \vert^{-(s+\sigma)+1/2} + \vert \lambda \vert^{s+\sigma+1/2} ) \vert \nu \vert^{s+\sigma}\,.
\end{aligned}
\right.
\end{equation}

We will compare the results of Lemma~\ref{lemm:f-bound}
and Theorem~\ref{thm:lower-bound-uf} that give
\begin{equation}\label{eq:uf-norm:1}
\begin{aligned}
& \Vert g \Vert_{s} \leq C_{s, n}^{(0)} \vert \nu \vert^{s}\,, \\
& \Vert \hat{\mathfrak{u}}_{j, i}^s \hat f \Vert > \frac{c_{s}^{(0)}}{\vert\lambda\vert} \vert \nu \vert^{2s + 1/2}\,.
\end{aligned}
\end{equation}

Recall from \eqref{eq:f-flambda} that the above analysis holds for functions $\hat g_{\mathbf{x}, \lambda}(x_{j-1}, \omega) := \hat g_{\mathbf{x}}(x_{j-1}, \lambda \omega)$,
so from \eqref{eq:fdef},
\begin{equation}\label{eq:f-lambda}
\hat g_{\mathbf{x}} (x_{j-1}, \omega) = \hat q(x_{j-1}, \lambda^{-1}\omega)(\lambda^{-\nu} \omega^{\nu} -x_{j-1}^{-\nu}) \hat G_2(\mathbf{x})\,,
\end{equation}
which means that $\hat g_{\mathbf{x}}$ is defined on
\[
\{(x_{j-1}, \omega) :  (x_{j-1}, \lambda\omega) \in I^{(j-1)}\}\,,
\]
and $\hat g_{\mathbf{x}}(\omega^{-1}, \lambda \omega) = 0$ for any $\omega \in \R$.

Recall the formulas for the basis vectors $\mathcal V := \{\hat X_l\}_{l = 1}^{n-2} \cup \{\hat{\mathfrak{u}}_{l m}\}_{1 \leq l \neq m\leq n-1} \subset \sl(n, \R)$, given in \eqref{eq:Xhat}, \eqref{eq:uhat1} and \eqref{eq:uhat2}, and define
\[
\begin{aligned}
&E^+ := -\sum_{m \neq i} \hat{\mathfrak{u}}_{m,i}^2 \,,  \ \ \
E^- := -\sum_{m \neq i} \hat{\mathfrak{u}}_{i, m}^2 \\
&E^0 := -\sum_{W \in \mathcal V} W^2 - E^+ - E^-\,.
\end{aligned}
\]
A calculation shows
\[
\partial_{x_{j-1}} g_\lambda = \lambda [\partial_{x_{j-1}} g]_\lambda\,, \ \ \ x_{j-1} g_\lambda = \lambda^{-1} [x_{j-1} g]_\lambda\,,
\]
so
\[
E^- g_{\lambda} = \lambda^{-2} [E^- g]_\lambda\,, \ \ \ E^0 g_\lambda = [E^0 g]_\lambda\,, \ \ \ E^+ g_{\lambda} = \lambda^2 [E^+ g]_\lambda\,.
\]

Then because each of the operators $E^\eta$,
for $\eta \in \{-, 0, +\}$ are positive operators,
\[
\begin{aligned}
\Vert \hat g_\lambda \Vert_{s + \sigma} & = \Vert (I - \sum_{W \in \mathcal V} W^2)^{(s+\sigma)/2} \hat g_\lambda \Vert \\
& = \Vert (1 + E^- + E^0 + E^+)^{(s+\sigma)/2} \hat g_\lambda \Vert \\
&  = \Vert [(I + \lambda^{-2} E^- + E^0 + \lambda^{2} E^+)^{(s+\sigma)/2} \hat g]_\lambda \Vert \\
& \geq \min\{\vert \lambda\vert^{-(s+\sigma)}, \vert \lambda \vert^{(s+\sigma)}\} \Vert [(I - \sum_{W \in \mathcal V} W^2)^{(s+\sigma)/2} \hat g]_\lambda \Vert \\
& = \min\{\vert \lambda\vert^{-(s+\sigma)}, \vert \lambda \vert^{(s+\sigma)}\} \lambda^{-1/2} \Vert  \hat g\Vert_{s + \sigma}  \\
& \geq (\vert \lambda\vert^{-(s+\sigma) + 1/2} +  \vert \lambda \vert^{s+\sigma+1/2})^{-1} \Vert \hat g \Vert_{s+\sigma} \,.
\end{aligned}
\]
Similarly,
\begin{align}
\Vert (\hat{\mathfrak{u}}_{j, i}^2)^{s/2} \hat f \Vert & = \Vert (\hat{\mathfrak{u}}_{j, i}^2)^{s/2}(\hat f_\lambda)_{1/\lambda} \Vert \notag \\
& = \vert \lambda \vert^{s - 1/2} \Vert (\hat{\mathfrak{u}}_{j, i}^2)^{s/2} \hat f_\lambda \Vert\,. \label{eq:u-lambda-f}
\end{align}

Then using \eqref{eq:uf-norm:1}, \eqref{eq:nu-large-enough}
and the above estimates, we get
\begin{align}
\Vert (\mathfrak{u}_{j, i}^2)^{s/2} f \Vert_s &
= \vert \lambda \vert^{s - 1/2} \Vert (\hat{\mathfrak{u}}_{j, i}^2)^{s/2} \hat f_\lambda \Vert \notag \\
&  \geq c_s^{(0)}  \vert \lambda \vert^{s-1/2} \vert \nu \vert^{2s+1/2} \notag\\
& > C C_{s+\sigma, n}^{(0)} (\vert \lambda \vert^{-(s+\sigma)+1/2} + \vert \lambda \vert^{s+\sigma+1/2} ) \vert \nu \vert^{s+\sigma}  \notag \\
& \geq C \Vert \hat g _\lambda\Vert_{s +\sigma} (\vert \lambda \vert^{-(s+\sigma)+1/2} + \vert \lambda \vert^{s+\sigma+1/2} ) \notag \\
& \geq C  \Vert g \Vert_{s+\sigma} \,.\notag
\end{align}
Therefore,
\[
\Vert f \Vert_s > C \Vert g \Vert_{s+\sigma}\,.
\]
\end{proof}

\subsection{Proof of Theorem~\ref{main_thm2-lower}}
\begin{proof}[Proof of Theorem~\ref{main_thm2-lower}]
We derive Theorem~\ref{main_thm2-lower} from Theorem~\ref{main_thm1}
Using the Fourier transform $L^2(\R^{n-1})$ model, the
cohomological equation \eqref{eq:coeqn_general}
for unipotent maps has the form
\begin{equation}\label{eq:Fourier-coeqn}
(e^{-L \sqrt{-1} x_{j-1} \omega} - 1)\hat f_{\mathbf{x}}(x_{j-1}, \omega) = \hat g_{\mathbf{x}}(x_{j-1}, \omega)\,.
\end{equation}
Using the notation from Section~\ref{sect:lower_setup},
define
\[
\begin{aligned}
\hat g_{\mathbf{x}}^{twist}(x_{j-1}, \omega) & :=  \hat G_1(x_{j-1}, (\frac{2\pi}{L})^{-1} \omega) G_2(\mathbf x) \\
& = q(x_{j-1}, \frac{L}{2\pi}\omega) [(\frac{L}{2\pi}\omega)^\nu - x_{j-1}^{-\nu}] G_2(\mathbf x)\,.
\end{aligned}
\]
Define $f$ by
\[
\begin{aligned}
\hat f_{\mathbf x}(x_{j-1}, \omega) = & \hat F(x_{j-1}, (\frac{2\pi}{L})^{-1}\omega) G_2(\mathbf x) \\
& = q(x_{j-1}, \frac{L}{2\pi}\omega) \left(\frac{(\frac{L}{2\pi}\omega)^\nu - x_{j-1}^{-\nu}}{\frac{L}{2\pi} \omega x_{j-1} - 1}\right) G_2(\mathbf x)\,,
\end{aligned}
\]
so $g^{twist}$ and $f$ satisfy \eqref{eq:fdef} and \eqref{eq:Fdef}, respectively,
in the scaled case $\lambda = \frac{2\pi}{L}$.
Note further that these functions also satisfy the twisted
equation \eqref{eq:coeqn-twist-gen_lambda}
\begin{equation}\label{eq:g_map-twist}
 \hat f_{\mathbf{x}}(x_{j-1}, \omega) = \frac{2\pi}{L} \frac{\hat g_{\mathbf{x}}^{twist}(x_{j-1}, \omega)}{(x_{j - 1}\omega - \frac{2\pi}{L})}
\end{equation}
up to multiplication by the scalar $\frac{L}{2\pi}\sqrt{-1}$,
and they are supported on
\[
(x_{j-1}, \omega) \in [\frac{3}{4}, \frac{4}{3}] \times [\frac{2\pi}{L} \frac{3}{4}, \frac{2\pi}{L} \frac{4}{3}]\,.
\]
Next, define $H$ on an open neighborhood of $[\frac{3}{4}, \frac{4}{3}] \times [\frac{2\pi}{L} \frac{3}{4}, \frac{2\pi}{L} \frac{4}{3}]$ by
\[
H(x_{j-1}, \omega) = \left(\frac{e^{-L \sqrt{-1} x_{j-1} \omega} - 1}{\frac{L}{2\pi} \omega x_{j-1}-  1}\right)\,.
\]
and notice that
\begin{equation}\label{eq:H-smooth}
H, H^{-1} \in C^\infty([\frac{3}{4}, \frac{4}{3}] \times [\frac{2\pi}{L} \frac{3}{4}, \frac{2\pi}{L} \frac{4}{3}])\,.
\end{equation}
Now define $g$ by
\[
\hat g = H \cdot \hat g^{twist}\,.
\]
Hence,
\[
\begin{aligned}
\hat g_{\mathbf x}(x_{j-1}, \omega) & = \left(\frac{e^{-L \sqrt{-1} x_{j-1} \omega} - 1}{\frac{L}{2\pi} \omega x_{j-1} - 1}\right) q(x_{j-1}, \frac{L}{2\pi}\omega) [(\frac{L}{2\pi}\omega)^\nu - x_{j-1}^{-\nu}] G_2(\mathbf x) \\
& = (e^{-L \sqrt{-1} x_{j-1} \omega} - 1) \hat F(x_{j-1}, (\frac{2\pi}{L})^{-1} \omega) G_2(\mathbf x) \\
& =  (e^{-L \sqrt{-1} x_{j-1} \omega} - 1) \hat f_{\mathbf x}(x_{j-1}, \omega)\,.
\end{aligned}
\]
So $f$ and $g$ satisfy \eqref{eq:Fourier-coeqn}.

Now let $s \in 2\N$  and $j > 1$.
For each $\beta \leq s$,
let $\mathcal B^{(\beta)}$ be the sum of all
products of $\beta$ vector fields in
$\{\hat X_j\}_{j = 1}^n \cup \{\hat{\mathfrak{u}}_{l, k}\}_{1 \leq l, k \leq n} \setminus\{\hat{\mathfrak{u}}_{1, i}, \hat{\mathfrak{u}}_{j, i} \}$.
From the proof of Lemma~\ref{lemm:f-bound} (see \eqref{eq:g-B-u} and \eqref{eq:f-est:1}),
we have
\begin{align}
\Vert g \Vert_{s} & \leq C_{s,n} \sum_{\substack{\beta \leq s \\ \beta_1 + \beta_2 \leq \beta}}
\Vert \mathcal B^{(s-\beta)} \hat{\mathfrak{u}}_{j, i}^{\beta_1} \hat{\mathfrak{u}}_{1, i}^{\beta_2} \hat g\Vert  \notag \\
& \leq C_{s, n} \sum_{\substack{\beta \leq s \\ \beta_1 + \beta_2 \leq \beta}}
\sum_{\substack{v_1 + v_2 = \beta_1 \\
l \leq \beta_2 }}
\sum_{\substack{1 \leq k \leq n - 1 \\ k \neq a \\ \sum_{k \neq a} m_k = \beta_2-l}}
\Vert \mathcal B^{(s-\beta)} \prod_{(m_k)}
\left( x_k^{m_k} \partial^{m_k} x_k G_2 \right) \notag \\
& \times \left(x_{j-1}^{m_{j-1}-v_1} \partial^{m_{j-1}+v_2} x_{j-1} \partial^{\beta_1}\omega\Pi_{\beta_2}^{(l)} (H \hat G_1(x_{j-1}, \frac{L}{2\pi}\omega))\right)\Vert  \label{eq:g-map:22}
\end{align}
Recall that $\Pi_{\beta_2}^{(l)}$ is the sum of all products of $l \leq \beta_2$ operators $\hat V_0$ and $\beta_2-l$ operators $\partial \omega$, and moreover, we have the Leibniz-type formula for $\hat V_0$ derivatives in \eqref{eq:Leibnitz}.
Specifically, there are universal
  coefficients $(b^{(\beta_2)}_{ijkm}) $ such that
\[
\begin{aligned}
 \hat V_0^{\beta_2}&(\hat G_1(x_{j-1}, \frac{L}{2\pi}\omega)\cdot H) \\
 & = \sum_{ \substack { i+j+ m\leq
  \beta \\
 k \leq m }}
 b^{(\beta)}_{ijkm} [(\frac{d}{d\omega})^m \hat V^i \hat G_1(x_{j-1}, \frac{L}{2\pi}\omega)]
 [(\omega \frac{d}{d\omega})^k \hat V^j H]\,.
 \end{aligned}
\]

Consider the case $L = 1$.  
By \eqref{eq:H-smooth}, and because $H$ is independent of $\nu$,
and by \eqref{eq:partial-x-omega-V},
there is a constant $C_\beta > 0$ such that
\[
\begin{aligned}
\vert \hat V_0^{\beta_2}(\hat G_1(x_{j-1}, \frac{1}{2\pi}\omega)\cdot H)\vert & \leq C_{\beta_2} L^{\beta_2} \vert \nu \vert^{m + i + j} \\
& \leq C_{\beta_2} L^{\beta_2} \vert \nu \vert^{\beta_2}\,.
\end{aligned}
\]
Hence, for general $L > 0$, we get 
\[
\begin{aligned}
\vert \hat V_0^{\beta_2}(\hat G_1(x_{j-1}, \frac{L}{2\pi}\omega)\cdot H)\vert \leq C_{\beta_2} L^{2\beta_2} \vert \nu \vert^{\beta_2}\,.
\end{aligned}
\]
So
\[
\begin{aligned}
\vert  \partial^{m_{j-1}+v_2} x_{j-1} \partial^{\beta_1} \omega \Pi_{\beta_2}^{(l)} (\hat G_1(x_{j-1}, \frac{L}{2\pi}\omega)\cdot H)\vert & \leq C_\beta L^{\beta_1 + 2\beta_2} \vert \nu \vert^{\max\{m_{j-1} + v_2, \beta_1 + \beta_2\}} \\
&  \leq C_\beta L^{2\beta} \vert \nu \vert^{\beta_1 + \beta_2} \\
&  \leq C_\beta L^{2\beta} \vert \nu \vert^{\beta}\,.
\end{aligned}
\]
Finally, because $G_2$ is compactly supported and is independent of $\nu$, we conclude
that there is a constant $C_{s, n}^{(1)} > 0$ such that
\[
\begin{aligned}
\eqref{eq:g-map:22} & \leq C_{s, n}^{(1)} (1+L)^{2s} \vert \nu \vert^{s-\beta+\beta}  \\
& = C_{s, n}^{(1)} (1+L)^{2s} \vert \nu \vert^s\,.
\end{aligned}
\]
By interpolation, the above estimate holds for any $s \geq 0$.
By a similar argument for $j = 1$,
we conclude that for any $j \geq 1$ and for any $s \geq 0$,
there is a constant $C_{s, n}^{(2)} > 0$ such that
\begin{equation}\label{eq:uij-lower_map-final}
\Vert g \Vert_s \leq C_{s, n}^{(2)} (1+L)^{2s} \vert \nu \vert^{s}\,.
\end{equation}

On the other hand, by \eqref{eq:u-lambda-f}, \eqref{eq:g_map-twist} and Theorem~\ref{thm:lower-bound-uf},
for any $s \geq 0$,
there are constants $\alpha_s^{(0)}, c_{s}^{(0)} > 0$
such that for any $\vert \nu \vert \geq \alpha_s^{(0)}$,
\begin{align}
\Vert (I - \hat{\mathfrak{u}}_{j, i}^2)^{s/2} \hat f \hat \Vert & = (\frac{2\pi}{L})^{- 1/2} \Vert (I - (\frac{2\pi}{L})^2 \hat{\mathfrak{u}}_{j, i}^2)^{s/2} \hat f_{(2\pi/L)} \Vert \notag \\
& > c_s^{(0)} (\frac{L}{2\pi} + \frac{2\pi}{L})^{-(s+1/2)} \Vert (I - \hat{\mathfrak{u}}_{j, i}^2)^{s/2} \hat f_{(2\pi/L)} \Vert \notag \\
& > c_s^{(0)} (\frac{L^2 + 4\pi^2}{2\pi L})^{-(s + 1/2)} \vert \nu \vert^{2s+1/2}\,. \label{eq:uij-upper_map-final}
\end{align}

So let $\sigma \in [0, s + 1/2)$  and $C > 0$.
Then for any $\vert \nu\vert$ large enough that
\[
\left\{
\begin{aligned}
& \vert \nu \vert \geq \alpha_{s}^{(0)}\,,  \\
&  c_s^{(0)} (\frac{L^2 + 4\pi^2}{2\pi L})^{-(s + 1/2)} \vert \nu \vert^{2s+1/2}  > C C_{s+\sigma, n}^{(2)} (1+L)^{2(s+\sigma)} \vert \nu \vert^{s+\sigma}\,,
\end{aligned}
\right.
\]
we get by \eqref{eq:uij-upper_map-final} and \eqref{eq:uij-lower_map-final} that
\[
\begin{aligned}
\Vert (I - \mathfrak{u}_{j, i}^2)^{s/2} f \Vert & \geq c_s^{(0)} (\frac{L^2 + 4\pi^2}{2\pi L})^{-(s + 1/2)} \vert \nu \vert^{2s+1/2} \\
& \geq C C_{s+\sigma, n}^{(2)} (1+L)^{2(s+\sigma)} \vert \nu \vert^{s+\sigma} \\
& > C \Vert g \Vert_{s+\sigma}\,.
\end{aligned}
\]
This concludes the proof of Theorem~\ref{main_thm2-lower}.
\end{proof}

\end{document}